\documentclass{amsart}
\usepackage{amssymb,amscd}
\usepackage{graphicx,subfigure,psfrag,color}
\usepackage[all]{xy}

\newtheorem{thm}{Theorem}[section]
\newtheorem{pro}[thm]{Proposition}
\newtheorem{lem}[thm]{Lemma}
\newtheorem{cor}[thm]{Corollary}
\newtheorem{exa}[thm]{Example}
\newtheorem{con}[thm]{Conjecture}

\newcommand{\im}{\operatorname{im}}

\begin{document}

\title{Characteristics of graph braid groups}
\author{Ki Hyoung Ko}
\author{Hyo Won Park}
\address{Department of Mathematics, Korea Advanced Institute of Science and Technology, Daejeon, 307-701, Korea}
\email{\{knot, h.w.park\}@kaist.ac.kr}
\begin{abstract}
We give formulae for the first homology of the $n$-braid group and the pure 2-braid group over a finite graph in terms of graph theoretic invariants. As immediate consequences, a graph is planar if and only if the first homology of the $n$-braid group over the graph is torsion-free and the conjectures about the first homology of the pure 2-braid groups over graphs in \cite{FH} can be verified. We discover more characteristics of graph braid groups: the $n$-braid group over a planar graph and the pure 2-braid group over any graph have a presentation whose relators are words of commutators, and the 2-braid group and the pure 2-braid group over a planar graph have a presentation whose relators are commutators. The latter was a conjecture in \cite{FS2} and so we propose a similar conjecture for higher braid indices.
\end{abstract}

\maketitle

\section{Introduction}

Given a topological space $\Gamma$, let $C_n\Gamma$ and $UC_n\Gamma$, respectively, denote the ordered and unordered configuration spaces of $n$-points in $\Gamma$. That is,
$$C_n\Gamma=\{(x_1,\ldots,x_n)\in
        \Gamma^n \mid x_i\ne x_j\mbox{ if } i\neq j \}$$
and
$$UC_n\Gamma=\{\{x_1,\ldots,x_n\}\subset
        \Gamma \mid x_i\ne x_j\mbox{ if } i\neq j \}.$$
By considering the symmetric group $S_n$ permuting $n$ coordinates in $\Gamma^n$, $UC_n\Gamma$ is identified with the quotient space
$C_n\Gamma/S_n$.

In this article, we assume $\Gamma$ is a finite connected graph regarded as an Euclidean subspace and we study topological characteristics, in particular their homologies and fundamental groups, of $C_n\Gamma$ and $UC_n\Gamma$ via graph theoretical characteristics of $\Gamma$.

Instead of the configuration spaces $C_n\Gamma$ and $UC_n\Gamma$ that have open boundaries, it is convenient to use their cubical complex alternatives--the ordered discrete configuration space $D_n$ and the unordered discrete configuration space $UD_n$. After regarding $\Gamma$ as an 1-dimensional CW complex, we define
$$D_n\Gamma=\{ (c_1,\cdots,c_n)\in\Gamma^n\mid\partial c_i\cap\partial c_j=\empty\mbox{ if }i\ne j\}$$
and
$$UD_n\Gamma=\{ \{c_1,\cdots,c_n\}\subset\Gamma\mid\partial c_i\cap\partial c_j=\empty\mbox{ if }i\ne j\}$$
where $c_i$ is either a 0-cell (or vertex) or an open 1-cell (or edge) in $\Gamma$ and $\partial c_i$ denotes either $c_i$ itself if $c_i$ is a 0-cell or its ends if $c_i$ is an open 1-cell.

If $\Gamma$ is \emph{suitably subdivided} in the sense that each path between two vertices of valency $\ne 2$ contains at least $n-1$ edges and each simple loop at a vertex contains at least $n+1$ edges, then according to \cite{Ab}, \cite{KKP} and \cite{PS}, the discrete configuration space $D_n\Gamma$($UD_n\Gamma$, respectively) is deformation retract of the usual configuration space $C_n\Gamma$($UC_n\Gamma$, respectively). Under the assumption of suitable subdivision, the \emph{pure graph braid group} $P_n\Gamma$ and the \emph{graph braid group} $B_n\Gamma$ of $\Gamma$ are the fundamental groups of the ordered and the unordered configuration spaces of $\Gamma$, that is,
$$P_n\Gamma=\pi_1(C_n\Gamma)\cong\pi_1(D_n\Gamma)\quad\mbox{and} \quad B_n\Gamma=\pi_1(UC_n\Gamma)\cong\pi_1(UD_n\Gamma).$$
Abrams showed in \cite{Ab} that discrete configuration spaces
$D_n\Gamma$ and $UD_n\Gamma$ are cubical complexes of non-positive
curvature and so locally CAT(0) spaces. In particular, $D_n\Gamma$
and $UD_n\Gamma$ are Eilenberg-MacLane spaces, and $P_n\Gamma$ and
$B_n\Gamma$ are torsion-free. Furthermore,
$$H_i(P_n\Gamma)\cong H_i(C_n\Gamma)\cong H_i(D_n\Gamma)\quad\mbox{and} \quad H_i(B_n\Gamma)\cong H_i(UC_n\Gamma)\cong H_i(UD_n\Gamma).$$

Conceiving applications to robotics, Abrams and Ghrist \cite{AR} began to study configuration spaces over graphs and graph braid groups around 2000 in the topological point of view. Research on graph braid groups has mainly been concentrated on characteristics of their presentations. An outstanding question was which graph braid group is a right-angled Artin group. The precise characterization of such graphs was given in \cite{KKP} for $n\ge 5$ by extending the result in \cite{FS1} for trees and $n\ge 4$. So it is natural to consider two other classes of groups defined by relaxing the requirement of right-angled Artin groups that have a presentation whose relators are commutators of generators. A {\em simple-commutator-related group} has a presentation whose relators are commutators, and a {\em commutator-related group} has a presentation whose relators are words of commutators. Farley and Sabalka proved in \cite{FS2} that $B_2\Gamma$ is simple-commutator-related if every pair of cycles in $\Gamma$ are disjoint and they conjectured that $B_2\Gamma$ is simple-commutator-related whose relators are related to two disjoints cycles if $\Gamma$ is planar.

On the other hand, Farley showed in \cite{Far} that the homology groups of the unordered configuration space $UC_nT$ for a tree $T$ are torsion free and computed their ranks. Kim, Ko and Park proved that if $\Gamma$ is non-planar, $H_1(UC_n\Gamma)$ has a 2-torsion and the converse holds for $n=2$ and they conjectured that $H_1(B_n\Gamma)$ is torsion free iff $\Gamma$ is planar \cite{KKP}. Barnett and Farber show in \cite{BF} that for a planar graph $\Gamma$ satisfying a certain condition (which implies that $\Gamma$ is either the $\Theta$-shape graph or a simple and triconnected graph),  $\beta_1(C_2\Gamma)=2\beta_1(\Gamma)+1$. Furthermore, Farber and Hanbury showed in \cite{FH} that for a non-planar graph $\Gamma$ satisfying a certain condition (which also implies that $\Gamma$ is a simple and triconnected graph), $\beta_1(C_2\Gamma)=2\beta_1(\Gamma)$. They also conjectured that $H_1(C_2\Gamma)$ is always torsion free and that $\beta_1(C_2\Gamma)=2\beta_1(\Gamma)$ iff $\Gamma$ is non-planar, simple and triconnected (this is equivalent to their hypothesis).

In this article, we express $H_1(UC_n\Gamma)$ and $H_1(C_2\Gamma)$ for an finite connected graph $\Gamma$ in terms of graph theoretic invariants (see Theorem~\ref{thm:H1Bn} and Theorem~\ref{thm:H1PB2}). All the results and the conjectures, mentioned above, on the first homologies of configuration spaces over graphs are immediate consequences of these expressions. In addition, we prove that $B_n\Gamma$ is commutator-related for a planar graph $\Gamma$ and $P_2\Gamma$ is always commutator-related (see Theorem~\ref{thm:crg}). By combining with a result of \cite{BF}, we finally prove that for a planar graph $\Gamma$, $B_2\Gamma$ and $P_2\Gamma$ are simple-commutator-related whose relators are commutators of words corresponding to two disjoint cycles on $\Gamma$ (see Theorem~\ref{thm:presentationforn2}).

The major tool for computing $H_1(UC_n\Gamma)$ is to use a Morse complex of $UD_n\Gamma$ obtained via discrete Morse theory. In \S\ref{s:two}, we first give an example that illustrates how to use the Morse complex to compute $H_1(UC_n\Gamma)$. Then we choose a nice maximal tree of $\Gamma$ and its planar embedding, the second boundary map of the Morse complex induced from these choices becomes so manageable that a description of the second boundary map can be given.

In \S\ref{s:three}, the matrix for the second boundary map is systematically simplified (see Theorem~\ref{thm:matrix}) via row operations after giving certain orders on generating 1-cells and 2-cells (called critical cells) of the Morse complex. Then we decompose $\Gamma$ into biconnected graphs and further decompose each biconnected graph into triconnected graphs and compute the contribution from critical 1-cells that disappear under these decompositions. Then we show all critical 1-cells except those coming from deleted edges are homologous up to signs for a given triconnected graph and generate a summand $\mathbb Z$ or $\mathbb Z_2$ depending on whether the graph is planar or not. Finally we collect results from all decompositions to have a formula for $H_1(UC_n\Gamma)$. For $n=2$, the second boundary map of the Morse complex of $D_n\Gamma$ is not any harder than the Morse complex of $UD_n\Gamma$. Thus the formula for $H_1(C_2\Gamma)$ is obtained by a similar argument.

In \S\ref{s:four}, we develop noncommutative versions of some of technique in the previous section to obtain optimized presentations of (pure) graph braid groups so that they have certain desired properties via Tietze transformation. In fact, the orders on critical 1-cells and 2-cells play crucial roles in systematic eliminations of canceling pairs of a 2-cell and an 1-cell.  And we show that (pure) graph braid groups have presentations with special characteristics mentioned above. We finish the paper with the conjecture about a graph $\Gamma$ such that $B_n\Gamma$ and $P_3\Gamma$ are simple-commutator-related groups.

\section{Discrete configuration spaces and discrete Morse theory}\label{s:two}

Given a finite graph $\Gamma$, the unordered discrete configuration space $UD_n\Gamma$ is collapsed to a complex called a \emph{Morse complex} by using discrete Morse theory developed by Forman \cite{For}. In \S\ref{ss21:Bn}, we briefly review this technology following \cite{FS,KKP} and use it to compute $H_1(UD_2K_{3,3})$ as a warm-up that demonstrates what is ahead of us. In \S\ref{ss22:PBn}, we extend the technique to the discrete configuration space $D_n\Gamma$ and compute $H_1(D_2K_{3,3})$ as an example. In \S\ref{ss23:images}, we show how to choose a nice maximal tree and its embedding so that the second boundary map of the induced Morse complex can be described in the fewest possible cases. Then we list up all of these cases in a few lemmas.

\subsection{Discrete Morse theory on $UD_n\Gamma$}\label{ss21:Bn}
Let $\Gamma$ be a suitably subdivided graph.
In order to collapse the unordered discrete configuration space $UD_n\Gamma$ via discrete Morse theory, we first choose a maximal tree $T$ of $\Gamma$. Edges in $\Gamma-T$ are called \emph{deleted edges}. Pick a vertex of valency 1 in $T$ as a basepoint and assign 0 to this vertex. We assume that the path between the base vertex 0 and the nearest vertex of valency $\ge 3$ in $T$ contains at least $n-1$ edges for the purpose that will be revealed later. Next we give an order on vertices as follows : Fix an embedding of $T$ on the plane. Let $R$ be a regular neighborhood of $T$. Starting from the base vertex 0, we number unnumbered vertices of $T$ as we travel along $\partial R$ clockwise. Figure~\ref{fig:K33n2} illustrates this procedure for the complete bipartite graph $K_{3,3}$ and for $n=2$. There are four deleted edges to form a maximal tree. All vertices in $\Gamma$ are numbered and so are referred by the nonnegative integers.
\begin{figure}[ht]
\psfrag{0}{\small0}
\psfrag{1}{\small1}
\psfrag{2}{\small2}
\psfrag{3}{\small3}
\psfrag{4}{\small4}
\psfrag{5}{\small5}
\centering
\includegraphics[height=2cm]{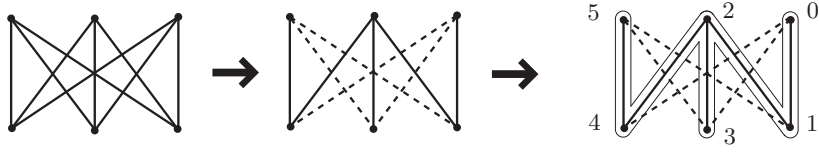}
\caption{Choose a maximal tree and give an order}
\label{fig:K33n2}
\end{figure}

Each edge $e$ in $\Gamma$ is oriented so that the initial vertex $\iota(e)$ is larger than the terminal vertex $\tau(e)$. The edge $e$ is denoted by $\tau(e)\mbox{-}\iota(e)$. A (open, cubical) cell $c$ in the unordered discrete configuration space $UD_n\Gamma$ can be written as an unordered $n$-tuple $\{c_1,\cdots,c_n\}$ where each $c_j$ is either a vertex or an edge in $\Gamma$. The cell $c$ is an $i$-cell if the number of edges among $c_j$'s is $i$. For example, $\{0\mbox{-}1,3\mbox{-}5\}$ represents a 2-cell in $U\kern-2pt D_2K_{3,3}$ under the order on vertices of $K_{3,3}$ given by Figure~\ref{fig:K33n2}. In fact, $UD_2K_{3,3}$ has fifteen 0-cells, thirty-six 1-cells and eighteen 2-cells as given on the left in Figure~\ref{fig:UD2}.

A vertex $v$ in an $i$-cell $c$ is said to be \emph{blocked} if for the edge $e$ in $T$ such that $\iota(e)=v$, $\tau(e)$ is in $c$ or is an end vertex of another edge in $c$.
Let $K_i$ denote the set of all $i$-cells of $UD_n\Gamma$ and $K_{-1}=\emptyset$. Define $W_i:K_i\to K_{i+1}\cup\{\mbox{\rm void}\}$ for $i\ge -1$ by induction on $i$. Let $c=\{c_1,c_2,\cdots,c_n\}$ be an $i$-cell. If $c\notin\im(W_{i-1})$ and there are unblocked vertices in $c$ and, say, $c_1$ is the smallest unblocked vertex then $W_i(c)=\{v\mbox{-}c_1,c_2,\cdots,c_n\}$ where the edge $v\mbox{-}c_1$ is in $T$. Otherwise, $W_i(c)=\mbox{\rm void}$. Let $K_* =\bigcup K_i$. Define $W:K_* \to K_* \cup\{\mbox{\rm void}\}$ by $W(c)= W_i(c)$ for a $i$-cell $c$. Then it is not hard to see that $W$ is well-defined, and each cell in $W(K_*)-\{\mbox{\rm void}\}$ has the unique preimage under $W$, and there is no cell in $K_*$ that is both an image and a preimage of other cells under $W$. For example, each arrow on the right of Figure~\ref{fig:UD2} points from $c$ to $W(c)$ in $UD_2K_{3,3}$ and the dashed lines represent 1-cells sent to {\rm void} under $W$.

\begin{figure}[ht]
\psfrag{{0,1}}{\footnotesize$\{0,1\}$}
\psfrag{{0,2}}{\footnotesize$\{0,2\}$}
\psfrag{{0,3}}{\footnotesize$\{0,3\}$}
\psfrag{{0,4}}{\footnotesize$\{0,4\}$}
\psfrag{{0,5}}{\footnotesize$\{0,5\}$}
\psfrag{{1,2}}{\footnotesize$\{1,2\}$}
\psfrag{{1,3}}{\footnotesize$\{1,3\}$}
\psfrag{{1,4}}{\footnotesize$\{1,4\}$}
\psfrag{{1,5}}{\footnotesize$\{1,5\}$}
\psfrag{{2,3}}{\footnotesize$\{2,3\}$}
\psfrag{{2,4}}{\footnotesize$\{2,4\}$}
\psfrag{{2,5}}{\footnotesize$\{2,5\}$}
\psfrag{{3,4}}{\footnotesize$\{3,4\}$}
\psfrag{{3,5}}{\footnotesize$\{3,5\}$}
\psfrag{{4,5}}{\footnotesize$\{4,5\}$}
\centering
\includegraphics[height=11cm]{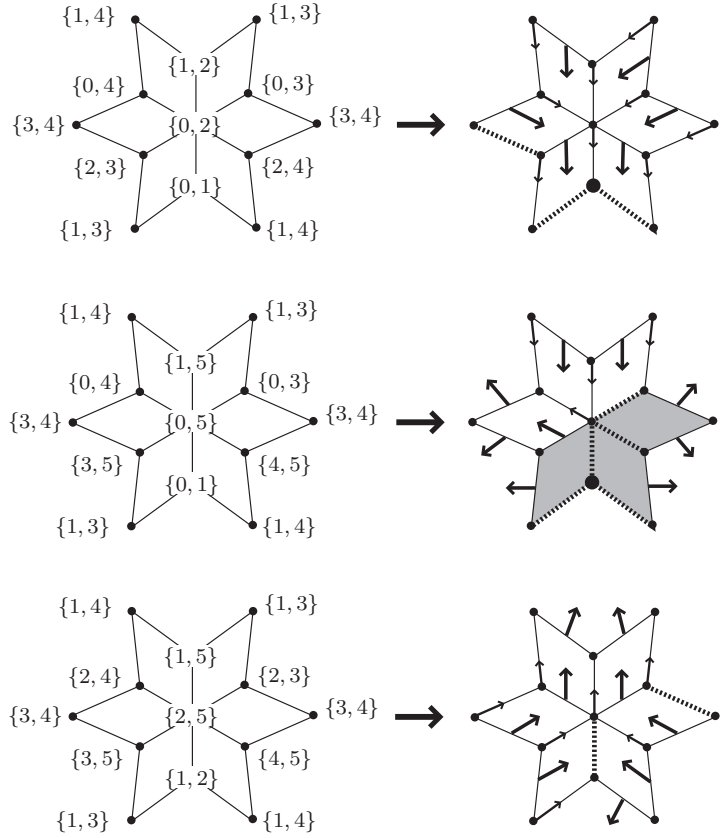}
\caption{$UD_2K_{3,3}$ and the map $W$}
\label{fig:UD2}
\end{figure}

For each pair $(c,W(c))\in K_*\times(W(K_*)-\{\mbox{\rm void}\})$, we homotopically collapse the closure $\overline{W(c)}$ onto $\overline{W(c)}-(W(c)\cup c)$ to obtain a Morse complex $UM_n\Gamma$ of $UD_n\Gamma$. Then cells $c$ and $W(c)$ are said to be \emph{redundant} and \emph{collapsible}, respectively. Redundant or collapsible cells disappear in a Morse complex. Cells in $W^{-1}(\mbox{void})-W(K_*)$ survive in a Morse complex and are said to be \emph{critical}. For example, the 0-cell $\{1,4\}$ is redundant and the 1-cell $\{0\mbox{-}1,4\}$ is collapsible in Figure~\ref{fig:UD2}. In fact, there are one critical 0-cell $\{0,1\}$, seven critical 1-cells and three critical 2-cells in the Morse complex $M_2K_{3,3}$ as shown in Figure~\ref{fig:MUD2}.

\begin{figure}[ht]
\psfrag{{0-3,1}}{\footnotesize$\{0\mbox{-}3,1\}$}
\psfrag{{0-4,1}}{\footnotesize$\{0\mbox{-}4,1\}$}
\psfrag{{0-4,5}}{\footnotesize$\{0\mbox{-}4,5\}$}
\psfrag{{1-5,0}}{\footnotesize$\{1\mbox{-}5,0\}$}
\psfrag{{1-5,2}}{\footnotesize$\{1\mbox{-}5,2\}$}
\psfrag{{2-4,3}}{\footnotesize$\{2\mbox{-}4,3\}$}
\psfrag{{3-5,0}}{\footnotesize$\{3\mbox{-}5,0\}$}
\psfrag{{0-3,1-5}}{\footnotesize$\{0\mbox{-}3,1\mbox{-}5\}$}
\psfrag{{0-4,1-5}}{\footnotesize$\{0\mbox{-}4,1\mbox{-}5\}$}
\psfrag{{0-4,3-5}}{\footnotesize$\{0\mbox{-}4,3\mbox{-}5\}$}
\centering
\includegraphics[height=4cm]{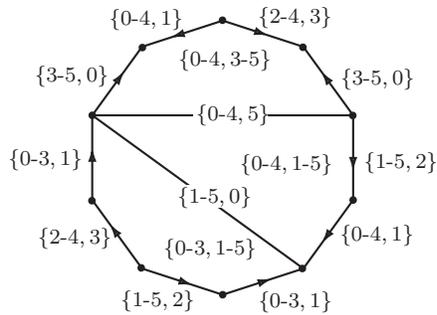}
\caption{Morse complex $UM_2K_{3,3}$ of $UD_2K_{3,3}$}
\label{fig:MUD2}
\end{figure}

Farley and Sabalka in~\cite{FS} gave an alternative description for these three kinds of cells in $UD_n\Gamma$ as follows : An edge $e$ in a cell $c=\{c_1,\cdots,c_{n-1},e\}$ is \emph{order-respecting} if $e$ is not a deleted edge and there is no vertex $v$ in $c$ such that $v$ is adjacent to $\tau(e)$ in $T$ and $\tau(e)<v<\iota(e)$. A cell is critical if it contains neither order-respecting edges nor unblocked vertices. A cell is collapsible if it contains at least one order-respecting edge and each unblocked vertex is larger than the initial vertex of some order-respecting edge. A cell is redundant if it contains at least one unblocked vertex that is smaller than the initial vertices of all order-respecting edges. Notice that there is exactly one critical 0-cell $\{0,1,\cdots,n-1\}$ by the assumption that there are at least $n-1$ edges between $0$ and the nearest vertex with valency $\ge3$ in the maximal tree.

A choice of a maximal tree of $\Gamma$ and its planar embedding determine an order on vertices and in turn a Morse complex $UM_n\Gamma$ that is homotopy equivalent to $UD_n\Gamma$. We wish to compute its homology groups via the cellular structure of $UM_n\Gamma$.

Let $(C_i(UD_n\Gamma),\partial)$ be the (cubical) cellular chain complex of $UD_n\Gamma$.
For an $i$-cell $c=\{e_1,e_2,\ldots,e_i,v_{i+1},\ldots,v_n\}$ of $UD_n\Gamma$ such that $e_1,\ldots,e_i$ are edges with $\tau(e_1)<\tau(e_2)<\cdots<\tau(e_i)$ and $v_{i+1},\ldots,v_n$ are vertices of $\Gamma$, let
$$\partial_k^\iota(c)=\{e_1, \ldots, e_{k-1}, e_{k+1}, \ldots, e_i,v_{i+1},\ldots,v_n,\iota(e_k)\},$$
$$\partial_k^\tau(c)=\{e_1, \ldots, e_{k-1}, e_{k+1}, \ldots, e_i,v_{i+1},\ldots,v_n, \tau(e_k)\}.$$ Then we define the boundary map as
$$\partial(c)=\sum_{k=1}^i(-1)^k(\partial_k^\iota(c)-\partial_k^\tau(c)).$$
Notice that this definition of $\partial$ on $UD_n\Gamma$ is different from that in \cite{FS} and \cite{KKP} in sign convention. This convention seems more convenient in the current work.
Let $M_i(UD_n\Gamma)$ be the free abelian group generated by critical $i$-cells. We now try to turn the graded abelian group $\{M_i(UD_n\Gamma)\}$ into a chain complex.

Let $R:C_i(UD_n\Gamma)\to C_i(UD_n\Gamma)$ be a homomorphism defined by $R(c)=0$ if $c$ is a collapsible $i$-cell, by $R(c)=c$ if $c$ is critical, and by $R(c)=\pm \partial W(c)+c$ if $c$ is redundant where the sign is chosen so that the coefficient of $c$ in $\partial W(c)$ is $-1$. By \cite{For}, there is a nonnegative integer $m$ such that $R^m=R^{m+1}$ and let $\widetilde R=R^m$. Then $\widetilde R(c)$ is in $M_i(UD_n\Gamma)$ and we have a homomorphism $\widetilde R : C_i(UD_n\Gamma)\to M_i(UD_n\Gamma)$. Define a map $\widetilde\partial:M_i(UD_n\Gamma)\to M_{i-1}(UD_n\Gamma)$ by $\widetilde\partial(c) = \widetilde R\partial(c)$. Then $(M_i(UD_n\Gamma),\widetilde\partial)$ forms a chain complex. However, the inclusion $M_*(UD_n\Gamma)\hookrightarrow C_*(UD_n\Gamma)$ is not a chain map. Instead, consider a homomorphism $\varepsilon:M_i(UD_n\Gamma)\to C_i(UD_n\Gamma)$ defined as follows: For a (critical) $i$-cell $c$, $\varepsilon(c)$ is obtained from $c$ by minimally adding collapsible $i$-cells until it becomes closed in the sense that for each redundant $(i-1)$-cell $c'$ in the boundary of every $i$-cell summand in $\varepsilon(c)$, $W(c')$ already appears in $\varepsilon(c)$. Then $\varepsilon$ is a chain map that is a chain homotopy inverse of $\widetilde R$. Thus $(M_i(UD_n\Gamma),\widetilde\partial)$ and $(C_i(UD_n\Gamma),\partial)$ have the same chain homotopy type.

\begin{exa}\label{ex:K33n2} Since $UM_2K_{3,3}$ is a nonorientable surface of nonorientable genus 5 as seen in Figure~\ref{fig:K33n2}, we easily see that $H_1(B_2K_{3,3})\cong\mathbb Z^4\oplus Z_2$. However we want to compute it directly from the chain complex $(M_i(UD_nK_{3,3}),\widetilde\partial)$ to demonstrate discrete Morse theory. In fact, $H_1(B_nK_{3,3})\cong H_1(B_2K_{3,3})$ for any braid index $n$ (see Lemma~\ref{lem:biconnected}) and the existence of a 2-torsion will be needed later.
\end{exa}
The Morse complex $UM_2K_{3,3}$ has seven critical 1-cells
$\{0\mbox{-}3,1\}$, $\{0\mbox{-}4,1\}$, $\{0\mbox{-}4,5\}$, $\{1\mbox{-}5,0\}$, $\{1\mbox{-}5,2\}$, $\{2\mbox{-}4,3\},\{3\mbox{-}5,0\}$
and three critical 2-cells
$\{0\mbox{-}3,1\mbox{-}5\}$, $\{0\mbox{-}4,1\mbox{-}5\}$, $\{0\mbox{-}4,3\mbox{-}5\}$.
We compute the boundary images of critical 2-cells. First,
$$\widetilde\partial(\{0\mbox{-}3,1\mbox{-}5\})=\widetilde R\circ\partial(\{0\mbox{-}3,1\mbox{-}5\})=\widetilde R(-\{1\mbox{-}5,3\}+\{1\mbox{-}5,0\}+\{0\mbox{-}3,5\}-\{0\mbox{-}3,1\})$$
Since $\{1\mbox{-}5,0\}$ and $\{0\mbox{-}3,1\}$ are critical 1-cells, we only consider other two 1-cells.
\begin{align*}
\widetilde R(\{1\mbox{-}5,3\})=&\widetilde R(-\partial(\{1\mbox{-}5,2\mbox{-}3\})+\{1\mbox{-}5,3\})=\widetilde R(\{2\mbox{-}3,5\}-\{2\mbox{-}3,1\}+\{1\mbox{-}5,2\})\\
=&\widetilde R(-\{2\mbox{-}3,1\})+\{1\mbox{-}5,2\}=\widetilde R(-\partial\{2\mbox{-}3,0\mbox{-}1\}-\{2\mbox{-}3,1\})+\{1\mbox{-}5,2\}\\
=&\widetilde R(-\{2\mbox{-}3,0\}-\{0\mbox{-}1,3\}+\{0\mbox{-}1,2\})+\{1\mbox{-}5,2\}=\{1\mbox{-}5,2\}
\end{align*}
In the above computation, $\{2\mbox{-}3,5\}$, $\{2\mbox{-}3,0\}$, $\{0\mbox{-}1,3\}$, and $\{0\mbox{-}1,2\}$ are collapsible.

The following computation make us feel the need of utilities such as Lemma~\ref{lem:naturalmove}.
\begin{align*}
\widetilde R(\{0\mbox{-}3,5\})=&\widetilde R(\partial(\{0\mbox{-}3,4\mbox{-}5\})+\{0\mbox{-}3,5\})=\widetilde R(\{4\mbox{-}5,3\}-\{4\mbox{-}5,0\}+\{0\mbox{-}3,4\})\\
=&\widetilde R(\{4\mbox{-}5,2\})+\widetilde R(-\partial\{0\mbox{-}3,2\mbox{-}4\}+\{0\mbox{-}3,4\})\\
=&\widetilde R(\{4\mbox{-}5,1\})+\widetilde R(\{2\mbox{-}4,3\}-\{2\mbox{-}4,0\}+\{0\mbox{-}3,2\})\\
=&\widetilde R(\{4\mbox{-}5,0\})+\{2\mbox{-}4,3\}+\widetilde R(\{0\mbox{-}3,2\})\\
=&\{2\mbox{-}4,3\}+\widetilde R(\{0\mbox{-}3,1\})=\{2\mbox{-}4,3\}+\{0\mbox{-}3,1\}
\end{align*}
So $\widetilde\partial(\{0\mbox{-}3,1\mbox{-}5\}) =-\{1\mbox{-}5,2\}+\{1\mbox{-}5,0\}+\{2\mbox{-}4,3\}$. This result can be expressed by a row vector of coefficients. The boundary images of the other two critical 2-cells give two more rows. Thus the second boundary map can be expressed by the following $(3\times 7)$-matrix and it can be put into an echelon form via row operations.
$${\left(
\begin{array}{ccccccc}
0&0&0&1&-1&1&0\\
0&-1&1&1&-1&0&0\\
0&-1&1&0&0&1&0
\end{array} \right)
\rightarrow
\left(
\begin{array}{ccccccc}
0&\framebox[0.4cm][l]{-1}&1&1&-1&0&0\\
0&0&0&\framebox[0.4cm][l]{1}&-1&1&0\\
0&0&0&0&0&\framebox[0.4cm][l]{2}&0
\end{array} \right)}$$

Since there is only one critical 0-cell, the first boundary map is zero. So the cokernel of the second boundary map is isomorphic to $H_1(B_2K_{3,3})$. The free part of $H_1(B_2K_{3,3})$ is generated by critical 1-cells corresponding to a column do not contain a pivot (the first non-zero entry in a row). The torsion part of $H_1(B_2K_{3,3})$ generated by critical 1-cells corresponding to a column contains a pivot that is not $\pm 1$. Thus $H_1(B_2K_{3,3})\cong\mathbb Z^4\oplus Z_2$.
\qed

\subsection{Discrete Morse theory on $D_n\Gamma$}\label{ss22:PBn}

The discrete Morse theory on $D_n\Gamma$ is similar to that on $UD_n\Gamma$ except the fact that it uses ordered $n$-tuples instead unordered $n$-tuples.

Let $\widetilde K_i$ denote the set of all $i$-cells of $D_n\Gamma$ and $\widetilde K_{-1}=\emptyset$. Define $\widetilde W_i:\widetilde K_i\to \widetilde K_{i+1}\cup\{\mbox{\rm void}\}$ for $i\ge -1$ by induction on $i$. Let $o=(c_1,c_2,\cdots,c_n)$ be an $i$-cell. If $o\notin\im(\widetilde W_{i-1})$ and there are unblocked vertices in $o$ as an entry and, say, $c_j$ is the smallest unblocked vertex then $\widetilde W_i(o)=(c_1,c_2,\cdots,v\mbox{-}c_j,\cdots,c_n)$ where the edge $v\mbox{-}c_j$ is in $T$. Otherwise, $\widetilde W_i(o)=\mbox{\rm void}$.
Let $\widetilde K_* =\bigcup \widetilde K_i$.
Define $\widetilde W:\widetilde K_* \to \widetilde K_*\cup\{\mbox{\rm void}\}$ by $\widetilde W(o)=\widetilde W_i(o)$ for an $i$-cell $o$. Then $\widetilde W$ is well-defined and each cell in $\widetilde W(\widetilde K_*)-\{\mbox{\rm void}\}$ has the unique preimage under $\widetilde W$, and there is no cell in $\widetilde K_* $ that is both an image and a preimage of other cells under $\widetilde W$.

Let $\rho:D_n\Gamma \to UD_n\Gamma$ be the quotient map defined by $\rho(c_1,\cdots,c_n)=\{c_1,\cdots,c_n\}$. From the definition of $\widetilde W$ it is easy to see that an $i$-cell $o$ in $D_n\Gamma$ is critical (or collapsible or redundant, respectively) if and only if so is an $i$-cell $\rho(o)$ in $UD_n\Gamma$. Note that there are $n!$ critical 0-cells. Critical cells produce a Morse complex $M_n\Gamma$ of $D_n\Gamma$. Figure~\ref{fig:MD2} is a Morse complex $M_2K_{3,3}$ of $D_2K_{3,3}$. The circular (respectively, square) dots give the critical 0-cell $(0,1)$ (respectively, $(1,0)$).

\begin{figure}[ht]
\psfrag{(0-3,1)}{\footnotesize$(0\mbox{-}3,1)$}
\psfrag{(0-4,1)}{\footnotesize$(0\mbox{-}4,1)$}
\psfrag{(0-4,5)}{\footnotesize$(0\mbox{-}4,5)$}
\psfrag{(1-5,0)}{\footnotesize$(1\mbox{-}5,0)$}
\psfrag{(1-5,2)}{\footnotesize$(1\mbox{-}5,2)$}
\psfrag{(2-4,3)}{\footnotesize$(2\mbox{-}4,3)$}
\psfrag{(3-5,0)}{\footnotesize$(3\mbox{-}5,0)$}
\psfrag{(0-3,1-5)}{\footnotesize$(0\mbox{-}3,1\mbox{-}5)$}
\psfrag{(0-4,1-5)}{\footnotesize$(0\mbox{-}4,1\mbox{-}5)$}
\psfrag{(0-4,3-5)}{\footnotesize$(0\mbox{-}4,3\mbox{-}5)$}
\psfrag{(1,0-3)}{\footnotesize$(1,0\mbox{-}3)$}
\psfrag{(1,0-4)}{\footnotesize$(1,0\mbox{-}4)$}
\psfrag{(5,0-4)}{\footnotesize$(5,0\mbox{-}4)$}
\psfrag{(0,1-5)}{\footnotesize$(0,1\mbox{-}5)$}
\psfrag{(2,1-5)}{\footnotesize$(2,1\mbox{-}5)$}
\psfrag{(3,2-4)}{\footnotesize$(3,2\mbox{-}4)$}
\psfrag{(0,3-5)}{\footnotesize$(0,3\mbox{-}5)$}
\psfrag{(1-5,0-3)}{\footnotesize$(1\mbox{-}5,0\mbox{-}3)$}
\psfrag{(1-5,0-4)}{\footnotesize$(1\mbox{-}5,0\mbox{-}4)$}
\psfrag{(3-5,0-4)}{\footnotesize$(3\mbox{-}5,0\mbox{-}4)$}
\centering
\includegraphics[height=4cm]{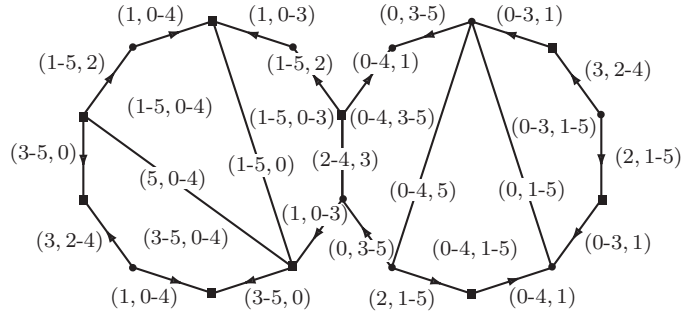}
\caption{Morse complex $M_2K_{3,3}$ of $D_2K_{3,3}$}
\label{fig:MD2}
\end{figure}

Give a $i$-cell $o\in D_n\Gamma$, let $\partial_k^\iota(o)$ ($\partial_k^\tau(o)$, respectively) denote the $(i-1)$-cell obtained from $o$ by replacing the $k$-th edge by its initial (terminal, respectively) vertex. Define
$$\partial(o)=\sum_{k=1}^i(-1)^k\partial_k^\iota(o)-\partial_k^\tau(o).$$
Then $(C_i(D_n\Gamma), \partial)$ forms a (cubical) cellular chain complex. Let $M_i(D_n\Gamma)$ be the free abelian group generated by critical $i$-cells. The reduction homomorphism $\widetilde R : C_i(D_n\Gamma)\to M_i(D_n\Gamma)$ is also well-defined. For $\widetilde\partial=\widetilde R\circ\partial$, $(M_i(D_n\Gamma), \widetilde\partial)$ forms a Morse chain complex that is chain homotopy equivalent to $(C_i(D_n\Gamma), \partial)$.

In order to carry over some of computational results on $UD_n\Gamma$ to $D_n\Gamma$, we introduce a bookkeeping notation.
Give an order among vertices and edges of $\Gamma$ by comparing the number assigned to vertices or terminal vertices of edges. Define a projection $\phi:D_n\Gamma\to S_n$ by sending $o=(c_1,\ldots,c_n)$ to the permutation $\sigma$ such that  $c_{\sigma(1)}<\cdots<c_{\sigma(n)}$. And define a bijection $\Phi:D_n\Gamma\to UD_n\Gamma\times S_n$ by $\Phi(o)=(\rho(o),\phi(o))$. For example, $\Phi((1\mbox{-}3,2))=(\{1\mbox{-}3,2\},id)$ and $\Phi((4,3\mbox{-}5))=(\{4,3\mbox{-}5)\},(1,2))$ where id is the identity permutation.
The maps $\widetilde W$, $\partial$, $\widetilde R$, and $\widetilde\partial$ are carried over to $K^*\times S_n$, $C_*(UD_n\Gamma)\times S_n$, and $M_*(UD_n\Gamma)\times S_n$ by conjugating with $\Phi$. For example, the $i$-th boundary homomorphism on $M_*(UD_n\Gamma)\times S_n$ is given by $\Phi\circ\widetilde\partial\circ\Phi^{-1}$. To make the notation more compact, an element $(c,\sigma)\in K^*\times S_n$ will be denoted by $c_\sigma$.

\begin{exa}\label{ex:K33n2-order}
Let $\Gamma$ be $K_{3,3}$ and a maximal tree and an order be given as Figure~\ref{fig:K33n2}. We want to compute $H_1(P_2K_{3,3})$ which will be used later.
\end{exa}
From Figure~\ref{fig:MD2}, one can see that $H_1(P_2K_{3,3})\cong\mathbb Z^8$. But we want to demonstrate how compute $H_1(P_2K_{3,3})$ using the Morse chain complex. Let $\sigma$ be the permutation $(1,2)\in S_2$. There are two critical 0-cells $\{0,1\}_{\rm id}$, $\{0,1\}_\sigma$. There are fourteen critical 1-cells $\{0\mbox{-}3,1\}_{\rm id}$, $\{0\mbox{-}3,1\}_\sigma$, $\{0\mbox{-}4,1\}_{\rm id}$, $\{0\mbox{-}4,1\}_\sigma$, $\{0\mbox{-}4,5\}_{\rm id}$, $\{0\mbox{-}4,5\}_\sigma$, $\{1\mbox{-}5,0\}_{\rm id}$,
$\{1\mbox{-}5,0\}_\sigma$, $\{1\mbox{-}5,2\}_{\rm id}$, $\{1\mbox{-}5,2\}_\sigma$, $\{2\mbox{-}4,3\}_{\rm id}$, $\{2\mbox{-}4,3\}_\sigma$, $\{3\mbox{-}5,0\}_{\rm id}$, $\{3\mbox{-}5,0\}_\sigma$ and their image under $\widetilde\partial$ are as follows:
$$\widetilde\partial(\{0\mbox{-}3,1\}_{\rm id})=\widetilde R(-\{1,3\}_\sigma+\{0,1\}_{\rm id})=-\{0,1\}_\sigma+\{0,1\}_{\rm id}$$
since $\widetilde R(\{1,3\}_\sigma)=\widetilde R(\partial(\{0\mbox{-}1,3\}_\sigma)+\{1,3\}_\sigma)=\widetilde R(\{0,3\}_\sigma)=\cdots=\{0,1\}_\sigma$ (see Lemma~\ref{lem:naturalmoveD2}). So $\widetilde\partial(\{0\mbox{-}3,1\}_\sigma)=-\{0,1\}_{\rm id}+\{0,1\}_\sigma$ because of $\sigma^2=\mbox{\rm id}$.
Similarly we can compute images of critical 1-cells as follows:
\begin{align*}
\widetilde\partial(\{0\mbox{-}4,1\}_{\rm id})=&\widetilde\partial(\{1\mbox{-}5,2\}_{\rm id})=\widetilde\partial(\{2\mbox{-}4,3\}_{\rm id})= -\{0,1\}_\sigma+\{0,1\}_{\rm id}\\
\widetilde\partial(\{0\mbox{-}4,5\}_{\rm id})=&\widetilde\partial(\{1\mbox{-}5,0\}_{\rm id})=\widetilde\partial(\{3\mbox{-}5,0\}_{\rm id})=0.
\end{align*}

There are six critical 2-cells $\{0\mbox{-}3,1\mbox{-}5\}_{\rm id}$, $\{0\mbox{-}3,1\mbox{-}5\}_\sigma$, $\{0\mbox{-}4,1\mbox{-}5\}_{\rm id}$, $\{0\mbox{-}4,1\mbox{-}5\}_\sigma$, $\{0\mbox{-}4,3\mbox{-}5\}_{\rm id}$, $\{0\mbox{-}4,3\mbox{-}5\}_\sigma$. We compute boundaries for the first two.
$$\widetilde\partial(\{0\mbox{-}3,1\mbox{-}5\}_{\rm id})
=\widetilde R(-\{1\mbox{-}5,3\}_\sigma+\{1\mbox{-}5,0\}_{\rm id}+\{0\mbox{-}3,5\}_{\rm id}-\{0\mbox{-}3,1\}_{\rm id}).$$
Since $\{1\mbox{-}5,0\}_{\rm id}$ and $\{0\mbox{-}3,1\}_{\rm id}$ are critical 1-cells, we only consider other two 1-cells.
\begin{align*}
\widetilde R(\{1\mbox{-}5,3\}_\sigma)=&\widetilde R(-\partial(\{1\mbox{-}5,2\mbox{-}3\}_\sigma)+\{1\mbox{-}5,3\}_\sigma)\\
=&\widetilde R(\{2\mbox{-}3,5\}_{\rm id}-\{2\mbox{-}3,1\}_\sigma+\{1\mbox{-}5,2\}_\sigma)=\{1\mbox{-}5,2\}_\sigma
\end{align*}
since $\widetilde R(\{2\mbox{-}3,1\})=0$ and no critical 1-cell appears in the process of computing $\widetilde R(\{2\mbox{-}3,1\})$ (see Example~\ref{ex:K33n2}).
\begin{align*}
\widetilde R(\{0\mbox{-}3,5\}_{\rm id})&=\widetilde R(\{0\mbox{-}3,4\}_{\rm id})=\widetilde R(-\partial(\{0\mbox{-}3,2\mbox{-}4\}_{\rm id})+\{0\mbox{-}3,4\}_{\rm id})\\
&=\widetilde R(\{2\mbox{-}4,3\}_\sigma-\{2\mbox{-}4,0\}_{\rm id}+\{0\mbox{-}3,2\}_{\rm id})\\
&=\{2\mbox{-}4,3\}_\sigma+\widetilde R(\{0\mbox{-}3,1\}_{\rm id})=\{2\mbox{-}4,3\}_\sigma+\{0\mbox{-}3,1\}_{\rm id}.
\end{align*}
So $\widetilde\partial(\{0\mbox{-}3,1\mbox{-}5\}_{\rm id})=-\{1\mbox{-}5,2\}_\sigma+\{1\mbox{-}5,0\}_{\rm id}+\{2\mbox{-}4,3\}_\sigma$. This implies $\widetilde\partial(\{0\mbox{-}3,1\mbox{-}5\}_\sigma)=-\{1\mbox{-}5,2\}_{\rm id}+\{1\mbox{-}5,0\}_\sigma+\{2\mbox{-}4,3\}_{\rm id}.$

Over these critical cells, the second boundary map is represented by the following matrix.
$$\left(
\begin{array}{cccccccccccccc}
0&0&0&0&0&0&1&0&0&-1&0&1&0&0\\
0&0&0&0&0&0&0&1&-1&0&1&0&0&0\\
0&0&-1&0&1&0&1&0&0&-1&0&0&0&0\\
0&0&0&-1&0&1&0&1&-1&0&0&0&0&0\\
0&0&-1&0&1&0&0&0&0&0&1&0&0&0\\
0&0&0&-1&0&1&0&0&0&0&0&1&0&0
\end{array} \right)$$
Since the kernel of the first boundary map is generated by either $o$ or $o\pm\{0\mbox{-}3,1\}_{\rm id})$ for all other critical 1-cells $o$, the matrix obtained from the above matrix by deleting the first column is a presentation matrix of $H_1(P_2K_{3,3})$. Using row operations on the presentation matrix, we obtain the following echelon form.
$$\left(
\begin{array}{ccccccccccccc}
0&\framebox[0.4cm][l]{-1}&0&1&0&1&0&0&-1&0&0&0&0\\
0&0&\framebox[0.4cm][l]{-1}&0&1&0&1&-1&0&0&0&0&0\\
0&0&0&0&0&\framebox[0.4cm][l]{-1}&0&0&1&1&0&0&0\\
0&0&0&0&0&0&\framebox[0.4cm][l]{-1}&1&0&0&1&0&0\\
0&0&0&0&0&0&0&0&0&\framebox[0.4cm][l]{1}&1&0&0\\
0&0&0&0&0&0&0&0&0&0&0&0&0
\end{array} \right)$$
Thus $H_1(P_2K_{3,3})\cong\mathbb Z^8$.
\qed

\subsection{The second boundary homomorphism}\label{ss23:images}

To give a general computation of the second boundary homomorphism $\widetilde\partial$ on a Morse complex, we first exhibit redundant 1-cells whose reductions are straightforward and then explain how to choose a maximal tree of a given graph to take advantage of these simple reductions.

Let $\Gamma$ be a graph and $T$ be a maximal tree of $\Gamma$. Let $c$ be a redundant $i$-cell in $UD_n\Gamma$, $v$ be an unblocked vertex in $c$ and $e$ be the edge in $T$ starting from $v$. Let $V_{e}(c)$ denote the $i$-cell obtained from $c$ by replacing $v$ by $\tau(e)$. Define a function $V:K_i\to K_i$ by $V(c)=V_e(c)$ if $c$ is redundant and
$\iota(e)$ is the smallest unblocked vertex in $c$, and by $V(c)=c$ otherwise.
The function $V$ should stabilize to a function $\widetilde V:K_i\to K_i$ under iteration, that is, $\widetilde V=V^m$ for some non-negative integer $m$ such that $V^m=V^{m+1}$.

\begin{lem}{\rm (Kim-Ko-Park \cite{KKP})}\label{lem:naturalmove}
Let $c$ be a redundant cell and $v$ be a unblocked vertex. Suppose that for the edge $e$ starting from $v$, there is no vertex $w$ that is either in $c$ or an end vertex of an edge in $c$ and satisfies $\tau(e)< w<\iota(e)$. Then $\widetilde R(c)=\widetilde RV_e(c)$.
\end{lem}

We continue to define more notations and terminology.  For each vertex $v$ in $\Gamma$, there is a unique edge path $\gamma_v$ from $v$ to the base vertex 0 in $T$. For vertices $v$, $w$ in $\Gamma$, $v\wedge w$ denotes the vertex that is the first intersection between $\gamma_v$ and $\gamma_w$. Obviously, $v\wedge w\le v$ and $v\wedge w\le w$. The number assigned to the branch of $v$ occupied by the path from $v$ to $w$ in $T$ is denoted by $g(v,w)$. If $v=w\wedge v$, $g(v,w)\ge 1$ and if $v>w\wedge v$, $g(v,w)=0$. An edge $e$ in $\Gamma$ is said to be {\em separated} by a vertex $v$ if $\iota(e)$ and $\tau(e)$ lie in two distinct components of $T-\{v\}$. It is clear that only a deleted edge can be separated by a vertex. If a deleted edges $d$ is not separated by $v$, then $\iota(d)$, $\tau(d)$, and $\iota(d)\wedge\tau(d)$ are all in the same component of $T-\{v\}$.

For redundant 1-cells, we can strengthen the above lemma as follows.
\begin{lem}{\rm [Special Reduction]}\label{lem:naturalmove-2}
Let $c$ be a redundant 1-cell containing an edge $p$. Suppose the redundant 1-cell $c$ has an unblocked vertex $v$ and the edge $e$ starting from $v$ satisfies the following:
\begin{itemize}
\item[(a)] Every vertex $w$ in $c$ satisfying $\tau(e)< w<\iota(e)$ is blocked.
\item[(b)] If an end vertex $w$ of $p$ satisfies $\tau(e)< w<\iota(e)$ then $p$ is not separated by $\tau(e)$.
\end{itemize}
Then $\widetilde R(c)=\widetilde RV_e(c)$. Therefore if $p$ is not a deleted edge then $\widetilde R(c)=\widetilde R\widetilde V(c)$.
\end{lem}
\begin{proof}
Assume that both ends of $p$ are not between $\tau(e)$ and $\iota(e)$. Since $p$ is the only edge in $c$ that can initiate a blockage, it is impossible to have a vertex between $\tau(e)$ and $\iota(e)$ due to the condition (a). Then we are done by Lemma~\ref{lem:naturalmove}.

Assume that an end of $p$ is between $\tau(e)$ and $\iota(e)$. By the condition (b),
both $\iota(p)$ and $\tau(p)$ are in the same component $T_p$ of $T-\{\tau(e)\}$ and are between $\tau(e)$ and $\iota(e)$. For a vertex $w$ in $T_p$, $c_w$ denotes the 1-cell obtained from $c$ by replacing $v$ by $e$ and $p$ by $w$.
We will show that $\widetilde R(c_{\iota(p)})=\widetilde R(c_{\tau(p)})$. Then
$$\widetilde R(c)=\widetilde RV_e(c)\pm \{\widetilde R(c_{\iota(p)})-\widetilde R(c_{\tau(p)})\}=\widetilde RV_e(c)$$ where the sign $\pm$ is determined by the order between the terminal vertices of $p$ and $e$.

Let $W$ be the set of all 1-cells obtained from $c_{\iota(p)}$ replacing vertices in $T_p$ by vertices that are also in $T_p$. If $c'\in W$ has no unblocked vertex in $T_p$ then $c'$ is unique because $\Gamma$ is suitably subdivided. This 1-cell is denoted by $c_p$. If $c'\in W$ has an unblocked vertex in $T_p$, let $u$ be the smallest unblocked vertex in $T_p$ and $e'$ be the edge starting from $u$. Then $c'$ and $u$ satisfy the hypothesis of Lemma~\ref{lem:naturalmove} since $e$ is the only edge in $c'$ and every vertex in $T-T_p$ is not between $\tau(e)$ and $\iota(e)$. So $\widetilde R(c')=\widetilde RV_{e'}(c')$. By iterating this argument, we have $\widetilde R (c_{\iota(p)})=\widetilde R (c_p)=\widetilde R (c_{\tau(p)})$ because $V_{e'}(c')$ is also in the finite set $W$.

If $p$ is not a deleted edge, then the condition (b) always holds and so $c$ and the smallest unblocked vertex in $c$ satisfy the hypothesis of this lemma. So $\widetilde R(c)=\widetilde RV(c)$. By repeating the argument, we have $\widetilde R(c)=\widetilde R\widetilde V(c)$.
\end{proof}

For an oriented discrete configuration space $D_n\Gamma$, the statement corresponding to Lemma~\ref{lem:naturalmove} holds at least for $n=2$ (see Lemma~\ref{lem:naturalmoveD2}), but the statement corresponding to Lemma~\ref{lem:naturalmove-2} is false in general.

For example, let $\Gamma$ be the graph in Figure~\ref{fig:Exo2}.
\begin{figure}[ht]
\psfrag{0}{\small0}
\psfrag{4}{\small4}
\psfrag{6}{\small6}
\psfrag{8}{\small8}
\psfrag{10}{\small10}
\psfrag{12}{\small12}
\centering
\includegraphics[height=1.5cm]{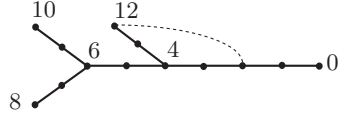}
\caption{The graph $\Gamma$}
\label{fig:Exo2}
\end{figure}
We consider the critical 2-cell $o=(2\mbox{-}12,6\mbox{-}9,7)$ in $D_3\Gamma$. In the unordered case, opposite sides have the same images under $\widetilde V$ but $\widetilde V((2\mbox{-}12,6,7))=(2\mbox{-}12,3,4)$ and $\widetilde V((2\mbox{-}12,9,7))=(2\mbox{-}12,4,3)$. Furthermore,
\begin{align*}
\widetilde R((12,6\mbox{-}9,7))&=\widetilde R((11,6\mbox{-}9,7))\\
&=\widetilde R((4,6\mbox{-}9,7)-(4\mbox{-}11,6,7)+(4\mbox{-}11,9,7))\\
&=(0,6\mbox{-}9,7)-(4\mbox{-}11,5,6)+(4\mbox{-}11,6,5)\\
&\ne \widetilde V((12,6\mbox{-}9,7)).
\end{align*}

Discrete Morse theory can be powerful in discrete situations but we need to reduce the number of instances to be investigated and the amount of computation involved for each instance. In our situation, it is important to choose a nice maximal tree and its planar embedding. The following lemma make such choices which will be used throughout the article. For example, the Morse complex induced from such choices has the second boundary map describable by using Lemma~\ref{lem:naturalmove-2}.

From now on, we assume that every graph is suitably subdivided, finite, and connected unless stated otherwise. When $n=2$, it is convenient to additionally assume that each path between two vertices of valency $\ne 2$ in a suitably subdivided graph contains at least two edges.

\begin{lem}{\rm [Maximal Tree and Order]}\label{lem:treeandorder}
For a given graph $\Gamma$, there is a maximal tree and its planar embedding so that the induced order on vertices satisfies:
\begin{itemize}
\item[(T1)] The initial vertices of all deleted edges are vertices of valency 2.
\item[(T2)] Every deleted edge $d$ is not separated by any vertex $v$ such that $v<\tau(d)$;.
\item[(T3)] If the $k$-th branch of a vertex $v$ has the property that $v$ separates a deleted edge $d$ and $g(v,\iota(d))=k$, and the $j$-th branch of $v$ does not have the property, then $j<k$.
\end{itemize}
\end{lem}
\begin{proof}
We construct a desired maximal tree in the following three steps.
\paragraph{(I) Choice of a base vertex 0 on $\Gamma$}
We assign 0 to a vertex $v$ such that $v$ is of valency 1 in $\Gamma$ or $\Gamma-\{v\}$ is connected if there is no vertex of valency 1. This is necessary to make the base vertex have valency 1 in a maximal tree so that there is one critical 0-cell.
\paragraph{(II) Choice of deleted edges}
We consider a metric on $\Gamma$ such that each edge is of length 1.
\begin{itemize}
\item[(1)] Delete an edge nearest from 0 on a circuit nearest from 0.
\item[(2)] Repeat (1) until the remainder is a tree $T$
\end{itemize}
Then the order on vertices obtained any planar embedding $p$ of $T$ satisfy the conditions (T1) and (T2) since the terminal vertices of all deleted edges are of valency $\ge 3$ in $\Gamma$.
\paragraph{(III) Modification of a planar embedding}
If the order on vertices obtained by $p$ does not satisfy the condition (T3), then there are a vertex $A$ with valency $\ge 3$ on $T$ and branches $j$ of $A$ that violate (T3). The base vertex 0 and branches $j$ do not lie on the same component of $\Gamma-\{A\}$. We slide the components containing branches $j$ over other branches so that every branch of $A$ satisfies (T3) (see Figure~\ref{fig:CMT}). We repeat this process until the induced order satisfies (T3).
\begin{figure}[ht]
\psfrag{0}{\small0}
\centering
\includegraphics[height=1.5cm]{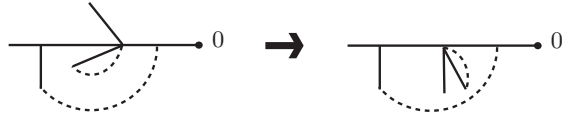}
\caption{Modification of a branch}
\label{fig:CMT}
\end{figure}
\end{proof}

From now on, we assume that we always choose a maximal tree and its embedding as given in Lemma~\ref{lem:treeandorder}.

\begin{exa}\label{ex:K5TO}
A maximal tree of $K_5$ and its planar embedding according to Lemma~\ref{lem:treeandorder} for $n=4$ is given in Figure~\ref{fig:K5n4T} and Figure~\ref{fig:K5n4TO}
\end{exa}

\begin{figure}[ht]
\psfrag{0}{\small0}
\centering
\includegraphics[height=2cm]{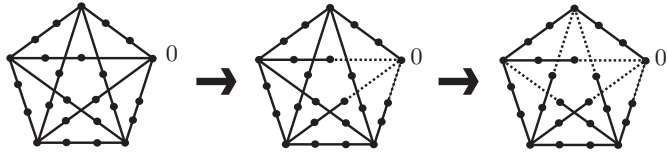}
\caption{Choice of a maximal tree of $K_5$}
\label{fig:K5n4T}
\end{figure}

\begin{figure}[ht]
\psfrag{0}{\small0}
\psfrag{3}{\small3}
\psfrag{6}{\small6($=A$)}
\psfrag{8}{\small8}
\psfrag{11}{\small11($=B$)}
\psfrag{13}{\small13}
\psfrag{15}{\small15}
\psfrag{18}{\small18($=C$)}
\psfrag{20}{\small20}
\psfrag{22}{\small22}
\psfrag{24}{\small24}
\psfrag{d1}{\small$d_1$}
\psfrag{d2}{\small$d_2$}
\psfrag{d3}{\small$d_3$}
\psfrag{d4}{\small$d_4$}
\psfrag{d5}{\small$d_5$}
\psfrag{d6}{\small$d_6$}
\centering
\includegraphics[height=3cm]{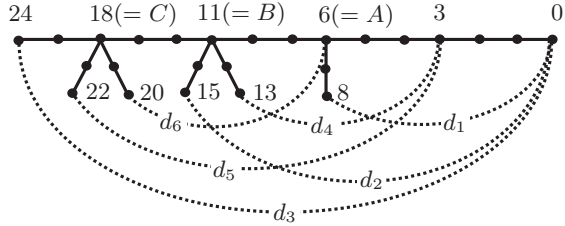}
\caption{Order on the maximal tree of $K_5$}
\label{fig:K5n4TO}
\end{figure}

When we work with an arbitrary graph $\Gamma$ of an arbitrary index $n$, it is convenient to represent cells of $UD_n\Gamma$ by using the following notations used in \cite{FS, KKP}. Let $A$ be a vertex of valency $\mu+1$ $(\ge 3)$ in a maximal tree of $\Gamma$. Starting from the branch clockwise next to the branch joining to the base vertex, we number branches incident to $A$ clockwise. Let $\vec a$ be a vector $(a_1,\ldots,a_\mu)$ of nonnegative integers and let $|\vec a|=\sum_{i=1}^\mu a_i$. And $\vec\delta_k$ denotes the $k$-th coordinate unit vector. Then for $1\le k\le \mu$, $A_k(\vec a)$ denotes the set consisting of one edge $e$ with $\tau(e)=A$ that lies on the $k$-th branch together with $a_i$ blocked vertices that lie on the $i$-th branch. Sometimes the edge $e$ is denoted by $A_k$. Note that this definition is little different from the one used in \cite{FS, KKP} but is more convenient in this work. For $1\le s\le n$, $0_s$ denotes the set $\{0,1,\ldots,s-1\}$ of $s$ consecutive vertices from the base vertex. Let $\dot A(\vec a)$ denote the set of vertices consisting of $A$
together with $a_i$ blocked vertices that lies on the $i$-th branch and let $A(\vec a)=\dot A(\vec a)-\{A\}$. Then $A(\vec a)$ can be obtained from $A_k(\vec a-\vec\delta_k)$ by replacing an edge $e$ with $\iota(e)$. Every critical $i$-cell is represented by the following union:
$$ A^1_{k_1}(\vec a^1)\cup\ldots\cup A^\ell_{k_\ell}(\vec a^\ell)\cup\{d_1,\ldots,d_q\}\cup\{v_1,\ldots,v_r\}\cup 0_s,$$
where $A^1,\ldots,A^\ell$ are vertices of valency $\ge 3$, and $d_1,\ldots,d_q$ are deleted edges, and $v_1,\ldots,v_r$ are blocked vertices blocked by deleted edges. Furthermore, since $s$ is uniquely determined by $s=n-(\ell+|\vec a^1|+\cdots+|\vec a^\ell|+q+r)$, we will omit $0_s$ in the notation. Let $\vec a-1$ denote the vector obtained from $\vec a$ by subtracting 1 from the first positive entry. Then $\vec a-\alpha$ denotes the vector obtained from $\vec a$ by iterating the above operation $\alpha$ times. Define $p(\vec a)=i$ if $a_i$ is the first nonzero entry of $\vec a$. For $1\le k\le \mu$, set $(\vec a)_k=(a_1,\ldots,a_{k-1},0,\ldots,0)$ and $|\vec a|_k=a_1+\cdots+a_{k-1}$.

By Condition (T1), there are no vertices blocked by the initial vertex of any deleted edge. Let $d(\vec a)$ denote the set consisting of a deleted edge $d$ together with $a_i$ blocked vertices that lie on the $i$-th branch of $\tau(d)$ for each $i$. Every critical 2-cell can be represented by one of the following forms:
$$A_k(\vec a)\cup B_\ell(\vec b),\ A_k(\vec a)\cup d(\vec b),\ d(\vec a) \cup d'(\vec b)$$
where $A$ and $B$ are vertices of valency $\ge 3$ in $T$, $d$ and $d'$ are deleted edges.
Condition (T2) implies that there is no pair of edges such that the terminal vertex of one edge separates the other edge and vice versa. So we need not handle this troublesome case. Condition (T3) will be used in Section~\ref{ss31:PMatrix}.

The following notation is useful in describing images under the second boundary map:
$$\mathbf{A}(\vec a,\ell)=R(\sum^{|\vec a|}_{\alpha = 0} A_{p(\vec a-\alpha)}((\vec a -\alpha)-\vec\delta_{p(\vec a-\alpha)}+\vec\delta_\ell))$$
where $A$ is a vertex of valency $\ge3$, $\vec a$ is a vector defined at $A$, and $1\le\ell\le\mu$. It is straightforward to see that a sum of critical 1-cells represented by this notation has the following properties.

\begin{pro}\label{pro:A1}
\begin{itemize}
\item[(i)] If $a_m=b_m$ for all $m>\ell$, then $\mathbf{A}(\vec a,\ell)=\mathbf{A}(\vec b,\ell)$.
\item[(ii)] If $p(\vec a)>\ell$, then $\mathbf{A}(\vec a,\ell)-\mathbf{A}(\vec a -1,\ell)=R(A_{p(\vec a)}(\vec a-\vec\delta_{p(\vec a)} + \vec\delta_\ell))$.
\end{itemize}
\end{pro}

As mentioned above, there are three types of critical 2-cells. We will describe the images of each of these three types under $\widetilde\partial$. Since an edge $A_k$ is never separated by any vertex, Lemma~\ref{lem:treeandorder} implies $\widetilde\partial(A_k(\vec a)\cup B_\ell(\vec b))=0$, which was first proved by Farley and Sabalka in~\cite{FS}. So we consider the remaining two types. To help grasp the idea behind,  examples are followed by general formulae.

\begin{exa}
Let $\Gamma$ be $K_5$ and a maximal tree
and an order be given as Example~\ref{ex:K5TO}. We want to
compute $\widetilde\partial(c)$ for the 2-cell $c=B_3(1,0,1)\cup d_2$ in $M_2(UD_4\Gamma)$.
\end{exa}
\begin{figure}[ht]
\psfrag{0}{\small0}
\psfrag{11}{\small$B$}
\psfrag{d2}{\small$d_2$}
\centering
\includegraphics[height=2cm]{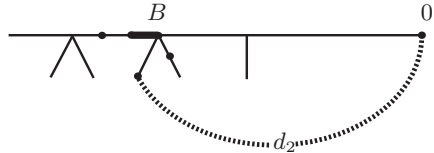}
\caption{$B_3(1,0,1)\cup d_2$}
\label{fig:Ex1}
\end{figure}
Since $\tau(d_2)<B$, using $\widetilde\partial=\widetilde R\partial$ and Lemma~\ref{lem:naturalmove} we have
\begin{align*}
\widetilde\partial(c)=&\widetilde R(B(1,0,2)\cup d_2)-\widetilde R(\dot B(1,0,1)\cup d_2)-\widetilde R(B_3(1,0,1)\cup\iota(d_2))+B_3(1,0,1)\\
=&\widetilde R(B(0,0,2)\cup d_2)-\widetilde R(B(0,0,1)\cup d_2)-B_3(1,1,1)+B_3(1,0,1)
\end{align*}
Since $B_3(0,0,1)\cup d_2$ is collapsible, using Lemma~\ref{lem:naturalmove}, we have
\begin{align*}
0=&\:\widetilde\partial\widetilde R(B_3(0,0,1)\cup d_2)=\widetilde R\partial(B_3(0,0,1)\cup d_2)\\
=&\:\widetilde R(B(0,0,2)\cup d_2-\dot B(0,0,1)\cup d_2-B_3(0,0,1)\cup \iota(d_2)+B_3(0,0,1))\\
=&\:\widetilde R(B(0,0,2)\cup d_2)-\widetilde R(B(0,0,1)\cup d_2) -B_3(0,1,1)
\end{align*}
Similarly
\begin{align*}
0=&\:\widetilde\partial\widetilde R(B_3\cup d_2)=\widetilde R\partial(B_3\cup d_2)\\
=&\:\widetilde R(B(0,0,1)\cup d_2-\{B\}\cup d_2-B_3\cup \iota(d_2)+B_3)\\
=&\:\widetilde R(B(0,0,1)\cup d_2)-d_2-B_3(0,1,0)
\end{align*}
So
\begin{align*}
\widetilde\partial(c)=&B_3(1,0,1)-B_3(1,1,1)+B_3(0,1,1)\\
=&B_3(1,0,1)-B_3(1,1,1)-\mathbf{B}((1,0,1),2)+\mathbf{B}((1,0,2),2)
\end{align*}
\qed

\begin{lem}{\rm[Boundary Formula I]}\label{lem:boundary-1}
Let $c=A_k(\vec a)\cup d(\vec b)$ and $\ell=g(A,\iota(d))$. If $d$ is separated by $A$, $$\widetilde\partial (c)=A_k(\vec a)-A_k(\vec a+\vec\delta_\ell)-\mathbf{A}(\vec a,\ell)+\mathbf{A}(\vec a+\vec\delta_k,\ell).$$
Otherwise, $\widetilde\partial (c)=0$.
\end{lem}
\begin{proof}
Let $B=\tau(d)$. Then
\begin{align*}
\widetilde\partial(c)&=\widetilde R\partial(A_k(\vec a)\cup d(\vec b))\\
&=\pm\widetilde R(\dot A(\vec a)\cup d(\vec b)-A(\vec a+\vec\delta_k)\cup d(\vec b)-A_k(\vec a)\cup \dot B(\vec b)+A_k(\vec a)\cup B(\vec b)\cup\{\iota(d)\})
\end{align*}
where the sign is determined by the order between $A$ and $B$.

Since $A_k$ is not separated by any vertex, Lemma~\ref{lem:naturalmove-2} implies $\widetilde R(A_k(\vec a)\cup \dot B(\vec b))=\widetilde R\circ\widetilde V(A_k(\vec a)\cup \dot B(\vec b))$ and $\widetilde R(A_k(\vec a)\cup B(\vec b)\cup\{\iota(d)\} )=\widetilde R\circ\widetilde V(A_k(\vec a)\cup B(\vec b)\cup\{\iota(d)\} )$.

Assume that $d$ is not separated by $A$. Then $\widetilde V(A_k(\vec a)\cup \dot B(\vec b))=\widetilde V(A_k(\vec a)\cup B(\vec b)\cup\{\iota(d)\})$. So we only consider $\widetilde R(\dot{A}(\vec a)\cup d(\vec b)-A(\vec a+\vec\delta_k)\cup d(\vec b))$. Let $C$ be the unique largest vertex of valency $\ge 3$ such that $C<A$. Since $d$ is not separated by any vertex between $C$ and $A$, Lemma~\ref{lem:naturalmove} implies $\widetilde R(\dot{A}(\vec a)\cup d(\vec b))=\widetilde R(C((|\vec a|+1)\vec\delta_{g(C,A)})\cup d(\vec b))=\widetilde R(A(\vec a+\vec\delta_k)\cup d(\vec b))$. Thus $\widetilde\partial(c)=0$.

Assume that $d$ is separated by $A$. By Condition (T1) on our maximal tree, $A>B=\tau(d)$ and so the negative sign is valid in the expression of $\widetilde\partial(c)$ above.  Lemma~\ref{lem:naturalmove-2} implies $\widetilde R(A_k(\vec a)\cup \dot B(\vec b))=A_k(\vec a)$ and $\widetilde R(A_k(\vec a)\cup B(\vec b)\cup\{\iota(d)\})=\widetilde R\circ\widetilde V(A_k(\vec a+\vec\delta_\ell)\cup B(\vec b))=A_k(\vec a+\vec\delta_\ell)$. Let $m=g(B,A)$. Since $\widetilde R(\dot A(\vec a)\cup d(\vec b))=\widetilde R(A(\vec a)\cup d(\vec b+\vec\delta_m))$, it is sufficient to prove the formula $$\widetilde R(A(\vec a)\cup d(\vec b))=d(\vec b +|\vec a|\vec\delta_m)+\mathbf{A}(\vec a,\ell).$$
We use the induction on $|\vec a|$.
\begin{align*}
&\widetilde R(A(\vec a)\cup d(\vec b))\\
=\:&\widetilde R(\dot A(\vec a -1)\cup d(\vec b)+A_{p(\vec a)}(\vec a-\vec\delta_{p(\vec a)})\cup B(\vec b) \cup \{\iota(d)\})\\
&-\widetilde R(A_{p(\vec a)}(\vec a-\vec\delta_{p(\vec a)})\cup \dot B(\vec b))\\
=\:& \widetilde R(A(\vec a -1)\cup d(\vec b+\vec\delta_m)+A_{p(\vec a)}(\vec a+\vec\delta_\ell-\vec\delta_{p(\vec a)})-A_{p(\vec a)}(\vec a-\vec\delta_{p(\vec a)}))\\
=\:& d(\vec b+|\vec a|\vec\delta_m)+\mathbf{A}(\vec a-1,\ell)+R(A_{p(\vec a)}(\vec a+\vec\delta_\ell-\vec\delta_{p(\vec a)}))\\
=\:& d(\vec b +|\vec a|\vec\delta_m)+\mathbf{A}(\vec a,\ell)
\end{align*}
Notice that $A_{p(\vec a)}(\vec a-\vec\delta_{p(\vec a)})$ is collapsible. It is easy to verify the formula for $|\vec a|=1$.
\end{proof}

Let $d$ and $d'$ be deleted edges such that $\tau(d)>\tau(d')$, $C=\iota(d)\wedge\iota(d')$, $\ell=\min\{g(C,\iota(d)),g(C,\iota(d')\}$ and $k=\max\{g(C,\iota(d)),g(C,\iota(d')\}$. Then we define $$\mbox{$\wedge$}(d,d')=C_k(\vec\delta_\ell)$$

\begin{exa}
Let $\Gamma$ be $K_5$ and a maximal tree
and an order be given as Example~\ref{ex:K5TO}. We want to
compute $\widetilde\partial(c)$ for the 2-cell $c=d_6(0,1)\cup d_4$ in $M_2(UD_4\Gamma)$.
\end{exa}
\begin{figure}[ht]
\psfrag{0}{\small0}
\psfrag{6}{\small$A$}
\psfrag{11}{\small$B$}
\psfrag{d4}{\small$d_4$}
\psfrag{d6}{\small$d_6$}
\centering
\includegraphics[height=2cm]{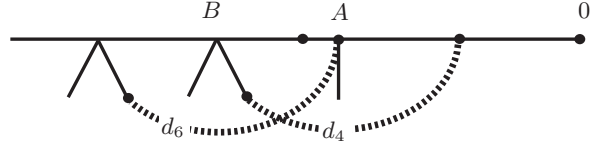}
\caption{$d_6(0,1)\cup d_4$}
\label{fig:Ex2}
\end{figure}
Since $\tau(d_4)<\tau(d_6)$, using $\widetilde\partial=\widetilde R\partial$ and Lemma~\ref{lem:naturalmove} we have
\begin{align*}
\widetilde\partial(c)=&\widetilde R(A(0,1)\cup\{\iota(d_6)\}\cup d_4)-\widetilde R(\dot A(0,1)\cup d_4)\\
&-\widetilde R(d_6(0,1)\cup\iota(d_4))+\widetilde R(d_6(0,1)\cup\tau(d_4))\\
=&\widetilde R(B(0,0,1)\cup d_4(1))-d_4(2)-d_6(0,2)+d_6(0,1)
\end{align*}
Since $B_3\cup d_4(1)$ is collapsible, using Lemma~\ref{lem:naturalmove}, we have
\begin{align*}
0=&\:\widetilde\partial\widetilde R(B_3\cup d_4(1))=\widetilde R\partial(B_3\cup d_4(1))\\
=&\:\widetilde R(B(0,0,1)\cup d_4(1)-\{B\}\cup d_4(1)-B_3\cup\iota(d_4)+B_3)\\
=&\:\widetilde R(B(0,0,1)\cup d_4(1))-d_4(2) -B_3(1,0,0).
\end{align*}
So
\begin{align*}
\widetilde\partial(c)=&d_6(0,1)-d_6(0,2)+\{d_4(2)+B_3(1,0,0)\}-d_4(2)\\
=&d_6(0,1)-d_6(0,2)+\mbox{$\wedge$}(d_6,d_4).
\end{align*}
\qed

\begin{lem}{\rm[Boundary Formula II]}\label{lem:boundary-2}
Let $c=d(\vec a)\cup d'(\vec b)$ such that $\tau(d)>\tau(d')$ and let $A=\tau(d)$, $k=g(A,\iota(d))$, and $\ell=g(A,\iota(d'))$. If $d'$ is separated by $A$,
$$\widetilde\partial (c)=d(\vec a)-d(\vec a+\vec\delta_k)-\mathbf{A}(\vec a,\ell)+\mathbf{A}(\vec a+\vec\delta_k,\ell)+\varepsilon\mbox{$\wedge$}(d,d')$$
where $\varepsilon=0$ for $k\ne\ell$, $\varepsilon=-1$ for $k=\ell$ and $\iota(d)<\iota(d')$, and $\varepsilon=1$ for $k=\ell$ and $\iota(d')<\iota(d)$.
Otherwise, $\widetilde\partial (c)=0$.
\end{lem}
\begin{proof}
Let $B=\tau(d')$. Then $A=\tau(d)>B$. So $$\widetilde\partial(c)=\widetilde R(A(\vec a)\cup d'(\vec b)\cup\{\iota(d)\}-\dot A(\vec a)\cup d'(\vec b)-d(\vec a)\cup B(\vec b)\cup\{\iota(d')\}+d(\vec a)\cup \dot B(\vec b)).$$
Note that $d$ is not separated by $B$. Assume $d'$ is not separated by $A$. By Lemma~\ref{lem:naturalmove-2},
$\widetilde R(A(\vec a)\cup d'(\vec b)\cup\{\iota(d)\})=\widetilde R(\dot A(\vec a)\cup d'(\vec b))$, $\widetilde R(d(\vec a)\cup B(\vec b)\cup\{\iota(d')\})=\widetilde R(d(\vec a)\cup \dot B(\vec b))$ and so $\widetilde\partial (c)=0$.

Now assume that $d'$ is separated by $A$. If $k\ne\ell$, then $d'$ (and $d$, respectively) is not separated by any vertex other than $A$ on the path between $A$ and $\iota(d)$ (and $\iota(d')$). So we see that $\widetilde R(A(\vec a)\cup d'(\vec b)\cup\{\iota(d)\})=\widetilde R(A(\vec a +\vec\delta_k)\cup d'(\vec b))$ and $\widetilde R(d(\vec a)\cup B(\vec b)\cup\{\iota(d')\})=\widetilde R(d(\vec a+\vec\delta_\ell)\cup B(\vec b))$.

Assume $k=\ell$. Let $C=\iota(d)\wedge\iota(d')$, $m=g(C,\iota(d'))$ and $p=g(C,\iota(d))$. Then $A<C$. If $\iota(d)<\iota(d')$, then $p<m$ and so Lemma~\ref{lem:naturalmove-2} implies
$$\widetilde R(A(\vec a)\cup d'(\vec b)\cup\{\iota(d)\})=\widetilde R(A(\vec a +\vec\delta_k)\cup d'(\vec b)).$$
And we have
\begin{align*}
\widetilde R(d(\vec a)\cup B(\vec b)\cup\{\iota(d')\})=&\widetilde R(d(\vec a)\cup B(\vec b)\cup\{C\}+A(\vec a)\cup B(\vec b)\cup C_m(\delta_m)\cup\{\iota(d)\})\\
&-\widetilde R(\dot A(\vec a)\cup B(\vec b)\cup C_m(\vec\delta_m))\\
=&\widetilde R(d(\vec a+\vec\delta_\ell)\cup B(\vec b))+\widetilde R(A(\vec a)\cup B(\vec b)\cup \mbox{$\wedge$}(d,d'))\\
=&\widetilde R(d(\vec a+\vec\delta_\ell)\cup B(\vec b))+\mbox{$\wedge$}(d,d')
\end{align*}

Finally if $\iota(d')<\iota(d)$, then $m<p$ and so Lemma~\ref{lem:naturalmove-2} implies
\begin{align*}
\widetilde R(A(\vec a)\cup d'(\vec b)\cup\{\iota(d)\})=&\widetilde R(A(\vec a)\cup d'(\vec b)\cup\{C\}-A(\vec a)\cup B(\vec b)\cup C_p(\delta_p)\cup\{\iota(d')\})\\
&+\widetilde R(A(\vec a)\cup \dot B(\vec b)\cup C_p(\vec\delta_p))\\
=&\widetilde R(A(\vec a+\vec\delta_k)\cup d'(\vec b))+\widetilde R(A(\vec a)\cup B(\vec b)\mbox{$\wedge$}(d,d'))\\
=&\widetilde R(A(\vec a+\vec\delta_k)\cup d'(\vec b))+\mbox{$\wedge$}(d,d').
\end{align*}
And we have
$$\widetilde R(d(\vec a)\cup B(\vec b)\cup\{\iota(d')\})=\widetilde R(d(\vec a+\vec\delta_\ell)\cup B(\vec b)).$$

The remaining part can be proved by the same argument as in the proof of Lemma~\ref{lem:boundary-1}.
\end{proof}

To prove that for planar graphs the first homologies of graph braid groups are torsion free, we need an additional requirement. So we modify Lemma~\ref{lem:treeandorder} for planar graphs as follows.

\begin{lem}{\rm[Maximal Tree and Order for Planar Graph]}\label{lem:treeandorderforplanar}
For a given planar graph $\Gamma$, there is a maximal tree and its planar embedding so that the induced order on vertices satisfies (T1), (T2), and (T3) in Lemma~\ref{lem:treeandorder} and additionally
\begin{itemize}
\item[(T4)] If $\tau(d')<\tau(d)$ and $g(\tau(d),\iota(d))=g(\tau(d),\iota(d'))$ then $\iota(d)<\iota(d')$.
\end{itemize}
\end{lem}
\begin{proof}
Since $\Gamma$ is suitably subdivided, each path between two vertices of valency $\ne 2$ passes through at least 2 edges.

\paragraph{(I) Choices of a base vertex 0 and a planar embedding}
We assign 0 to a vertex $v$ such that $v$ is of valency 1 in $\Gamma$ or $\Gamma-\{v\}$ is connected if there is no vertex of valency 1. Choose a planar embedding of $\Gamma$ such that the base vertex 0 lies in the outmost region. Let $T=\Gamma$. Go to Step II.
\paragraph{(II) Choice of deleted edges}
Take a regular neighborhood $R$ of $T$. As traveling the outmost component of $\partial R$ clockwise from the base vertex until either coming back to 0 or meeting an edge that is on a circuit. If the former is the case, we are done. If the latter is the case, delete the edge and let $T$ be the rest. Repeat Step II.

\begin{figure}[ht]
\psfrag{0}{\small0}
\psfrag{d1}{\small$d$}
\psfrag{d2}{\small$d'$}
\centering
\subfigure[]{\includegraphics[height=1.5cm]{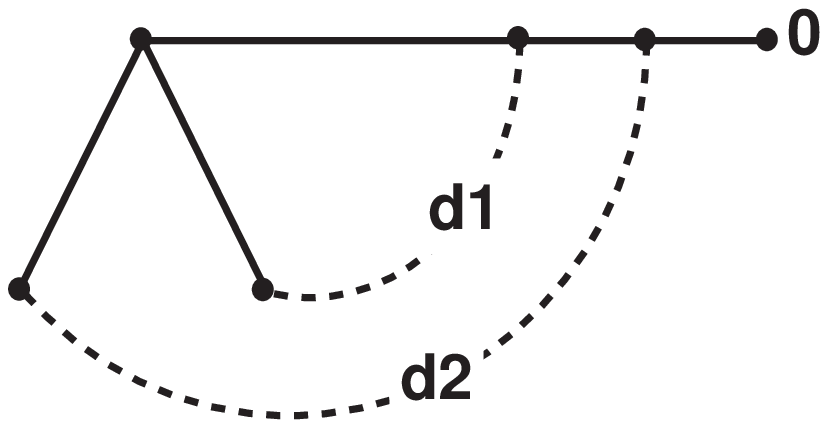}}\qquad
\subfigure[]{\includegraphics[height=2cm]{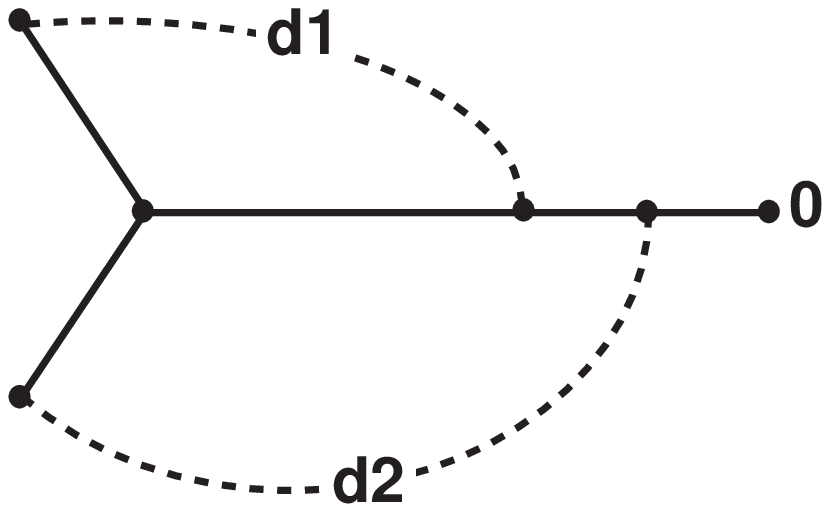}}
\caption{Relative locations of two deleted edges}\label{fig:dd}
\end{figure}

Then the order on vertices obtained by traveling a regular neighborhood $R$ of the maximal tree $T$ clockwise from 0 satisfies Conditions (T1) and (T2) since the terminal vertices of all deleted edges are vertices of valency $\ge 3$ in $\Gamma$. Moreover if $\tau(d')<\tau(d)$ and $g(\tau(d),\iota(d))=g(\tau(d),\iota(d'))$ for two deleted edges $d$ and $d'$ then there are two possibilities as Figure~\ref{fig:dd} since there is no intersection of the two edges. But by Step II the second possibility in Figure (b) is impossible. So the order satisfies Condition (T4). Concerning Condition (T3), we modify the planar embedding as in Lemma~\ref{lem:treeandorder}.
\end{proof}

\begin{exa}\label{ex:K4n3}
A maximal tree and an order on $K_4$ for $n=3$, which satisfy Lemma~\ref{lem:treeandorderforplanar}.
\end{exa}
\begin{figure}[ht]
\psfrag{0}{\small0}
\psfrag{2}{\small2}
\psfrag{4}{\small4}
\psfrag{6}{\small6}
\psfrag{9}{\small9}
\centering
\includegraphics[height=3cm]{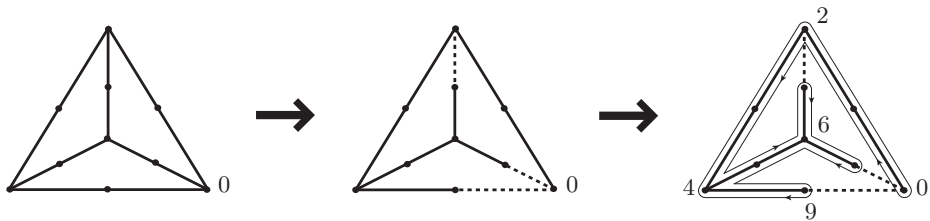}
\caption{The maximal tree and the order on $K_4$}
\label{fig:K4PE}
\end{figure}\qed

Condition (T4) implies that there are no critical 2-cells whose boundary images correspond to the case $\varepsilon=1$ in Lemma~\ref{lem:boundary-2}. Note that Condition (T4) implies that the given graph is planar. Thus a given graph has a maximal tree and an order on vertices satisfy (T1)--(T4) if and only if the graph is planar.

\section{First homologies}\label{s:three}

We will derive formulae for $H_1(B_n\Gamma)$ and $H_1(P_2\Gamma)$ in terms of graph-theoretical quantities. We will characterize presentation matrices for $H_1(B_n\Gamma)$ over bases given by critical 2-cells and critical 1-cells in \S\ref{ss31:PMatrix} and will count the number of relevant critical 1-cells in terms of graph-theoretical quantities in \S\ref{ss32:H1}. A parallel discussion for $H_1(P_2\Gamma)$ will be presented in \S\ref{ss33:H1PB2}.

\subsection{Presentation matrices}\label{ss31:PMatrix}

A presentation matrix of $H_1(B_n\Gamma)$ is determined by the second boundary homomorphism
over bases given by critical 2-cells and critical 1-cells. We will give orders on critical 1-cells and critical 2-cells to easily locate pivots and zero rows in the presentation matrixes.

The number of critical cells enormously grows in both the size of graph and the braid index. For example, consider $K_5$ with braid index 4 and its maximal tree and an order given in Example~\ref{ex:K5TO}. The numbers of critical 1-cells of the form $A_k(\vec a)$ and $d(\vec a)$ are 58 and 21. And the numbers of critical 2-cells of the form $A_k(\vec a)\cup B_\ell(\vec b)$, $A_k(\vec a)\cup d(\vec b)$ and $d(\vec a)\cup d'(\vec b)$ are 15, 167 and 56. So we have a presentation matrix of the size $238\times 79$. Fortunately rows of the matrix are highly dependent. The following lemmas illustrates some of this phenomena.

\begin{lem}\label{cor:boundary}{\rm [Dependence among Boundary Images I]}\\%
{\rm (1)} $\widetilde\partial (A_k(\vec a)\cup d'(\vec b))=\widetilde\partial (A_k(\vec a)\cup d')$\\
{\rm (2)} $\widetilde\partial (d(\vec a)\cup d'(\vec b))=\widetilde\partial (d(\vec a)\cup d')$ for $\tau(d)>\tau(d')$
\end{lem}
\begin{proof}
We can observe that the boundary images in Lemma~\ref{lem:boundary-1} and~\ref{lem:boundary-2} are independent of $\vec b$ and depend only on the initial vertex of the first edge whose terminal vertex is less than ends of the second edge.
\end{proof}

\begin{lem}\label{cor:boundary2}{\rm [Dependence among Boundary Images II]}\\%
{\rm (1)} If $A$ separates $d'$ and $d''$ and $g(A,\iota(d'))=g(A,\iota(d''))$, $$\widetilde\partial (A_k(\vec a)\cup d')=\widetilde\partial (A_k(\vec a)\cup d'').$$
{\rm (2)} If $\tau(d)$ separates $d'$ and $d''$ and $g(\tau(d),\iota(d))\ne g(\tau(d),\iota(d'))=g(\tau(d),\iota(d''))$,  $$\widetilde\partial (d(\vec a)\cup d')=\widetilde\partial (d(\vec a)\cup d'').$$
{\rm (3)} If $\tau(d)$ separates $d'$ and $d''$ and $g(\tau(d),\iota(d))= g(\tau(d),\iota(d'))=g(\tau(d),\iota(d''))$, $$\widetilde\partial (d(\vec a)\cup d'-d(\vec a)\cup d'')=\pm(\mbox{$\wedge$}(d,d')\pm\mbox{$\wedge$}(d,d'')).$$
\end{lem}
\begin{proof}
Immediate from Lemmas~\ref{lem:boundary-1} and \ref{lem:boundary-2}.
\end{proof}

Using the lemmas, we can reduce the size of the presentation matrix of $H_1(B_4K_5)$ to $91\times79$ by ignoring zero rows. We will see that the number of rows is still large comparing to the number of pivots. In order to find pivots systematically, we need to order critical cells.


Define the size $s(c)$ of a critical 1-cell $c$ to be the number of vertices blocked by the edge in $c$, more precisely, define $s(c)=|\vec a|$ for $c=A_k(\vec a)$ or $c=d(\vec a)$. Define the size $s(c)$ of a critical 2-cell $c$ to be the number of vertices blocked by the edge in $c$ that has the larger terminal vertex.

We assume that a set of $m$-tuples is always lexicographically ordered in the discussion below.
For edges $e, e'$, Declare $e>e'$ if $e$ is a deleted edge and $e'$ is an edge on $T$ or if both are either deleted edges or edges on $T$ and $(\tau(e),\iota(e))>(\tau(e'),\iota(e'))$. The set of critical 1-cells $c$ is linearly ordered by triples $(s(c),e,\vec a)$ where $c$ is given by either $A_k(\vec a)$ or $d(\vec a)$. The following lemma motivates this order.

\begin{lem}\label{lem:but}{\rm [Leading Coefficient]}
Let $c$ be a critical 2-cell containing two edge $e$ and $e'$ such that $\tau(e)>\tau(e')$. Assume that $\vec a$ represent vertices blocked by $\tau(e)$ in $c$. If $\widetilde\partial(c)\ne 0$ then the largest summand in $\widetilde\partial(c)$ has the triple $(s(c)+1,e,\vec a+\vec\delta_{g(\tau(e),\iota(e'))})$. Furthermore, if $e$ is a deleted edge $d$ then the largest summand is $- d(\vec a+\vec\delta_{g(\tau(e),\iota(e'))})$ and if $e$ is on $T$, then the largest summand is $- A_k(\vec a+\vec\delta_{g(\tau(e),\iota(e'))})$ where $A=\tau(e)$ and $k=g(A,\iota(e))$.
\end{lem}

\begin{proof}
By Lemmas ~\ref{lem:boundary-1} and \ref{lem:boundary-2} we see that $\widetilde\partial(c)$ is determined by $e$, $\vec a$ and $\tau(e')$. Using the order on critical 1-cells, it is easy to verify the lemma.
\end{proof}

In the view of this lemma, it is natural to order critical 2-cells as follows. For a critical 2-cell $c$, let $e$ and $e'$ denote edges in $c$ such that $\tau(e)>\tau(e')$ and $\vec a$ and $\vec a'$ represent vertices blocked by $e$ and $e'$, respectively.
The set of critical 2-cells $c$ is linearly ordered by 6-tuples $$(s(c),e,\vec a+\vec\delta_{g(\tau(e),\iota(e'))},g(\tau(e),\iota(e')),e',\vec a').$$
Then the first three terms determine the largest summand in $\widetilde\partial(c)$. The fourth term helps to find the boundary image of $c$ other than a summand of the form $\mbox{$\wedge$}(d,d')$  and the last two terms are added to make the order linear.

Lemma~\ref{lem:but} implies that the second boundary homomorphism $\widetilde\partial$ is represented by a block-upper-triangular matrix over bases of critical 2-cells and critical 1-cells ordered reversely. In fact, the presentation matrix is divided into blocks by $s(c)$ and each block is further divided into smaller blocks by the value $e$ of 6-tuples. The first column of each diagonal block is a vector of $- 1$. The $- 1$ entry at the lower left corner of each diagonal block will be called a {\em pivot} and  a critical 2-cell corresponding to a pivotal row is said to be {\em pivotal}. In other word, a pivotal 2-cell is the smallest one among all critical 2-cells that have the same (up to sign) largest summand in their boundary images. The following lemma says that non-pivotal rows turn into a zero row with few exceptions under row operations.

\begin{lem}\label{lem:rows}{\rm [Non-Pivotal Rows]}
Let $c$ be a non-pivotal critical 2-cell such that $\widetilde\partial(c)\ne0$. If $s(c)\ge1$, then the row corresponding to $c$ is a linear combination of rows below. If $s(c)=0$, then the row corresponding to $c$ is either a linear combination of rows below or made into a row consisting of only two nonzero entries that are $\pm 1$ by row operations.
\end{lem}
\begin{proof}
Assume $e'$ is a deleted $d'$ separated by $\tau(e)$ since $\widetilde\partial(c)=0$ otherwise. We may also assume that $c$ is the smallest among all critical 2-cells whose boundary images equal to $\widetilde\partial(c)$. Then by Lemma~\ref{cor:boundary} and
Lemma~\ref{cor:boundary2}, the 6-tuple for $c$ is given by
$$(s(c),e,\vec a+\vec\delta_{g(\tau(e),\iota(d'))},g(\tau(e),\iota(d')),d',0)$$
so that there is no smaller deleted edge $d''$ separated by $\tau(e)$ satisfying $g(\tau(e),\iota(d'))=g(\tau(e),\iota(d''))$. Set $k=g(\tau(e),\iota(e))$ and $\ell=g(\tau(e),\iota(d'))$.

There are three possibilities: (I) $s(c)\ge 1$ and $c=d(\vec a)\cup d'$, (II) $s(c)\ge 1$ and $c=A_k(\vec a)\cup d'$ and (III) $s(c)=0$ and $c=d\cup d'$.

\paragraph{(I) Assume $s(c)\ge 1$ and $c=d(\vec a)\cup d'$}
Set $A=\tau(d)$. We consider the following two cases separately:
\begin{itemize}
\item[(a)] There is a deleted edge $d''$ separated by $A$ such that $a_m\ne0$ and $m<\ell$ for $m=g(A,\iota(d''))$ ;
\item[(b)] There is no such a deleted edge.
\end{itemize}
For Case (a), we consider the following boundary image of a linear combination:
$$\widetilde\partial(d(\vec a)\cup d'-d(\vec a+\vec\delta_\ell-\vec\delta_m)\cup d'' - d(\vec a -\vec\delta_m)\cup d' +d(\vec a -\vec\delta_m)\cup d'')=$$ $$\mathbf{A}(\vec a +\vec\delta_\ell-\vec\delta_m,m)-\mathbf{A}(\vec a-\vec\delta_m,m) -\{\mathbf{A}(\vec a +\vec\delta_\ell-\vec\delta_m+ \vec\delta_k,m)-\mathbf{A}(\vec a-\vec\delta_m+\vec\delta_k,m)\}$$
The three term other than $c$ in the left side of the equation are critical 2-cells less than $c$. So it is sufficient to show that the right side, that will be denoted by $\mathbf{R}$, is a linear combination of boundary images of critical 2-cells less than $c$. The sum $\mathbf{R}$ depends on the order among $k$, $\ell$ and $m$. If $m\ge k$ then $\mathbf{A}(\vec a +\vec\delta_\ell-\vec\delta_m,m)=\mathbf{A}(\vec a +\vec\delta_\ell-\vec\delta_m+ \vec\delta_k,m)$ and $\mathbf{A}(\vec a-\vec\delta_m,m)=\mathbf{A}(\vec a-\vec\delta_m+\vec\delta_k,m)\}$ by Proposition~\ref{pro:A1} and so $\mathbf{R}=0$.

Since $m<\ell$, Proposition~\ref{pro:A1} and Lemma~\ref{lem:boundary-1} implies that for any $\vec x$
\begin{align*}
\widetilde\partial\circ R(\sum_{\alpha=0}^{|\vec x|_\ell}A_{p(\vec x-\alpha)}&((\vec x-\alpha)-\vec\delta_{p(\vec x-\alpha)}+\vec\delta_m)\cup d')\\
&=\mathbf{A}(\vec x,m)-\mathbf{A}(\vec x+\vec\delta_\ell,m)+A_\ell((\vec x -|\vec x|_\ell)+\vec\delta_m)
\end{align*}
To shorten formulae, let $\vec b=\vec a-\vec\delta_m+\vec\delta_k$ and $\vec c=\vec a-\vec\delta_m$.
If $m<k$ then
\begin{align*}
\mathbf{R}=&\:\:\widetilde\partial\circ R (\sum_{\alpha=0}^{|\vec b|_\ell}A_{p(\vec b-\alpha)}((\vec b-\alpha)-\vec\delta_{p(\vec b-\alpha)}+\vec\delta_m)\cup d')\\
&\:\:-\widetilde\partial\circ R (\sum_{\alpha=0}^{|\vec c|_\ell}A_{p(\vec c-\alpha)}((\vec b-\alpha)-\vec\delta_{p(\vec b-\alpha)}+\vec\delta_m)\cup d')\\
&-A_\ell((\vec b -|\vec b|_\ell)+\vec\delta_m)+A_\ell((\vec c -|\vec c|_\ell)+\vec\delta_m)
\end{align*}
If $m<k<\ell$, $\vec b-|\vec b|_\ell=\vec c -|\vec c|_\ell$ by Lemma~\ref{lem:boundary-1}. If $m<\ell\le k$,
$$\widetilde\partial (A_\ell((\vec c-|\vec c|_\ell)+\vec\delta_m)\cup d''')=A_\ell((\vec b -|\vec b|_\ell)+\vec\delta_m)-A_\ell((\vec c -|\vec c|_\ell)+\vec\delta_m)$$
since there is a deleted edge $d'''$ separated by $A$ such that $k=g(A,\iota(d'''))$ by (T3) of Lemma~\ref{lem:treeandorder}.

In Case (b), by the assumption there is no deleted edge $d''$ separated by $\tau(d)$ such that $g(A,\iota(d''))=m<\ell$ and $x_m\ne0$ for $\vec x=\vec a+\vec\delta_\ell$ and so there is no critical 2-cell with the 6-tuple $(s(c),d,\vec x,m,d'',0)$ such that $m<\ell$ and $A$ separates $d''$. If $k\ne\ell$, $c$ would be pivotal by the assumption on $c$. So $k=\ell$. By Lemma~\ref{cor:boundary2}(3), $\widetilde\partial(d(\vec a)\cup d'-d(\vec a)\cup d''')=\widetilde\partial(d\cup d'-d\cup d''')$ where $d'''$ is the smallest deleted edge such that $A$ separates $d'''$ and $g(A,\iota(d'''))=\ell$. Note that $|\vec a|\ge1$ since $s(c)\ge1$. And $d'''<d'$ since $c$ is pivotal. Thus we have a desired linear combination.

\paragraph{(II) Assume $s(c)\ge 1$ and $c=A_k(\vec a)\cup d'$}
Consider the following cases separately:
\begin{itemize}
\item[(a)] There is a deleted edge $d''$ separated by $A$ such that $g(A,\iota(d''))=m<\ell$ and one of the following conditions holds:\\
    (i) $a_m\ge1$ if $k\le m$, (ii) $a_m\ge1$ and $|\vec a|_k\ge2$ if $m<k\le\ell$, (iii) $a_m\ge1$ and $|\vec a|_k\ge2$ if $m<\ell<k$, and (iv) $a_m\ge1$ and $|\vec a|_k=1$ if $m<\ell<k$;
\item[(b)] There is no such a deleted edge.
\end{itemize}
For Cases (a)(i)-(iii), we consider the following boundary image of the linear combination:
$$\widetilde\partial(A_k(\vec a)\cup d'-A_k(\vec a +\vec\delta_\ell-\vec\delta_m)\cup d'' - A_k(\vec a -\vec\delta_m)\cup d' +A_k(\vec a - \vec\delta_m)\cup d'')=$$ $$\mathbf{A}(\vec a +\vec\delta_\ell-\vec\delta_m,m)-\mathbf{A}(\vec a-\vec\delta_m,m) -\{\mathbf{A}(\vec a +\vec\delta_\ell-\vec\delta_m +\vec\delta_k,m)-\mathbf{A}(\vec a-\vec\delta_m +\vec\delta_k,m)\}$$
The three terms other than $c$ in the left side of the equation are critical 2-cells less than $c$. Then it is sufficient to show that the right side is a linear combination of boundary images of critical 2-cells less than $c$. We omit the proof since it is similar to Case (I)(a).

For Case (a)(iv), we consider the following boundary image of the linear combination:
$$\widetilde\partial(A_k(\vec a)\cup d' -A_k(\vec a+\vec\delta_\ell-\vec\delta_m)\cup d'' -A_\ell(\vec a+\vec\delta_\ell-\vec\delta_k)\cup d''')=0$$ where $d'''$ is a deleted edge separated by $A$ and $g(A,\iota(d'''))=k$. Note that the existence of $d'''$ is guaranteed by Condition (T3) of Lemma~\ref{lem:treeandorder}.

We will show that Case (b) does not occur. Suppose that there is no deleted edge $d''$ separated by $A$ such that $g(A,\iota(d''))=m<\ell$ and $A_k(\vec x-\vec\delta_m)$ is critical for $\vec x=\vec a+\vec\delta_\ell$. So there is no critical 2-cell with the 6-tuple $(s(c),A_k(\vec\delta_k),\vec x,m,d'',0)$ such that $m<\ell$ and $A$ separates $d''$. Then $c$ would be pivotal since $d'$ is the the smallest among deleted edges $d''$ separated by $A$ such that $g(A,\iota(d''))=\ell$.

\paragraph{(III) Assume $s(c)=0$ and $c=d\cup d'$}
Let $k=g(\tau(d),\iota(d))$. Since $c$ is non-pivotal, $k=\ell$.  By Lemma~\ref{cor:boundary2}(3), $\widetilde\partial(d\cup d'-d\cup d'')=\pm(\mbox{$\wedge$}(d,d')\pm\mbox{$\wedge$}(d,d''))$ where $d''$ is the smallest deleted edge separated by $A$ such that $g(A,\iota(d''))=\ell$. Note that if $d'=d''$ then $c$ would be pivotal. This completes the proof.
\end{proof}

We are ready to see the main theorem of this section.

\begin{thm}\label{thm:matrix}
Let $M$ be a presentation matrix of $H_1(B_n\Gamma)$ represented by $\tilde\partial$ over bases of critical 2-cells and 1-cells ordered reversely.
Up to row operations, each row of $M$ satisfies one of the followings:
\begin{enumerate}
\item[\rm{(1)}] consists of all zeros;
\item[\rm{(2)}] there is a $\pm 1$ entry that is the only nonzero entry in the column it belongs to;
\item[\rm{(3)}] there are only two nonzero entries which are $\pm1$.
\end{enumerate}
If $\Gamma$ is planar then two nonzero entries in {\rm(3)} have opposite signs. Furthermore, the number of rows satisfying {\rm(3)} does not depend on braid indices.
\end{thm}

\begin{proof} A pivotal row satisfies (2) by killing all entries above the pivot via row operations. A row of the type (3) is produced from the relation $\widetilde\partial(d\cup d''-d\cup d')=\pm(\mbox{$\wedge$}(d,d'')\pm\mbox{$\wedge$}(d,d'))$ in the last part of the proof of the previous lemma. Obviously the number of these relations does not depend on braid indices. If $\Gamma$ is planar, the relation becomes $\widetilde\partial(d\cup d''-d\cup d')=\pm(\mbox{$\wedge$}(d,d'')-\mbox{$\wedge$}(d,d'))$
by Lemma~\ref{lem:boundary-2} and Lemma~\ref{lem:treeandorderforplanar}. Therefore two nonzero entries in {\rm(3)} have opposite signs.
\end{proof}

Further row operations among rows of the type (3) in the theorem may produce new pivots $\pm 2$ but if two nonzero entries have opposite signs, all of new pivots are $\pm 1$ and so we have the following corollary.

\begin{cor}\label{thm:torsionfree}
If $H_1(B_n\Gamma)$ has a torsion, it is a 2-torsion and the number of 2-torsions does not depend on braid indices. For a planar graph $\Gamma$,  $H_1(B_n\Gamma)$ is torsion-free.
\end{cor}

We classify critical 1-cells according to Theorem~\ref{thm:matrix}. A critical 1-cell is said to be
\begin{itemize}
\item[(i)] \emph{pivotal} if it corresponds to pivotal columns, which is related to (2);
\item[(ii)] \emph{separating} if it corresponds to columns of nonzero entries of (3);
\item[(iii)] \emph{free} otherwise.
\end{itemize}

Clearly a pivotal 1-cell has no contribution to $H_1(B_n\Gamma)$ and a free 1-cell contribute a free summand to $H_1(B_n\Gamma)$. To complete the computation of $H_1(B_n\Gamma)$, it is enough to consider the submatrix obtained by deleting pivotal rows and zero rows and deleting pivotal columns and columns of free 1-cells. This submatrix will be referred as a {\em undetermined block} for $H_1(B_n\Gamma)$ and will be studied in \S\ref{ss32:H1}. Rows of an undetermined block are of the type (3) and columns corresponds to separating 1-cells. It will be useful later to have a geometric characterization of pivotal 1-cells.

\begin{lem}\label{cor:pivotal1cell}{\rm [Pivotal 1-Cell]}
A critical 1-cell $c$ is pivotal if and only if $c$ is either $A_k(\vec a)$ or $d(\vec a)$
such that there is a deleted edge $d'$ separated by $A$ or $\tau(d)$ and $a_m\ge1$ for $m=g(A,\iota(d'))$ and in addition $s(c)\ge2$ when $c=A_k(\vec a)$.
\end{lem}

\begin{proof}
By the definition of pivotal 1-cell and Lemma~\ref{lem:but}, $c$ is a pivotal 1-cell iff there is a critical 2-cell whose boundary image has the largest summand $c$ iff $s(c)\ge2$ for $A_k(\vec a)$ ($s(c)\ge1$ for $d(\vec a)$, respectively) and there is a deleted edge $d'$ separated by $A$ such that the 1-cell $A_k(\vec a-\vec\delta_m)$ ($d(\vec a-\vec\delta_m)$, respectively) exits and is critical for $m=g(A,\iota(d'))$. A critical 1-cell $d(\vec a-\vec\delta_m)$ exits iff $a_m\ge1$. So we are done.

Assume that $c=A_k(\vec a)$. The ``only if'' part is now clear. To show the ``if'' part, consider $|\vec a|_{k}$ and $m$. If $|\vec a|_{k}\ge2$ or $|\vec a|_{k}=1$ and $m\ge k$, then $A_k(\vec a-\vec\delta_m)$ is a critical 1-cell and we are done. If $|\vec a|_{k}=1$ and $m\le k-1$, then $a_j\ge1$ for some $j\ge k$ since $s(c)\ge2$. By Condition (T3) in Lemma~\ref{lem:treeandorder}, there is a deleted edge $d''$ separated by $A$ such that $g(A,\iota(d''))=j$. Then the largest summand of $\widetilde\partial(A_k(\vec a-\vec\delta_j)\cup d'')$ is $c$ and so $c$ is pivotal.
\end{proof}

We can also have a geometric characterization for a separating 1-cells which is clear from the definition of separating 1-cells and Lemma~\ref{cor:boundary2}(3).

\begin{lem}\label{cor:separating1cell}{\rm [Separating 1-Cell]}
A critical 1-cell $c$ is separating if and only if  there are three deleted edges such that $c$ is a summand of $\widetilde\partial(d\cup d'-d\cup d'')$ such that $\tau(d)>\tau(d')$, $\tau(d)>\tau(d'')$ and $g(\tau(d),\iota(d))=g(\tau(d),\iota(d'))=g(\tau(d),\iota(d''))$. In fact, $c$ is of the form $A_k(\vec\delta_m)$ such that $c=\mbox{$\wedge$}(d,d')$ (or $\mbox{$\wedge$}(d,d'')$, respectively) and deleted edges $d$ and $d'$ (or $d''$) are separated by $A$.
\end{lem}

It is now easy to recognize free 1-cells. So we can compute $H_1(B_n\Gamma)$ by using the undetermined block after counting the number of free 1-cells.

\begin{exa}\label{ex:H1B4K5}
Suppose a maximal tree and an order is given as Example~\ref{ex:K5TO} for the complete graph $K_5$. We want to compute $H_1(B_4K_5)$ which will be needed later.
\end{exa}
Recall the maximal tree and the order on vertices as Figure~\ref{fig:ExK5n4}.
\begin{figure}[ht]
\psfrag{0}{\small0}
\psfrag{6}{\small$A$}
\psfrag{11}{\small$B$}
\psfrag{18}{\small$C$}
\psfrag{d1}{\small$d_1$}
\psfrag{d2}{\small$d_2$}
\psfrag{d3}{\small$d_3$}
\psfrag{d4}{\small$d_4$}
\psfrag{d5}{\small$d_5$}
\psfrag{d6}{\small$d_6$}
\centering
\includegraphics[height=3cm]{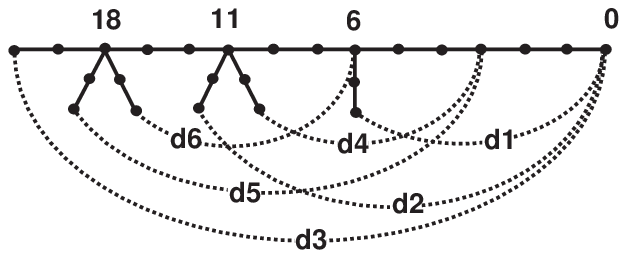}
\caption{The maximal tree and the order on $K_5$}
\label{fig:ExK5n4}
\end{figure}
By Lemma~\ref{cor:pivotal1cell}, all critical 1-cells but of the forms $d$ or $A_k(\vec a)$ with $|\vec a|=1$ are pivotal. All of critical 1-cells of the form $A_k(\vec a)$ with $|\vec a|=1$ are separating by Lemma~\ref{cor:separating1cell}. Thus the number of free 1-cells is 6 that equals $\beta_1(\Gamma)$. Critical 2-cells of the form $d\cup d'-d\cup d''$ give separating 1-cells by Lemma~\ref{cor:boundary2}(3).
Over the basis $\{d_6\cup d_5-d_6\cup d_2,\ d_6\cup d_4-d_6\cup d_2,\ d_6\cup d_3-d_6\cup d_2,\ d_5\cup d_3-d_5\cup d_1,\ d_5\cup d_2-d_5\cup d_1,\ d_4\cup d_3-d_4\cup d_1,\ d_4\cup d_2-d_4\cup d_1\}$ of critical 2-cells and the basis $\{C_3(1,0,0),\ C_3(0,1,0),\ C_2(1,0,0),\ B_3(1,0,0),\ B_3(0,1,0),\ B_2(1,0,0),\ A_2(1,0)\}$ of separating 1-cells, we have the undetermined block for $H_1(B_4\Gamma)$ as follows:
$${\left(
\begin{array}{ccccccc}
&&-1&&-1&&\\
&&&1&-1&&\\
-1&&&&-1&&\\
&-1&&&&&-1\\
&&&&1&&-1\\
&&&-1&&&-1\\
&&&&&-1&-1
\end{array} \right)
\rightarrow
\left(
\begin{array}{ccccccc}
-1&&&&-1&&\\
&-1&&&&&-1\\
&&-1&&-1&&\\
&&&1&-1&&\\
&&&&1&&-1\\
&&&&&-1&-1\\
&&&&&&\framebox[0.5cm][l]{-2}
\end{array} \right)}$$

After putting the undetermined block into a row echelon form, we see that all separating 1-cells but $A_2(1,0)$ are null homologous and $A_2(1,0)$ represents a 2-torsion homology class. Thus $H_1(B_4K_5)\cong\mathbb Z^6\oplus \mathbb Z_2$ and the free part is generated by $[d_i]$ for $i=1,\cdots,6$.
\qed

\subsection{First homologies of graph braid groups}\label{ss32:H1}
In this section we will discuss how to compute the first integral homology of a graph braid group in terms of graph-theoretic invariants. Our strategy is to decompose a given graph into simpler graphs and to compute the contribution from simpler pieces and from the cost of decomposition. The following example illustrates this strategy.

\begin{exa}\label{ex:B3n3}
Let $\Gamma$ be a graph with a maximal tree given in Figure~\ref{fig:B3T}. We want to compute $H_1(B_3\Gamma)$.
\end{exa}
\begin{figure}[ht]
\psfrag{0}{\small0}
\psfrag{A}{\small$A$}
\psfrag{d1}{\small$d_1$}
\psfrag{d2}{\small$d_2$}
\psfrag{d3}{\small$d_3$}
\psfrag{d4}{\small$d_4$}
\centering
\includegraphics[height=3cm]{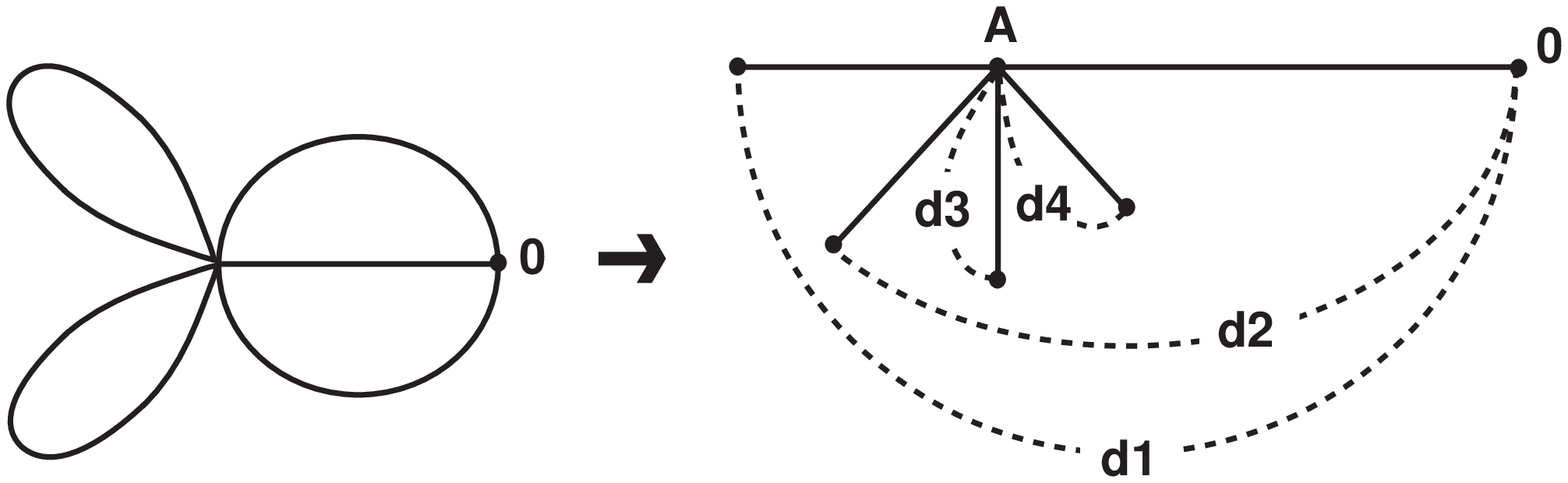}
\caption{$\Gamma$ and a maximal tree $T$}
\label{fig:B3T}
\end{figure}
Give an order on vertices obtained by traveling a regular neighborhood of the maximal tree $T$ clockwise from 0. There are no pairs of critical 2-cells that induce a row satisfying (3) in Theorem~\ref{thm:matrix}. So there are no separating 1-cells. Thus there is no torsion and the rank of $H_1(B_3\Gamma)$ is equal to the number of free 1-cells. There are 28 free 1-cells as follows:\\
$d_i$ for $i=1,2,3,4$; $d_i(\vec a)$ for $i=1,2$ and $\vec a=(1,0,0,0)$, $(0,1,0,0)$, $(2,0,0,0)$, $(1,1,0,0)$, $(0,2,0,0)$; $A_2(\vec a)$ for $\vec a=(1,0,0,0)$, $(2,0,0,0)$, $(1,1,0,0)$; $A_3(\vec a)$ for $\vec a=(1,0,0,0), (0,1,0,0), (2,0,0,0), (1,1,0,0), (0,2,0,0)$; and $A_4(\vec a)$ for $\vec a=(1,0,0,0)$, $(0,1,0,0)$, $(0,0,1,0)$, $(2,0,0,0)$, $(1,1,0,0)$, $(0,2,0,0)$.\\
Consequently, $H_1(B_3\Gamma)\cong\mathbb Z^{28}$.

The vertex $A$ decomposes $\Gamma$ to two circles and one $\Theta$-shape graph that are all subgraphs of the original. The first homologies of two circles are generated by $d_4(2,0,0,0)$ and $d_3(0,2,0,0)$. And the first homology of $\Theta$-shape graph is generated by $d_1$, $d_2$ and $A_4(0,0,1,0)$. The remaining free 1-cells lie over at least two distinct components and they are the cost of decomposition. So the first homology of $\Gamma$ can also be decomposed as
$$H_1(B_3\Gamma)=\langle d_4(2,0,0,0)\rangle\oplus\langle d_3(0,2,0,0)\rangle\oplus\langle d_1,d_2,A_4(0,0,1,0)\rangle\oplus\mathbb Z^{23}$$
\qed

In order to formalize this idea, we need some notions and facts from graph theory. A {\em cut} of a connected graph is a set of vertices whose removal separates at least a pair of vertices.  A graph is {\em $k$-vertex-connected} if the size of a smallest cut is $\ge k$. If a graph has no cut (for example, complete graphs) and the number $m$ of vertices is $\ge 2$ then the graph is defined to be (m-1)-vertex-connected. The graph of one vertex is defined to be 1-vertex-connected. ``2-vertex-connected" and ``3-vertex-connected" will be referred as {\em biconnected} and {\em triconnected}. Let $C$ ba a cut of $\Gamma$. A {\em $C$-component} is the closure of a connected component of $\Gamma-C$ in $\Gamma$ viewed as topological spaces. So a $C$-component is a subgraph of $\Gamma$.

Recall that we are assuming that every graph is suitably subdivided, finite, and connected. A suitably subdivided graph is always simple, i.e has neither multiple edges nor loops, and moreover it has no edge between vertices of valency $\ge 3$.
A cut is called a {\em $k$-cut} if it contains $k$ vertices. The set of 1-cuts of a graph $\Gamma$ is well-defined and we can decompose $\Gamma$ into components that are either biconnected or the complete graph $K_2$ by iteratively taking $C$-components for all 1-cut $C$. This decomposition is unique. The topological types of biconnected components of a given graph do not depend on subdivision. In fact, a subdivision merely affects the number of $K_2$ components.

Let $C$ be a 2-cut $\{x,y\}$ of a biconnected graph $\Gamma$. We find it convenient to modify each $C$-component by adding an extra edge between $x$ and $y$. We refer to this modified $C$-component as a {\em marked $C$-component}. If a marked $C$-component has a 2-cut $C'$, we take all marked $C'$-components of the marked $C$-component. By iterating this procedure, we can decompose a biconnected graph into components that are either triconnected or the complete graph $K_3$. This decomposition is unique for a biconnected suitably subdivided graph (for example, see \cite{CE}) and will be called a {\em marked decomposition}. The topological types of triconnected components of a given graph do not depend on subdivision. In fact, a subdivision merely affects the number of $K_3$ components.

A graph is said to have {\em topologically} a certain property if it has the property after ignoring vertices of valency 2.
We assume that each component in the above two decompositions is always suitably subdivided by subdividing it if necessary. Then triconnected components in the above decompositions are topologically triconnected. Note that a subdivision of a biconnected graph is again biconnected.

\begin{lem}{\rm [Decomposition of Connected Graph]} \label{lem:cutvertex}
Let $x$ be a 1-cut in a graph $\Gamma$. Then
$$H_1(B_n\Gamma)\cong (\oplus_{i=1}^\mu H_1(B_n\Gamma_{x,i}))\oplus \mathbb Z^{N(n,\Gamma,x)}$$
where $\Gamma_{x,i}$ are $x$-components of $\Gamma$,
$$N(n,\Gamma,x)=\left(\begin{array}{c} n+\mu-2 \\ n-1 \end{array}\right) \times(\nu-2) -\left(\begin{array}{c} n+\mu-2 \\ n \end{array}\right)-(\nu-\mu-1),$$
$\mu$ is the number of $x$-components of $\Gamma$, and $\nu$ is the valency of $x$ in $\Gamma$.
\end{lem}
\begin{proof}
Assume that $\Gamma$ has a maximal tree $T$ and an order on vertices as Lemma~\ref{lem:treeandorder}. Except the $x$-component containing the base vertex 0, each $x$-component $\Gamma_{x,i}$ has new base point $x$ and we maintain the numbering on vertices. Then $x$ is the smallest vertex on each $x$-component not containing the original base vertex 0. Unless $A=x$, every critical 1-cell of the type $A_k(\vec a)$ can be thought of as a critical 1-cell in one of $x$-components by regarding vertices blocked by 0 as vertices blocked by $x$. Similarly, unless $\iota(d)=x$ or $\tau(d)=x$, a deleted edge $d$ does not join distinct $x$-components and so a critical 1-cell of the type $d(\vec a)$ can be regarded as a critical 1-cell in one of $x$-components. Therefore a critical 1-cell in $UD_n\Gamma$ that belong to none of $x$-components must contain an edge incident to $x$.

We first claim that the undetermined block for $H_1(B_n\Gamma)$ is a block sum of the undetermined blocks for $H_1(B_n\Gamma_{x,i})$'s.
A row of an undetermined block is obtained by the boundary image of a critical 2-cell of the form $d\cup d'$ (see Lemma~\ref{cor:separating1cell}). If two deleted edges $d$ and $d'$ are in distinct $x$-component, the boundary image is trivial since the terminal vertex of one edge cannot separate the other. Thus both $d$ and $d'$ are in the same $x$-component and so each separating 1-cell for $UD_n\Gamma$ must be a separating 1-cell for exactly one of $x$-components.

The proof is completed by counting the number of free 1-cells that cannot be regarded as those in any one of $x$-components. Let $m$ be the valency of $x$ in the maximal tree. Then $\mu\le m$. Recall that branches incident to $x$ are numbered by $0, 1,\ldots,m-1$ clockwise starting from the 0-th branch pointing the base vertex 0.  The $i$-th and the $j$-th branches do not belong to the same $x$-component for $1\le i, j\le \mu-1$ by (T2) of Lemma~\ref{lem:treeandorder}. When $\mu\le m-1$, the $i$-th and the $0$-th branches belong to the same $x$-component for $\mu\le i\le m-1$ by Condition (T3) of Lemma~\ref{lem:treeandorder}. For $1\le i\le \mu$, let $\Gamma_{x,i}$ denote the $x$-component containing the $i$-branch. Then the $x$-component $\Gamma_{x,\mu}$ contains the $\mu$-th to the $(m-1)$-st branches and the 0-th branch.

Set $A=x$. If $1\le k\le\mu-1$ or $|\vec a|_\mu\ge1$ then $A_k(\vec a)$ cannot be a critical 1-cell over any one of $x$-components. We divide this situation into the following four cases:

\begin{itemize}
\item[(a)] $1\le k\le\mu-1$ and $|\vec a|=|\vec a|_\mu$
\item[(b)] $1\le k\le\mu-1$ and $|\vec a|>|\vec a|_\mu$
\item[(c)] $\mu\le k\le m-1$ and $|\vec a|=|\vec a|_\mu$
\item[(d)] $\mu\le k\le m-1$ and $|\vec a|>|\vec a|_\mu$
\end{itemize}
To use Lemma~\ref{cor:pivotal1cell}, consider a deleted edge $d'$ such that $g(A,\iota(d'))=i$. For $1\le i\le\mu-1$, $\tau(d')$ is in $\Gamma_{x,i}$ since $\iota(d')$ is in $\Gamma_{x,i}$. So $A$ cannot separate $d'$. Thus every critical 1-cell satisfying either (a) or (c) is free. On the other hand, for $\mu\le i\le m-1$, we may choose $d'$ such that $g(A,\tau(d'))=0$ since both the $i$-th and the $0$-th branches lie on $\Gamma_{x,\mu}$. So $A$ separates $d'$. Thus every critical 1-cell satisfying either (b) or (d) is pivotal. Note that in cases of (a) and (c), $|\vec a|_\mu\ge1$ since $A_k(\vec a)$ is critical.

There are $\nu-m$ deleted edges $d$ such that $\tau(d)=x$ and $\iota(d)$ lies on the $i$-th branch of $x$ for some $1\le i\le m-1$. Unless all $(n-1)$ vertices blocked by $\tau(d)$ lie on the $x$-component containing $\iota(d)$, $d(\vec a)$ cannot be a critical 1-cell over any one of $x$-components. If $|\vec a|>|\vec a|_\mu$ then a critical 1-cell $d(\vec a)$ is  pivotal. Otherwise it is free. This means that vertices in $\Gamma_{x,\mu}$ must lie on the 0-th branch in order to be free. Counting combinations with repetition, the numbers of free 1-cells for the three cases are given as follows:
\begin{align*}
\mbox{The number of } A_k(\vec a) \mbox{ in (a) } =& \left(\begin{array}{c} n+\mu-2\\ n-1 \end{array} \right) \times (\mu-2)- \left(\begin{array}{c} n+\mu-2\\ n \end{array} \right) +1\\
\mbox{The number of } A_k(\vec a) \mbox{ in (c) } =& \left(\begin{array}{c} n+\mu-2\\ n-1 \end{array} \right) \times (m-\mu)- (m-\mu)\\
\mbox{The number of } d(\vec a) =& \left(\begin{array}{c} n+\mu-2\\ n-1 \end{array} \right) \times (\nu-m)- (\nu-m).
\end{align*}
The sum is equal to $N(n,\Gamma,x)$ which is the number of free 1-cells that cannot be seen inside each $x$-component.
\end{proof}

The above lemma decomposes the first homology of a graph braid group into the first homologies of graph braid groups on biconnected components
together with a free part determined by the valency and the number of $x$-component of each 1-cut $x$. Since $N(n,\Gamma,x)=0$ for a 1-cut $x$ of valency 2 and $UD_n(\Gamma)$ is contractible if $\Gamma$ is topologically a line segment, this decomposition of $H_1(B_n\Gamma)$ is independent of subdivision. Farley obtained a similar decomposition in \cite{Far} when $\Gamma$ is a tree.

\begin{lem}\label{lem:biconnected}
For a biconnected graph $\Gamma$, $H_1(B_n\Gamma)\cong H_1(B_2\Gamma)$.
\end{lem}

\begin{proof}
A sequence of vertices starting from the base vertex in a critical cell can be ignored to give a corresponding critical cell for a lower braid index. So a critical 1-cell with $s(c)\le1$ in $UD_n\Gamma$ can be regarded as a critical 1-cell in $UD_2\Gamma$.
An undetermined block involves only critical 2-cells with $s(c)=0$ and critical 1-cells with $s(c)=1$ and so it is well-defined independently of braid indices $\ge 2$.

It is now sufficient to show that every critical 1-cell $c$ with $s(c)\ge2$ is pivotal.
To show that a critical 1-cell $A_k(\vec a)$ with $|\vec a|\ge 2$ is pivotal, we need to find a deleted edge satisfying Lemma~\ref{cor:pivotal1cell}. Suppose there is no deleted edge $d'$ such that $A$ separate $d'$ and $g(A,\iota(d'))=g(A,v)$ for the second smallest vertex $v$ blocked by $A$. By Lemma~\ref{lem:treeandorder} (T2), $\tau(d')<A$. This means that the vertex $A$ disconnects the $g(A,v)$-th branch of $A$ from the rest of $\Gamma$. This contradicts the biconnectivity of $\Gamma$.

For a critical 1-cell  $d(\vec a)$ with $|\vec a|\ge2$, let $v$ be the smallest vertex blocked by $\tau(d)$. Then we can argue similarly to show $d(\vec a)$ is pivotal.
\end{proof}

For the sake of the previous lemma, it is enough to consider 2-braid groups for biconnected graphs in order to compute $n$-braid groups.

\begin{lem}\label{lem:decom}
Let $\{x,y\}$ be a 2-cut in a biconnected graph $\Gamma$, $\Gamma'$ be a $\{x,y\}$-component of $\Gamma$, $\widehat\Gamma'$ be the marked $\{x,y\}$-component of $\Gamma'$, $\Gamma''$ be the complementary subgraph, i.e. $\Gamma''$ be the closure of $\Gamma-\Gamma'$ in $\Gamma$, and  $\widehat\Gamma''$ be obtained from $\Gamma''$ by adding an extra edge between $x$ and $y$. Then
$$H_1(B_2\Gamma)\oplus\mathbb Z\cong H_1(B_2\widehat\Gamma')\oplus H_1(B_2\widehat\Gamma'')$$
\end{lem}

\begin{proof}
If either $\widehat\Gamma'$ or $\widehat\Gamma''$ is a topological circle, this lemma is a tautology since $H_1(B_2S^1)\cong\mathbb Z$. So we assume that $\widehat\Gamma'$ and $\widehat\Gamma''$ are not a topological circle. For a biconnected graph, we may regard $x$ as the base vertex 0 and choose a maximal tree $T$ of $\Gamma$ that contains a path between $0$ and $y$ through $\Gamma'$. Choose a planar embedding of $T$ as given in Figure~\ref{fig:2Cde}(a) by using Lemma~\ref{lem:treeandorder} and number vertices of $\Gamma$. Then maximal trees of $\widehat\Gamma'$ and $\widehat\Gamma''$ and their planar embeddings are induced as Figure~\ref{fig:2Cde}(b)(c) where $d_0$ is the new deleted edge on the (subdivided) edge added between $0$ and $y$ and $d_i$'s for $i\ge1$ are deleted edges incident to 0 in $\Gamma$ and $\widehat\Gamma''$. We maintain the numbering on vertices of $\widehat\Gamma'$ and $\widehat\Gamma''$ so that all vertices of valency 2 on the added edge that is subdivided is larger than any vertex in $\widehat\Gamma'$ and $y$ is the second smallest vertex of valency$\ge3$ in $\widehat\Gamma''$. Let $\nu$ and $\nu'$ be valencies of $y$ in maximal trees of $\Gamma$ and $\Gamma'$, respectively. Then $\nu-\nu'+1$ is in fact the number of $\{0,y\}$-components by Lemma~\ref{lem:treeandorder}.
\begin{figure}[ht]
\psfrag{0}{\small0}
\psfrag{3}{\small$y$}
\psfrag{d1}{\small$\Gamma'-\{x,y\}$}
\psfrag{d2}{\small$\Gamma''-\{x,y\}$}
\psfrag{d3}{\small$d_i$}
\psfrag{d4}{\small$d_0$}
\centering
\subfigure[A maximal tree of $\Gamma$]{\includegraphics[height=1.5cm]{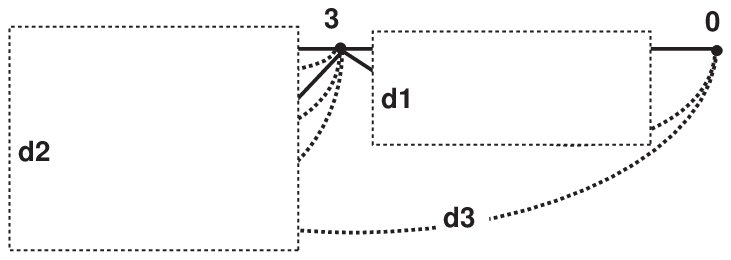}}\qquad
\subfigure[$\widehat\Gamma'$]{\includegraphics[height=1.5cm]{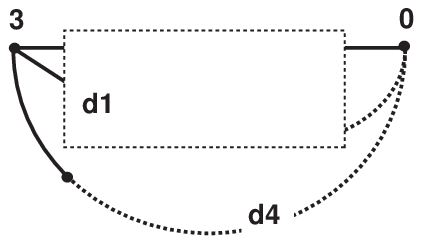}}\qquad
\subfigure[$\widehat\Gamma''$]{\includegraphics[height=1.5cm]{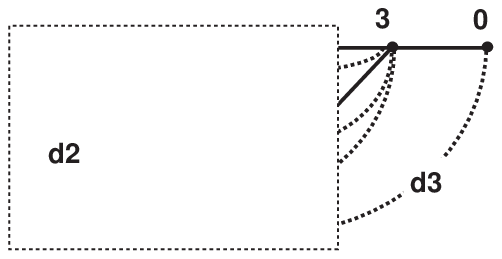}}
\caption{A decomposition of $\Gamma$}\label{fig:2Cde}
\end{figure}

There is a natural graph embedding $f':\widehat \Gamma'\to \Gamma$ by sending the extra edge to a path from $y$ to 0 via the $\nu'$-th branch of $y$ after suitable subdivision. Then the delete edge $d_0$ is sent to one of $d_i$'s. Also there is a natural graph embedding $f'':\widehat \Gamma''\to \Gamma$ by sending the extra edge to the path from 0 to $y$ in the maximal tree of $\Gamma$ after subdivision. Both $f'$ and $f"$ are order-preserving. It is easy to see that $f''$ induces a bijection between critical 1-cells of $UD_2\widehat\Gamma''$ and those of $UD_2\Gamma$ and it preserves the types of critical 1-cells: pivotal, free or separating. Thus the induced homomorphism $f_*'':H_1(B_2\widehat\Gamma'')\to H_1(B_2\Gamma)$ is injective.
Every critical 2-cell in $UD_2\Gamma$ is of the form $d\cup d'$.  If a critical 2-cell $d\cup d'$ is in neither $UD_2(f'(\Gamma'))$ nor $UD_2(f''(\Gamma''))$ then both deleted edges are not simultaneously in the same image under $f'$ or $f''$ and so $\widetilde\partial(d\cup d')=0$ by Lemma~\ref{lem:boundary-2}. Thus the induced homomorphisms $f_*':H_1(B_2\widehat\Gamma')\to H_1(B_2\Gamma)$ and $f_*'':H_1(B_2\widehat\Gamma'')\to H_1(B_2\Gamma)$ are injective. Moreover it is clear that $\mbox{im}(f_*')\cap\mbox{im}(f_*'')$ is isomorphic to $\mathbb Z$ generated by $f_*'([d_0])$.

We are done if we show  $\mbox{im}(f_*')+\mbox{im}(f_*'')=H_1(B_2\Gamma)$. Set $A=y$. There are the following two types of 1-cells in $UD_2\Gamma$ that are neither in $UD_2(f'(\Gamma'))$ nor in $UD_2(f''(\Gamma''))$: $d(\vec\delta_m)$ for $\tau(d)=A$ and $1\le m<\nu'\le g(A,\iota(d))\le\nu-1$ or $A_k(\vec\delta_m)$ for $1\le m<\nu'\le k\le\nu-1$.
Since $\Gamma'$ is a $\{x,y\}$-component, for each $m$-th branch of $A$ such that $1\le m<\nu'$ there is a deleted edge $d'$ separated by $A$ satisfying $\tau(d')>0$ and $g(A,\iota(d'))=m$ and so $d(\vec\delta_m)$ are pivotal and so it vanishes in $H_1(B_2\Gamma)$.

Since $\{x,y\}$ is a 2-cut, for each $k$-th branch of $A$ such that $\nu'\le k\le\nu-1$ there is a deleted edge $d_i$ such that $g(A,\iota(d_i))=k$ and $\tau(d_i)=0$. Since $g(\tau(d'),\iota(d_i))=g(\tau(d'),A)=g(\tau(d'),\iota(f'(d_0)))$ for the deleted edge $d'$ found above,
$$\widetilde\partial(d'\cup d_i-d'\cup f'(d_0))=\pm(\mbox{$\wedge$}(d',d_i)\pm\mbox{$\wedge$}(d',f'(d_0)))=\pm(A_k(\vec\delta_m)\pm A_{\nu'}(\vec\delta_m))$$
by Lemma~\ref{cor:boundary2}(3).
Thus $A_k(\vec\delta_m)$ and $A_{\nu'}(\vec\delta_m)$ are homologous up to signs and $A_{\nu'}(\vec\delta_m)$ is a critical 1-cell in $UD_2(f'(\Gamma'))$.
\end{proof}

Let $\Theta_m$ be the graph consisting two vertices and $m$ edges between them. For example, $\Theta_3$ is the letter shape of $\Theta$.

\begin{lem}{\rm [Decomposition of Biconnected Graph]}\label{lem:bidecomp}
Let $\{x,y\}$ be a 2-cut in a biconnected graph $\Gamma$, and $\Gamma_1,\ldots,\Gamma_m$ denote $\{x,y\}$-components. Then $$H_1(B_2\Gamma)\oplus\mathbb Z\cong \oplus_{i=1}^{m} H_1(B_2\Gamma_i)\oplus\mathbb Z^{(m-1)(m-2)/2}.$$
\end{lem}

\begin{proof}
By repeated application of Lemma~\ref{lem:decom} on the marked complementary graph,  we have
$$H_1(B_2\Gamma)\oplus\mathbb Z^{m}\cong\oplus_{i=1}^{m} H_1(B_2\Gamma_i)\oplus H_1(B_2\Theta_m).$$
To compute $H_1(B_2\Theta_m),$ we choose a maximal tree and give an order according to Lemma~\ref{lem:treeandorder}. Then there are $(m-1)(m-2)/2$ critical 1-cells of the type $A_k(\vec\delta_m)$, $(m-1)$ critical 1-cells of the type $d$ and no critical 2-cells.
Thus $$H_1(B_2\Theta_m)\cong\mathbb Z^{(m-1)(m-2)/2+(m-1)}$$
and the formula follows.
\end{proof}

Note that $\Theta_m$ for $m\ge3$ only occurs as a marked complementary graph and it never appears in a marked decomposition of a simple biconnected graph by 2-cuts.
We can repeatedly apply Lemma~\ref{lem:bidecomp} to each marked 2-cut component unless it is topologically a circle and end up with the problem how to compute $H_1(B_2\Gamma)$ for a topologically triconnected graph $\Gamma$. Note that topologically triconnected components of a given biconnected graph are topologically simple since we assuming that graphs are suitably subdivided.

Given any triconnected graph $\Gamma$, there exists a sequence $\Gamma_1,\Gamma_2,\cdots,\Gamma_r$ of graphs such that $\Gamma_1=K_4$, $\Gamma_r=\Gamma$, and for $1\le i \le r-1$, $\Gamma_{i+1}$ is obtained from $\Gamma_i$ by either adding an edge or expanding at a vertex of valency $\ge 4$ as Figure~\ref{fig:Expand} (for example, see \cite{BM}). Note that an expansion at a vertex is a reverse of a contraction of an edge with end vertices of valency $\ge 3$. When we deal with a topologically triconnected graph, we first ignore vertices of valency 2 and find a sequence and then we subdivide each graph on the sequence if necessary.

\begin{figure}[ht]
\psfrag{0}{\small$x$}
\psfrag{1}{\small$x_1$}
\psfrag{2}{\small$x_2$}
\psfrag{d1}{$\Gamma_1$}
\psfrag{d2}{$\Gamma_2$}
\psfrag{d3}{$\Gamma_3$}
\centering
\includegraphics[height=2.5cm]{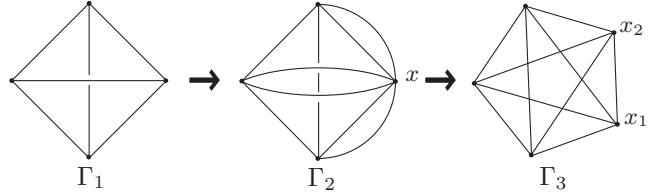}
\caption{A sequence for $K_5$}
\label{fig:Expand}
\end{figure}

\begin{lem}\label{lem:triconnect}{\rm [Topologically Simple Triconnected Graph]}
Let $\Gamma$ be a topologically simple and triconnected graph. Then all critical 1-cell of the type $A_k(\vec\delta_m)$ are homologous up to signs.
Furthermore $$H_1(B_2\Gamma)\cong\mathbb Z^{\beta_1(\Gamma)}\oplus K$$
where $K$ is $\mathbb Z$ if $\Gamma$ is planar or $\mathbb Z_2$ if $\Gamma$ is non-planar.
\end{lem}
\begin{proof}
We use induction on the number $s$ of vertices of valency $\ge 3$. To check for the smallest triconnected graph $K_4$, consider the maximal tree of $K_4$ and the order on vertices given in Figure~\ref{fig:K4PE}. Then it is easy to see that the lemma is true and in fact $H_1(B_2K_4)\cong\mathbb Z^4$.

Assume that for $s>4$, the lemma holds. Let $\Gamma$ be a triconnected graph with $s+1$ vertices of valency $\ge 3$. There is a sequence $K_4=\Gamma_1,\ldots,\Gamma_{r-1}, \Gamma_r=\Gamma$ of triconnected graphs described above. Since $\Gamma$ is topologically simple, we may assume that $\Gamma$ is obtained from $\Gamma_{r-1}$ by expanding at a vertex $x$. After ignoring vertices of valency 2, $\Gamma_{r-1}$ is a triconnected graph with $s$ vertices that may have double edges incident to $x$ and let
$\Gamma'_{r-1}$ be a simple triconnected graph obtained from $\Gamma_{r-1}$ by deleting one edge from each pair of double edges. Then there is an obvious graph embedding $\Gamma'_{r-1}\hookrightarrow\Gamma_{r-1}$. Let $x_0$ and $x_1$ be the expanded vertices of $x$ in $\Gamma$. Choose maximal trees $T$, $T_{r-1}$, and $T'_{r-1}$ of of $\Gamma$ and $\Gamma_{r-1}$ and $\Gamma'_{r-1}$ and orders on vertices according to Lemma~\ref{lem:treeandorder} so that $x$ is the base vertex for $\Gamma_{r-1}$ and $\Gamma'_{r-1}$ and $x_0$ is the base vertex 0 for $\Gamma$ as Figure~\ref{fig:Exp2}.
Then there are natural graph embeddings $T'_{r-1}\hookrightarrow T_{r-1}\hookrightarrow T$
that preserve the base vertices and orders.

Let $X$ be the second smallest vertex of valency$\ge3$ in $\Gamma_{r-1}$. If there are topologically double edges between $0$ and $X$ in $\Gamma_{r-1}$, then $X$ has valency 3 in $T_{r-1}$. Other vertices with topological double edges are situated in $\Gamma_{r-1}$ like $Y$ in Figure~\ref{fig:Exp2}. Vertices of the types $X$ or $Y$ behave in the same way in both $\Gamma_{r-1}$ and $\Gamma$. So there is a one-to-one corresponding between the set of all critical 1-cells of the form $A_k(\vec\delta_\ell)$ in $UD_2\Gamma_{r-1}$ and the set of those in $UD_2\Gamma$.

\begin{figure}[ht]
\psfrag{0}{\small0$(=x)$}
\psfrag{1}{\small0$(=x_0)$}
\psfrag{2}{\small$x_1$}
\psfrag{A}{\small$X$}
\psfrag{B}{\small$Y$}
\psfrag{d}{\small$d$}
\psfrag{d1}{\small$d_1$}
\psfrag{d2}{\small$d_2$}
\psfrag{d3}{\small$d_3$}
\centering
\subfigure[$\Gamma'_{r-1}$]{\includegraphics[height=2cm]{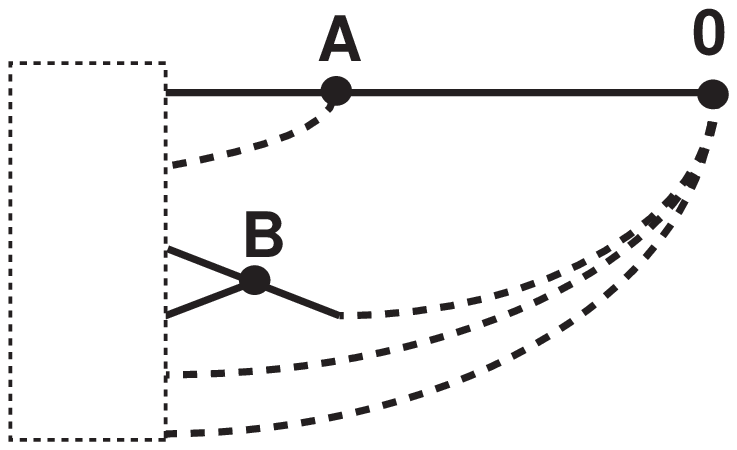}}\qquad
\subfigure[$\Gamma_{r-1}$]{\includegraphics[height=2cm]{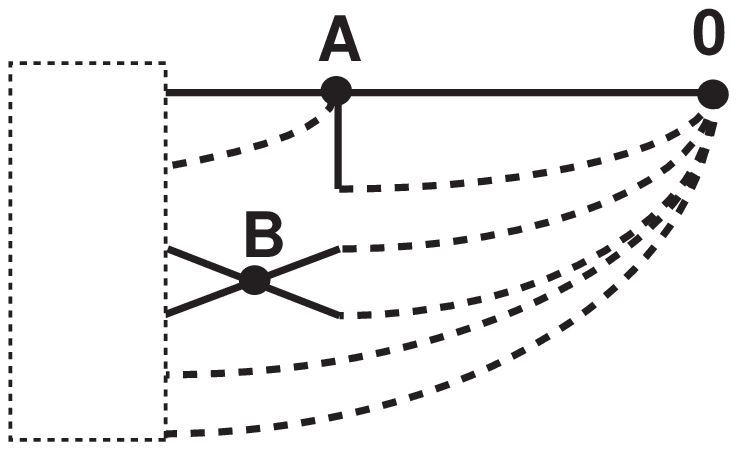}}\qquad
\subfigure[$\Gamma$]{\includegraphics[height=2cm]{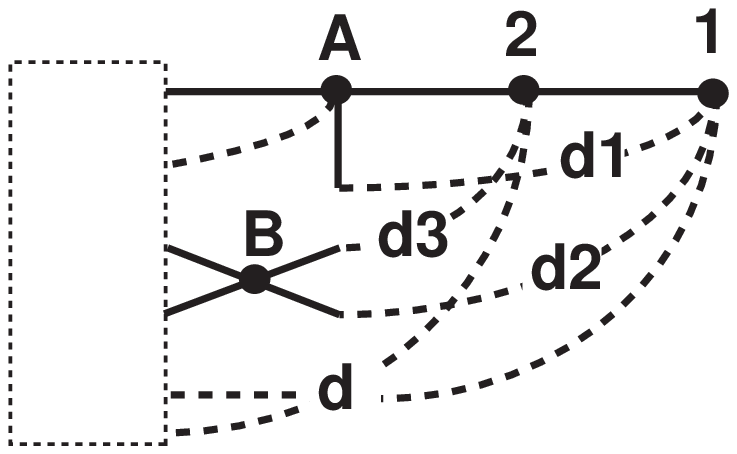}}
\caption{Expanding $\Gamma_{r-1}$ at $x$}\label{fig:Exp2}
\end{figure}

By the induction hypothesis, all critical 1-cells of the form $A_k(\vec\delta_\ell)$ in $UD_2\Gamma'_{r-1}$ are separating and homologous up to signs. We first find out which critical 1-cells of the form $A_k(\vec\delta_\ell)$ in $UD_2\Gamma_{r-1}$ is not separating. It is enough to check for the vertices of the type either $X$ or $Y$ since $A_k(\vec\delta_\ell)$ can be regarded as a critical 1-cell in $UD_2\Gamma'_{r-1}$ for all other vertices $A$ and $UD_2\Gamma_{r-1}$ has more critical 2-cells than $UD_2\Gamma'_{r-1}$. For $X$, there is only one critical 1-cell $X_2(\vec\delta_1)$ and it is not separating by Lemma~\ref{cor:separating1cell}. Suppose the $p$-th and the $(p+1)$-st branches of $Y$ are topological double edges from $Y$ to 0. Then $Y_{p+1}(\vec\delta_p)$ is not separating either by Lemma~\ref{cor:separating1cell}. Unless $k=p+1$ and $\ell=q$, $Y_k(\vec\delta_\ell)$ is separating because one of the $p$-th and the $(p+1)$-st branches lies on $\Gamma'_{r-1}$ and so $Y_k(\vec\delta_\ell)$ is homologous up to signs to other separating 1-cells by the induction hypothesis.

Finally we show that as critical 1-cells of $UD_2\Gamma$, $X_2(\vec\delta_1)$ and $Y_p(\vec\delta_q)$ are separating and homologous up to signs to other separating 1-cells. Let $d_1$ ($d_2$, respectively) be a deleted edge lying on the topological edge between 0 and $X$ ($Y$, respectively) in $\Gamma$, $d_3$ be a deleted edge lying on the topological edge between $x_1$ and $Y$ in $\Gamma$. Then $g(Y,\iota(d_3))$ and $g(Y,\iota(d_2))$ corresponds to $p$ and $p+1$. Since $\Gamma$ is topologically triconnected, there is a deleted edge $d$ other than $d_1$, $d_2$ and $d_3$ such that $\tau(d)$ is either 0 or $x_1$. Otherwise, $\{X,Y\}$ would be a 2-cut in $\Gamma$. In fact, Figure~\ref{fig:Exp2} shows examples of $d_1$, $d_2$, $d_3$, and $d$.
Consider the following boundary images on the Morse chain complex of $UD_2\Gamma$:\\
For $\tau(d)=0$,
\begin{align*}
\widetilde\partial(d_3\cup d_1-d_3\cup d)=&X_2(\vec\delta_1)\pm\mbox{$\wedge$}(d_3,d) \\
\widetilde\partial(d_3\cup d_2-d_3\cup d)=&-Y_p(\vec\delta_q)\pm\mbox{$\wedge$}(d_3,d).
\end{align*}
For $\tau(d)=x_2$,
\begin{align*}
\widetilde\partial(d_3\cup d_1-d_3\cup d_2)=&X_2(\vec\delta_1)+Y_p(\vec\delta_q)\\
\widetilde\partial(d\cup d_1-d\cup d_2)=&X_2(\vec\delta_1)\pm\mbox{$\wedge$}(d,d_2).
\end{align*}
So $$[X_2(\vec\delta_1)]=-[Y_p(\vec\delta_q)]=\pm[\mbox{$\wedge$}(d_3,d)]\mbox{(or }\pm[\mbox{$\wedge$}(d,d_2)]).$$
Thus all critical 1-cells of the form $A_k(\vec a)$ in $UD_2\Gamma$ are separating and homologous up to signs.

Now we consider $H_1(B_2\Gamma)$. Since we know $A_k(\vec a)$ are separating,  free 1-cells are of the form $d$ for some deleted edge $d$ by Lemma~\ref{cor:pivotal1cell}. The number of deleted edges is equal to $\beta_1(\Gamma)$. So $$H_1(B_2\Gamma)\cong\mathbb Z^{\beta_1(\Gamma)}\oplus\langle [A_k(\vec a)]\rangle.$$
It is easy to see that $[A_k(\vec a)]$ is not trivial in $H_1(B_2\Gamma)$. If $\Gamma$ is planar then $[A_k(\vec a)]$ is torsion free by Corollary~\ref{thm:torsionfree}. It is easy to see that if a topologically simple triconnected graph $\Gamma$ is embedded in a topologically simple triconnected graph $\widetilde\Gamma$ as graphs then the embedding induces a homomorphism $:H_1(B_2\Gamma)\to H_1(B_2\widetilde\Gamma)$ which corresponds the homology class $[A_k(\vec a)]$ to the same kind of homology classes. By Example~\ref{ex:H1B4K5} $[A_k(\vec a)]$ in $H_1(B_2K_5)$ is a 2-torsion. It is easy to check that $[A_k(\vec a)]$ in $H_1(B_2K_{3,3})$ is a 2-torsion from Example~\ref{ex:K33n2}. So if $\Gamma$ is a non-planar graph then $[A_k(\vec a)]$ in $H_1(B_2\Gamma)$ generates the summand $\mathbb Z_2$.
\end{proof}

By combining lemmas in this section, we can give a formula for $H_1(B_n\Gamma)$ for a finite connected graph $\Gamma$ and any braid indices using the connectivity of graphs. Recall
$$N(n,\Gamma,x)=\left(\begin{array}{c} n+\mu(x)-2 \\ n-1 \end{array}\right) \times(\nu(x)-2) -\left(\begin{array}{c} n+\mu(x)-2 \\ n \end{array}\right)-(\nu(x)-\mu(x)-1)$$ where $\mu(x)$ is the number of $x$-components of $\Gamma$ and $\nu(x)$ is the valency of $x$ in $\Gamma$. Note that if $\nu(x)=2$ (and so $\mu(x)=2$), then $N(n,\Gamma,x)=0$.
Let $V_1(\Gamma)$ denote a set of 1-cuts that decomposes $\Gamma$ into biconnected components and copies of topological line segments. Define $$N_1(n,\Gamma)=\sum_{x\in V_1(\Gamma)} N(n,\Gamma,x).$$

For a biconnected graph $\Gamma$, let $V_{2}$ denote a set of 2-cuts whose marked decomposition decomposes $\Gamma$ into triconnected components and copies of topological circles.
Define $$N_2(\Gamma)=\sum_{\{x,y\}\in V_{2}}\frac{(\mu(\{x,y\})-1)(\mu(\{x,y\})-2)}{2}$$
where $\mu(\{x,y\})$ denotes the number of $\{x,y\}$-components in $\Gamma$.
Note that for $C\in V_2$, $\mu(C)$ in $\Gamma$ is equal to that in any marked $D$-component for $D\in V_2$.
And note that if one of $x$ and $y$ has valency 2 for a 2-cut $\{x,y\}\in V_{2}$, then $\mu(\{x,y\})=2$.

For a connected graph $\Gamma$,
define $N_2(\Gamma)=\sum_{i=1}^k N_2(\Gamma_i)$
where $\Gamma_1,\ldots,\Gamma_k$ are biconnected components of $\Gamma$.

For a connected graph $\Gamma$, let $N_3(\Gamma)$ ($N'_3(\Gamma)$, respectively) be the number of triconnected components of $\Gamma$ that are planar (non-planar, respectively).

\begin{thm}\label{thm:H1Bn}
For a finite connected graph $\Gamma$,
$$H_1(B_n\Gamma)=\mathbb Z^{N_1(n,\Gamma)+N_2(\Gamma)+N_3(\Gamma)+\beta_1(\Gamma)}\oplus\mathbb Z_2^{N'_3(\Gamma)}$$
\end{thm}
\begin{proof}
By Lemmas~\ref{lem:cutvertex} and \ref{lem:biconnected} we have $$H_1(B_n\Gamma)= (\oplus_{i} H_1(B_2\Gamma_i))\oplus \mathbb Z^{N_1(n,\Gamma)}$$
where $\Gamma_i$'s are biconnected components of $\Gamma$.
Since $N_2(\Gamma)$, $N_3(\Gamma)$, $N'_3(\Gamma)$ and $\beta_1(\Gamma)$ are equal to the sum of those for $\Gamma_i$, it is sufficient to show that for a biconnected graph $\Gamma$, $$H_1(B_2\Gamma)=\mathbb Z^{N_2(\Gamma)+N_3(\Gamma)+\beta_1(\Gamma)}\oplus\mathbb Z_2^{N'_3(\Gamma)}.$$

Let $V_2$ be a set of 2-cuts in $\Gamma$ such that the marked decomposition along $V_2$ decompose a biconnected graph $\Gamma$ into triconnected components and copies of topological circles.
Let $\{\widehat\Gamma_i\}$ be the set of marked components obtained from $\Gamma$ by cutting along $V_2$. By Lemma~\ref{lem:bidecomp},
$$H_1(B_2\Gamma)\oplus\mathbb Z^{|V_2|}=(\oplus H_1(B_2\widehat\Gamma_i))\oplus\mathbb Z^{N_2(\Gamma)}.$$
By Lemma~\ref{lem:triconnect},
$H_1(B_2\widehat\Gamma_i)\cong\mathbb Z^{\beta_1(\widehat\Gamma_i)}\oplus\mathbb Z$,  $\mathbb Z^{\beta_1(\widehat\Gamma_i)}\oplus\mathbb Z_2$, or $\mathbb Z$ if $\widehat\Gamma_i$ is a planar triconnected graph, a nonplanar triconnected graph, or a topological circle, respectively.
Thus $\oplus H_1(B_2\widehat\Gamma_i)\cong\mathbb Z^{N_3(\Gamma)+\sum\beta_1(\widehat\Gamma_i)}\oplus\mathbb Z_2^{N'_3(\Gamma)}$. Since we are dealing with marked components, $\sum\beta_1(\widehat\Gamma_i)=\beta_1(\Gamma)+|V_2|$.
Thus $H_1(B_2\Gamma)\cong\mathbb Z^{N_2(\Gamma)+N_3(\Gamma)+\beta_1(\Gamma)}\oplus\mathbb Z_2^{N'_3(\Gamma)}$.
\end{proof}

It seems difficult to compute higher homology groups of $B_n\Gamma$ in general.
However $UD_2\Gamma$ is a 2-dimensional complex and so $H_2(B_2\Gamma)$ is torsion-free. And the second Betti number of $B_2\Gamma$ is given as follows:

\begin{cor}
For a finite connected graph $\Gamma$,
\begin{align*}
\beta_2(B_2\Gamma)=&\:N_1(n,\Gamma)+N_2(\Gamma)+N_3(\Gamma)-\frac{1}{2}\sum_{x\in V(\Gamma)}(\nu(x)-1)(\nu(x)-2)\\
&+\frac{1}{2}\beta_1(\Gamma)(\beta_1(\Gamma)-1)+2.
\end{align*}
\end{cor}
\begin{proof}
We choose a maximal tree such that two end vertices of every deleted edge have with valency 2. Then the number of critical 2-cells is equal to $\frac{1}{2}\beta_1(\Gamma)(\beta_1(\Gamma)-1)$ and the number of critical 1-cells is $\frac{1}{2}\sum_{x\in V(\Gamma)}(\nu(x)-1)(\nu(x)-2)+\beta_1(\Gamma)$. Using Euler characteristic of the Morse chain complex, we have
\begin{align*}
&\:1-\beta_1(B_2\Gamma)+\beta_2(B_2\Gamma)\\
=&\:1-\frac{1}{2}\sum_{x\in V(\Gamma)}(\nu(x)-1)(\nu(x)-2)-\beta_1(\Gamma)+\frac{1}{2}\beta_1(\Gamma)(\beta_1(\Gamma)-1)
\end{align*}
We use $\beta_1(B_2\Gamma)=N_1(n,\Gamma)+N_2(\Gamma)+N_3(\Gamma)+\beta_1(\Gamma)$ to complete the proof.
\end{proof}

\subsection{The homologies of pure graph 2-braid groups}\label{ss33:H1PB2}
In \S\ref{ss22:PBn}, we describe a Morse chain complex $M_n\Gamma$ of $D_n\Gamma$. The technology developed for $UD_n\Gamma$ in this article is not enough to compute $H_1(P_n\Gamma)$. For example, the boundary image of $(A_k(\vec a)\cup B_\ell(\vec b))_\sigma$ never vanishes in $M_n\Gamma$ for $n\ge4$. However for braid index 2 the second boundary map behaves in the way similar to unordered cases. This is because there are only one type critical 2-cells $(d\cup d')_\sigma$.

In general, the image of $c_\sigma$ under $\widetilde R$ or $\widetilde\partial$ is obtained by right multiplication by $\sigma$ on the permutation subscript of each term in the image of $c_{\rm id}$. For example, if $\widetilde R(c_{\rm id})=\sum_i (c_i)_{\tau_i}$ then $\widetilde R(c_{\sigma})=\sum_i (c_i)_{\tau_i\sigma}$. Thus we only consider $c_{\rm id}$. We will discuss 2-braid groups in this section and $\rho$ denotes the nontrivial permutation in $S_2$.

We have the following lemma for $D_2\Gamma$ that is similar to Lemma~\ref{lem:naturalmove} for $UD_n\Gamma$ but it is hard to have a lemma corresponding to Lemma~\ref{lem:naturalmove-2}.

\begin{lem}\label{lem:naturalmoveD2}{\rm [Special Reduction]}
Suppose a redundant 1-cell $c_{\rm id}$ in $D_2\Gamma$ has a simple unblocked vertex. Then $\widetilde R(c_{\rm id})=\widetilde R(V(c)_{\rm id})$.
\end{lem}
\begin{proof}
Let $e$ and $v$ be the edge and the vertex in $c$. Since $c$ contains only one vertex, $v$ is the smallest unblocked vertex. Let $e_v$ be the edge starting from $v$. Then $$R(c_{\rm id})=V(c)_{\rm id}+\{e_v,\iota(e)\}_{\rho^m}-\{e_v,\tau(e)\}_{\rm id}$$
where $m$ is 1 if $\tau(e)<v<\iota(e)$ or 0 otherwise.

We use induction on $i$ such that $R^i(c_{\rm id})=R^{i+1}(c_{\rm id})$. Since $c_{\rm id}$ is redundant, $i\ge2$.
Since $V(\{e_v,\iota(e)\}_{\rho^m})=\{e_v,\tau(e)\}_{\rho^{2m}}$,
$\widetilde RV(\{e_v,\iota(e)\}_{\rho^m})=\widetilde R(\{e_v,\tau(e)\}_{\rho^{2m}})$ by induction hypothesis.
Thus $\widetilde R(\{e_v,\iota(e)\}_{\rho^m}-\{e_v,\tau(e)\}_{\rm id})=\widetilde RV(\{e_v,\iota(e)\}_{\rho^m})-\widetilde R(\{e_v,\tau(e)\}_{\rm id})=0$. Thus $\widetilde R(c_{\rm id})=\widetilde R(V(c)_{\rm id})$.
\end{proof}

Since all critical 2-cells in $D_\Gamma$ is of the form $(d\cup d')_\sigma$, we only need the following:

\begin{lem}\label{lem:boundary-order}{\rm [Boundary Formulae]}
Let $c=d\cup d'$, $\tau(d)>\tau(d')$, $k=g(\tau(d),\iota(d))$ and $\ell=g(\tau(d),\iota(d'))$.
\begin{itemize}
\item[(a)] If $d'$ is separated by $\tau(d)$, $k\ne\ell$ and $\iota(d)<\iota(d')$ then $$\widetilde\partial (c_{\rm id})=d_{\rm id}-d(\vec\delta_\ell)_\rho.$$
\item[(b)] If $d'$ is separated by $\tau(d)$, $k=\ell$ and $\iota(d)<\iota(d')$ then $$\widetilde\partial (c_{\rm id})=d_{\rm id}-d(\vec\delta_\ell)_\rho-\mbox{$\wedge$}(d,d')_{\rm id}.$$
\item[(c)] If $d'$ is separated by $\tau(d)$ and $\iota(d')<\iota(d)$ then $$\widetilde\partial (c_{\rm id})=d_{\rm id}-d(\vec\delta_\ell)_\rho+\mbox{$\wedge$}(d,d')_\rho.$$
\item[(d)] Otherwise $\widetilde\partial (c_{\rm id})=0$.
\end{itemize}
\end{lem}
\begin{proof}
It is sufficient to compute images under $\widetilde R$ for each boundary 1-cell after obtaining the boundary of $c_{\rm id}$ in $D_2\Gamma$.

If $\iota(d')<\tau(d)$, then
$$\partial((d\cup d')_{\rm id})=(d\cup\{\tau(d')\})_{\rm id}-(d\cup\{\iota(d')\})_{\rm id}+(d'\cup\{\iota(d)\})_{\rm id}-(d'\cup\{\tau(d)\})_{\rm id}.$$
By Lemma~\ref{lem:naturalmoveD2} we have $\widetilde R((d\cup\{\tau(d')\})_{\rm id})=\widetilde R((d\cup\{\iota(d')\})_{\rm id})=d_{\rm id}$.
Since $\iota(d')<\tau(d)<\iota(d)$, we have $\widetilde R((d'\cup\{\iota(d)\})_{\rm id})=d'\cup\{\tau(d)\}_{\rm id}$. So if $\iota(d')<\tau(d)$, then
$\widetilde\partial((d\cup d')_{\rm id})=0$.

If $\iota(d')>\tau(d)$, then
$$\partial((d\cup d')_{\rm id})=(d\cup\{\tau(d')\})_{\rm id}-(d\cup\{\iota(d')\})_\rho+(d'\cup\{\iota(d)\})_{\rm id}-(d'\cup\{\tau(d)\})_{\rm id}.$$
Let $B=\iota(d')\wedge \tau(d)$ and $m=g(B,\iota(d'))$. If $d'$ is not separated by $\tau(d)$ then $\iota(d')>\iota(d)$ from $g(B,\iota(d'))>g(B,\tau(d))=g(B,\iota(d))$. So we have $\widetilde R((d'\cup\{\iota(d)\})_{\rm id})=(d'\cup\{\tau(d)\})_{\rm id}$ and
\begin{align*}
\widetilde R((d\cup\{\iota(d')\})_\rho)=&\widetilde R((d\cup B(\vec\delta_m))_\rho)\\
=&\widetilde R((d\cup\{ B\})_{\rm id}+(B_m\cup\{\iota(d)\})_{\rm id}-(B_m\cup\{\tau(d)\})_{\rm id})\\
=&\widetilde R((d\cup\{ B\})_{\rm id})=\widetilde R((d\cup\{\tau(d')\})_{\rm id}),
\end{align*}
since $\widetilde R((B_m\cup\{\iota(d)\})_{\rm id})=\widetilde R((B_m\cup\{\tau(d)\})_{\rm id})$ by Lemma~\ref{lem:naturalmoveD2}. So if $\iota(d')>\tau(d)$ and $d'$ is not separated by $\tau(d)$, then $\widetilde\partial((d\cup d')_{\rm id})=0$.

Let $q=g(\tau(d'),\tau(d))$. If $d'$ is separated by $\tau(d)$, then it is easy to see that $$\widetilde R((d\cup\{\tau(d')\})_{\rm id})=d_{\rm id} \mbox{ and } \widetilde R((d'\cup\{\tau(d)\})_{\rm id})=d'(\vec\delta_q)_{\rm id}.$$
Let $C=\iota(d')\wedge\iota(d)$ and $m_1=g(C,\iota(d'))$. Consider $\widetilde R((d\cup\{\iota(d')\})_\rho)$. If either $k\ne\ell$ or $k=\ell$ and $\iota(d')<\iota(d)$, then $\widetilde R((d\cup\{\iota(d')\})_\rho)=d(\vec\delta_\ell)_\rho$ by Lemma~\ref{lem:naturalmoveD2}. If $k=\ell$ and $\iota(d)<\iota(d')$, then by Lemma~\ref{lem:naturalmoveD2}
\begin{align*}
\widetilde R((d\cup\{\iota(d')\})_\rho)=&\widetilde R((d\cup C(\vec\delta_{m_1}))_\rho)\\
=\widetilde R((d\cup\{ C\})_\rho&+(C_{m_1}\cup\{\iota(d)\})_{\rm id}-(C_{m_1}\cup\{\tau(d)\})_\rho)\\
=\widetilde R((d\cup\{ C\})_\rho&+(C_{m_1}\cup\{\iota(d)\})_{\rm id})=d(\vec\delta_\ell)_\rho+\mbox{$\wedge$}(d,d')_{\rm id}.
\end{align*}

Let $m_2=g(C,\iota(d))$. Finally consider $\widetilde R((d'\cup\{\iota(d)\})_{\rm id})$. If $\iota(d)<\iota(d')$ then $\widetilde R((d'\cup\{\iota(d)\})_{\rm id})=d'(\vec\delta_q)_{\rm id}$. If $\iota(d')<\iota(d)$, then
\begin{align*}
\widetilde R((d'\cup\{\iota(d)\})_{\rm id})=&\widetilde R((d'\cup C(\vec\delta_{m_2}))_{\rm id})\\
=\widetilde R((d'\cup\{ C\})_{\rm id}&+(C_{m_2}\cup\{\iota(d)\})_\rho-(C_{m_2}\cup\{\tau(d)\})_{\rm id})\\
=\widetilde R((d'\cup\{ C\})_{\rm id}&+(C_{m_2}\cup\{\iota(d)\})_\rho)=d(\vec\delta_\ell)_{\rm id}+\mbox{$\wedge$}(d,d')_\rho.
\end{align*}
Combining the results, we obtain the desired formulae.
\end{proof}

Using the above lemma, we have the following lemma similar to Lemma~\ref{cor:boundary2}.

\begin{lem}\label{cor:boundary-order}{\rm [Dependence among Boundary Images]}
If $d_1$ and $d_2$ are separated by $\tau(d)$ and $g(\tau(d),\iota(d_1))=g(\tau(d),\iota(d_2))$, then
\begin{itemize}
\item[(1)] If $g(\tau(d),\iota(d))\ne g(\tau(d),\iota(d_1))$, then $\widetilde\partial ((d\cup d_1)_{\rm id})=\widetilde\partial ((d\cup d_2)_{\rm id})$.
\item[(2)] If $g(\tau(d),\iota(d))=g(\tau(d),\iota(d_1))$, then $$\widetilde\partial ((d\cup d_1)_{\rm id}-(d\cup d_2)_{\rm id})=-(-1)^i\mbox{$\wedge$}(d,d_1)_{\rho^i}+(-1)^j\mbox{$\wedge$}(d,d_2)_{\rho^j}.$$
\end{itemize}
\end{lem}

Note that the second formula of the above lemma contains $i,j$ only for the parity purpose and play an important role of showing $H_1(P_2\Gamma)$ is torsion-free.

Declare an order on $S_2$ by $\mbox{\rm id}>\rho$. Recall the orders on critical 1-cells and critical 2-cells of $UD_n\Gamma$ from \S\ref{ss31:PMatrix}. By adding a permutation as the last component of the orders, we obtain orders given by 4-tuples $(s(c),e,\vec a,\sigma)$ for critical 1-cells in $D_2\Gamma$ and by 7-tuples $(s(c),e,\vec a+\vec\delta_{g(\tau(e),\iota(e'))},g(\tau(e),\iota(e')), e',\vec b,\sigma)$ for critical 2-cells.

The second boundary homomorphism $\widetilde\partial$ is represented by a matrix over bases of critical 2-cells and critical 1-cells ordered reversely. We go through the exactly same arguments as Sec~\ref{ss31:PMatrix} by using Lemmas~\ref{cor:boundary-order} and \ref{lem:boundary-order} and obtain the following theorem:

\begin{thm}\label{thm:matrix-D2}
Let $M$ be the matrix representing the second boundary homomorphism of $D_2\Gamma$ over bases of critical 2-cells and critical 1-cells ordered reversely.
Up to row operations, each row of $M$ satisfies one of the following:
\begin{enumerate}
\item[\rm{(1)}] consists of all zeros;
\item[\rm{(2)}] there is a $\pm 1$ entry that is the only nonzero entry in the column it belongs to;
\item[\rm{(3)}] there are only only two nonzero entries which are $\pm1$.
\end{enumerate}
Furthermore, up to multiplications of column by $-1$, the property (3) above can be modified to
\begin{enumerate}
\item[\rm{($3'$)}] there are two nonzero entries which are $\pm1$ and have opposite signs.
\end{enumerate}
\end{thm}
\begin{proof} Lemma~\ref{cor:boundary-order}(2) implies that ($3'$) can be achieved by choosing a basis of critical 1-cells in which $(-1)^m c_{\rho^m}$ is used instead of just $c_{\rm id}$ or $c_{\rho}$.
\end{proof}

Since there are exactly two critical 0-cells, the 0-th skeleton $(M_2\Gamma)^0$ of a Morse complex of $M_2\Gamma$ of $D_2\Gamma$ consists of two points. Then the second boundary homomorphism gives a presentation matrix for $H_1(M_2\Gamma,(M_2\Gamma)^0)$. And $H_1(M_2\Gamma,(M_2\Gamma)^0)\allowbreak\cong H_1(M_2\Gamma)\oplus\mathbb Z\cong H_1(P_2\Gamma)\oplus\mathbb Z$.

Critical 1-cells of $D_2\Gamma$ can be classified to be pivotal, free, or separating as before. The undetermined block of separating 1-cells produces no torsion due to the property ($3'$) and so we have the following:

\begin{cor}
For a finite connected graph $\Gamma$, $H_1(P_2\Gamma)$ is torsion-free.
\end{cor}

Using free 1-cells and the undetermined block for $H_1(M_2\Gamma,(M_2\Gamma)^0)$,  we can compute $H_1(P_2\Gamma)$.

\begin{exa}\label{ex:H1PB2K5}
Let $\Gamma$ be $K_5$ and a maximal tree and an order be given as Example~\ref{ex:K5TO}. We want to compute $H_1(P_2\Gamma)$.
\end{exa}
From $\widetilde\partial(d_6\cup d_5-d_6\cup d_2)=-C_2(1,0,0)-B_3(0,1,0)$ in Example~\ref{ex:H1B4K5}, we obtain
$$\widetilde\partial((d_6\cup d_5)_{\rm id}-(d_6\cup d_4)_{\rm id})=-C_2(1,0,0)_{\rm id}+\{-B_3(0,1,0)_\rho\}.$$
From Example~\ref{ex:H1B4K5} and Lemma~\ref{lem:boundary-order}, we obtain the undetermined block as follows:
$$\left(
\begin{array}{cccccccccccccc}
&&&&-1&&&&&1&&&&\\
&&&&&1&&&-1&&&&&\\
&&&&&&&-1&&1&&&&\\
&&&&&&1&&-1&&&&&\\
-1&&&&&&&&&1&&&&\\
&1&&&&&&&-1&&&&&\\
&&-1&&&&&&&&&&&1\\
&&&1&&&&&&&&&-1&\\
&&&&&&&&&-1&&&&1\\
&&&&&&&&1&&&&-1&\\
&&&&&&-1&&&&&&&1\\
&&&&&&&1&&&&&-1&\\
&&&&&&&&&&-1&&&1\\
&&&&&&&&&&&1&-1&
\end{array} \right)$$

$$\rightarrow
\left(
\begin{array}{cccccccccccccc}
-1&&&&&&&&&1&&&&\\
&1&&&&&&&-1&&&&&\\
&&-1&&&&&&&&&&&1\\
&&&1&&&&&&&&&-1&\\
&&&&-1&&&&&1&&&&\\
&&&&&1&&&-1&&&&&\\
&&&&&&1&&-1&&&&&\\
&&&&&&&-1&&1&&&&\\
&&&&&&&&1&&&&-1&\\
&&&&&&&&&-1&&&&1\\
&&&&&&&&&&-1&&&1\\
&&&&&&&&&&&1&-1&\\
&&&&&&&&&&&&-1&1\\
&&&&&&&&&&&&&0
\end{array} \right)$$
There are twelve free 1-cells, all of which are of the form $d_\sigma$. From the above matrix, $H_1(M_2\Gamma,(M_2\Gamma)^0)\cong\mathbb Z^{13}$ and so $H_1(P_2\Gamma)\cong\mathbb Z^{12}$.
\qed

For a free 1-cell $c$ in $UD_2\Gamma$, $c_{\rm id}$ and $c_\rho$ are free 1-cells in $D_2\Gamma$. So it is easy to modify Lemma~\ref{lem:cutvertex} and \ref{lem:biconnected} for $H_1(M_2\Gamma,(M_2\Gamma)^0)$ accordingly and one can verify that the contribution by $N_1(2,\Gamma)$ and $N_2(\Gamma)$ doubles because the number of free 1-cells doubles. However the proof of Lemma~\ref{lem:triconnect} deals with the undetermined block and it is safe to redo.

\begin{lem}{\rm [Topologically Simple Triconnected Graph]}
For a topologically simple and triconnected $\Gamma$, $$H_1(P_2\Gamma)\cong\mathbb Z^{2\beta_1(\Gamma)+\epsilon}$$
where $\epsilon$ is 1 if $\Gamma$ is planar or 0 if $\Gamma$ is non-planar.
\end{lem}

\begin{proof}
We need to show $H_1(M_2\Gamma,(M_2\Gamma)^0)\cong\mathbb Z^{2\beta_1(\Gamma)+\epsilon+1}$. Critical 1-cells are of the forms $d_\sigma$, $d(\delta_\ell)_\sigma$ and $A_k(\vec\delta_\ell)_\sigma$ with $k>\ell$. It is easy to see that every critical 1-cell of the form $d(\delta_\ell)_\sigma$ is pivotal and the number of critical 1-cells of the form $d_\sigma$ is equal to $2\beta_1(\Gamma)$. We consider the undetermined block. From the proof of Lemma~\ref{lem:triconnect}, there are at most two homology classes of the form $[A_k(\vec\delta_\ell)_{\rm id}]$ and $[A_k(\vec\delta_\ell)_\rho]$. So it is sufficient to show $[A_k(\vec\delta_\ell)_{\rm id}]\ne[A_k(\vec\delta_\ell)_\rho]$ if $\Gamma$ is planar and $[A_k(\vec\delta_\ell)_{\rm id}]=[A_k(\vec\delta_\ell)_\rho]$ if $\Gamma$ is non-planar. If $\Gamma$ is planar, then by Condition (T4) in Lemma~\ref{lem:treeandorderforplanar}, there is no row representing $\pm\{A_k(\vec\delta_\ell)_{\rm id}+B_{k'}(\vec\delta_{\ell'})_\rho\}$. So $[A_k(\vec\delta_\ell)_{\rm id}]\ne[A_k(\vec\delta_\ell)_\rho]$. For non-planar graphs, we only need to verify for $K_5$ and $K_{3,3}$ as explain in the proof of Lemma~\ref{lem:triconnect}. Examples~\ref{ex:K33n2-order} and \ref{ex:H1PB2K5} show that $H_1(P_2K_5)$ and $H_1(P_2K_{3,3})$ satisfy the lemma.
\end{proof}

Using the same arguments in the proof of Theorem~\ref{thm:H1Bn}, we obtain the formula
$$H_1(M_2\Gamma,(M_2\Gamma)^0)\cong\mathbb Z^{2N_1(2,\Gamma)+2N_2(\Gamma)+2N_3(\Gamma)+2\beta_1(\Gamma)+N'_3(\Gamma)}.$$
This implies the following theorem.

\begin{thm}\label{thm:H1PB2}
For a finite connected graph $\Gamma$,
$$H_1(P_2\Gamma)\cong\mathbb Z^{2N_1(2,\Gamma)+2N_2(\Gamma)+2N_3(\Gamma)+2\beta_1(\Gamma)+N'_3(\Gamma)-1}$$
\end{thm}

Since there are no critical $i$-cells for $i\ge3$ in $D_2\Gamma$, $H_*(P_2\Gamma)$ is torsion-free. So we can can compute $H_2(P_2\Gamma)$ as follows. Choose a maximal tree such that two end vertex of every deleted edge have valency 2. Then there are two critical 0-cells and $\beta_1(\Gamma)(\beta_1(\Gamma)-1)$ critical 2-cells and the number of critical 1-cells is equal to $\sum_{x\in V(\Gamma)}(\nu(x)-1)(\nu(x)-2)+2\beta_1(\Gamma)$. So we have the second Betti number of $P_2\Gamma$ as follows:
\begin{align*}
\beta_2(P_2\Gamma)=&\:2N_1(2,\Gamma)+2N_2(\Gamma)+2N_3(\Gamma)+N'_3(\Gamma)+
\beta_1(\Gamma)(\beta_1(\Gamma)-1)\\
&-\sum_{x\in V(\Gamma)}(\nu(x)-1)(\nu(x)-2).
\end{align*}
A formula for $\beta_1(P_2\Gamma)-\beta_2(P_2\Gamma)$ was given by Barnett and Farber in~\cite{BF}.

As a closing thought of this section, it is tempting to use Lyndon-Hochschild-Serre spectral sequence for $S_n=B_n\Gamma/P_n\Gamma$ to extract some information about $H_1(P_n\Gamma)$ via homologies of the other two groups. In fact, we have the exact sequence
$$H_2(B_n\Gamma)\to H_2(S_n)\to H_1(P_n\Gamma)_{S_n}\to H_1(B_n\Gamma)\to H_1(S_n)\to 0$$
where $H_1(P_n\Gamma)_{S_n}$ is isomorphic to $H_1(P_n\Gamma)/\mathcal IH_1(P_n\Gamma)$ as $\mathbb Z[S_n]$-modules and $\mathcal I$ is the kernel of the augmentation $\mathbb Z[S_n]\to \mathbb Z$. Even though the action on critical 1-cells by $S_n$ is clearly understood, the action on homology classes is not so clear without any information about the second boundary map.

\section{Applications and more characteristics of graph braid groups}\label{s:four}
In this section we first discuss consequences of the formulae obtained in the previous section. Then we develop a technology for graph braid groups themselves that is parallel to the technology successfully applied for the first homologies of graph braid groups. And we discover more characteristics of graph braid groups and pure braid groups beyond their homologies. These characteristics are defined by weakening the requirement for right-angled Artin groups.

\subsection{Planar and non-planar triconnected graphs}

Since we are not interested in trivial graphs such as a topological line segment of a topological circle, we assume $\Gamma$ has at least a vertex of of valency $\ge 3$ in this discussion.
For any 1-cut $x$ of valency $\ge 3$, $N(n,\Gamma,x)> 0$. Thus
$N_1(n,\Gamma)=0$ if and only if there is no 1-cut of valency $\ge 3$ if and only if $\Gamma$ is biconnected.
If $N_2(\Gamma)=0$ for a biconnected graph $\Gamma$, then $\mu(\{x,y\})=2$ for every 2-cut $\{x,y\}$.
If $\Gamma$ has multiple edges between vertices $x$ and $y$ after ignoring vertices with valency 2, then $\mu(\{x,y\})>2$ for some 2-cut $\{x,y\}$ and so $N_2(\Gamma)>0$. Thus if $N_2(\Gamma)=0$ for a biconnected graph, then $\Gamma$ is topologically simple.
If $N_3(\Gamma)=N_3'(\Gamma)=0$, then $\Gamma$ does not topologically contain the complete graph $K_4$ by the construction of triconnected graphs.
If $N_1(n,\Gamma)=N_2(\Gamma)=0$ and $N_3(\Gamma)+N_3'(\Gamma)=1$, then $\Gamma$ is topologically simple and triconnected.

Barnett and Farber proved in \cite{BF} that if for a finite connected planar graph $\Gamma$ with no vertices of valency $\le2$ that is embedded in $\mathbb R^2$, the connected components $U_0,U_1,\cdots,U_r$ of the complement $\mathbb R^2-p(\Gamma)$ with the unbounded component $U_0$ satisfy
\begin{itemize}
\item[(i)] the closure of every domain $\bar U_{i>0}$ is contractible and $\bar U_0$ is homotopy equivalent to $S^1$,
\item[(ii)] for every $i,j\in\{1,\cdots,r\}$, $\bar U_i\cap\bar U_j$ is connected,
\end{itemize}
then $\beta_1(D_2\Gamma)=2\beta_1(\Gamma)+1$.

Condition (i) implies that $\Gamma$ has no 1-cut. Condition (ii) imply that either $\Gamma$ is the $\Theta$-shape graph if $|V(\Gamma)|=2$ or $\Gamma$ has neither multiple edges nor 2-cuts if $|V(\Gamma)|>2$. So the hypotheses imply that $\Gamma$ is either the $\Theta$-shape graph or a planar simple triconnected graph. Thus Theorem~\ref{thm:H1PB2} covers this result. Furthermore, for any planar graph $\Gamma$, $\beta_1(P_2\Gamma)=2\beta_1(\Gamma)+1$ if and only if $N_1(2,\Gamma)+N_2(\Gamma)+N_3(\Gamma)=1$ and $N_3'(\Gamma)=0$. There are three nonnegative solutions: $(N_1(2,\Gamma),N_2(\Gamma),N_3(\Gamma))=(1,0,0)$, $(0,1,0)$ and $(0,0,1)$.

In the case of $(1,0,0)$, $\Gamma$ has only one 1-cut vertex of valency 3. So $\Gamma$ is either the $Y$-shape tree or the $P$-shape graph. In the case of $(0,1,0)$, $\Gamma$ is biconnected and has only one 2-cut $\{x,y\}$ with $\mu(\{x,y\})=3$. So $\Gamma$ is the $\Theta$-shape graph. Finally, the solution $(0,0,1)$ implies that $\Gamma$ is topologically simple and triconnected. Thus for a connected planar graph $\Gamma$ with no vertices of valency $\le 2$, $\beta_1(P_2\Gamma)=2\beta_1(\Gamma)+1$ if and only if $\Gamma$ is either the $\Theta$-shape graph or a simple triconnected graph.
Note that we cannot remove the assumption of being planar because there is a counterexample given in Figure~\ref{fig:coun-ex}.

\begin{figure}[ht]
\centering
\includegraphics[height=2.5cm]{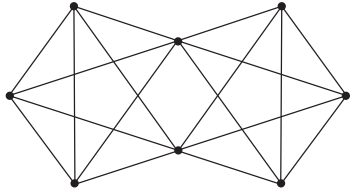}
\caption{A non-planar graph $\Gamma$ with $\beta_1(P_2\Gamma)=2\beta_1(\Gamma)+1$}
\label{fig:coun-ex}
\end{figure}

Farber and Hanbury proved in \cite{FH} that if for a graph $\Gamma$ with no vertices of valency $\le 2$, there exists a sequence $\Gamma_1,\Gamma_2,\cdots,\Gamma_r$ of graphs satisfying:
\begin{itemize}
\item[(i)] $\Gamma_1$ is either $K_5$ or $K_{3,3}$ and $\Gamma_r=\Gamma$.
\item[(ii)] For $1\le i \le r-1$, $\Gamma_{i+1}$ is obtained by adding an edge with ends $\{x,y\}$ to $\Gamma_i$ such that the complement $\Gamma_{i}-\{x,y\}$ is connected where $x$ and $y$ are points in $\Gamma_i$.
\end{itemize}
Then $\beta_1(D_2\Gamma)=2\beta_1(\Gamma)$.

The above construction obviously produces a non-planar, simple and triconnected graph $\Gamma$. Then $N_1(2,\Gamma)=N_2(\Gamma)=N_3(\Gamma)=0$ and $N_3'(\Gamma)=1$ and so Theorem~\ref{thm:H1PB2} contains this result. Moreover they conjectured that $\Gamma$ is non-planar and triconnected (This is equivalent to their hypothesis) if and only if $\beta_1(D_2\Gamma)=2\beta_1(\Gamma)$ and $H_1(P_2\Gamma)$ is torsion free. The same theorem also verifies this conjectures. Theorem~\ref{thm:H1PB2} implies that $\beta_1(P_2\Gamma)=2\beta_1(\Gamma)$ if and only if $2N_1(2,\Gamma)+2N_2(\Gamma)+2N_3(\Gamma)+N_3'(\Gamma)=1$. There is only one nonnegative solution $N_1(2,\Gamma)=N_2(\Gamma)=N_3(\Gamma)=0$ and $N_3'(\Gamma)=1$ for the equation. Thus $\beta_1(P_2\Gamma)=2\beta_1(\Gamma)$ if and only if the graph $\Gamma$ is non-planar, topologically simple and topologically triconnected.

\subsection{Graph braid groups and commutator-related groups}\label{ss41:GBG}

A group $G$ is \emph{commutator-related} if it has a finite presentation $\langle x_1,\cdots,x_n \mid r_1,\cdots, r_m\rangle$ such that each relator $r_j$ belongs to the commutator subgroup $[F,F]$ of the free group $F$ generated by $x_1,\ldots,x_n$.
We will prove that planar graph braid groups and pure graph 2-braid groups are commutator-related groups.

Since the abelianization of a given group $G$ is the first homology of $G$, we have the following.

\begin{pro}\label{pro:CRG}
Let $G$ be a group such that $H_1(G)\cong\mathbb Z^m$. If $G$ has a finite presentation with $m$-generators, then $G$ is commutator-related.
\end{pro}

Let $\Gamma$ be a planar graph. Since $UD_n\Gamma$ is a finite complex, $B_n\Gamma$ has a finite presentation. To prove that $B_n\Gamma$ is a commutator-related group, it is sufficient to show that there is a finite presentation with $m$ generators for $B_n\Gamma$ for $m=\beta_1(UD_n\Gamma)$.

The braid group $B_n\Gamma$ is given by the fundamental group of a Morse complex $UM_n\Gamma$ of $UD_n\Gamma$. Thus $B_n\Gamma$ has a presentation whose generators are critical 1-cells and whose relators are boundary words of critical 2-cells in terms of critical 1-cells. On the other hand, the computation using critical 1-cells and critical 2-cells in a Morse complex $M_n\Gamma$ of $D_n\Gamma$ does not give $P_n\Gamma$ since there are n! critical 0-cells and critical 1-cells between distinct critical 0-cells are also treated as generators. Instead it gives $\pi_1(M_n\Gamma/\sim)$ where $M_n\Gamma/\sim$ is the quotient obtained by identifying all critical 0-cells.

Even though discrete Morse theory can apply to $D_n\Gamma$ for any braid index $n$, we have not reached a level of sophistication enough to make good use due to obstacles explained in \S\ref{ss33:H1PB2}. For $n=2$, $\pi_1(M_2\Gamma/\sim)=P_2\Gamma*\mathbb Z$.  In fact $M_2\Gamma/\sim$ is homotopy equivalent to the wedge product of $M_2\Gamma$ and $S^1$ under a homotopy sliding one critical 0-cell to the other along a critical 1-cell and therefore
a presentation of $P_2\Gamma$ is obtained from that of $\pi_1(M_2\Gamma/\sim)$ by killing any one of critical 1-cells joining two 0-cells in the Morse complex $M_2\Gamma$, for example, a critical 1-cell of the form $A_k(\vec\delta_\ell)_{\rm id}$. Thus it is enough to show $\pi_1(M_2\Gamma/\sim)$ is a commutator-related group.

In order to rewrite a word in 1-cells of $UD_n\Gamma$ into an equivalent word in critical 1-cells, we use the rewriting homomorphism $\tilde r$ from the free group on 1-cells to the free group on critical 1-cells defined as follows: First define a homomorphism $r$ from the free group on $K_1$ to itself by $r(c)= 1$ if $c$ is collapsible, $r(c)= c$ if $c$ is critical, and $$r(c)=\{v\mbox{-}v_1,v_2,\ldots,v_{n-1},\iota(e)\} \{v,v_2,\ldots,v_{n-1},e\} \{v\mbox{-}v_1,v_2,\ldots,v_{n-1},\tau(e)\}^{-1}$$ if $c=\{v_1,v_2,\ldots,v_{n-1},e\}$ is redundant such that $v_1$ is the smallest unblocked vertex and $e$ is the edge in $c$. In fact, the abelian version of $r$ is the map $R$ defined in \S\ref{ss21:Bn}. Forman's discrete Morse theory in \cite{For} guarantees that there is a nonnegative integer $k$ such that $r^k(c)=r^{k+1}(c)$ for all 1-cells $c$. Let $\tilde r=r^k$, then for any 1-cell $c$, $\tilde r(c)$ is a word in critical 1-cells that is the image of $c$ under the quotient map defined by collapsing $UD_n\Gamma$ onto its Morse complex. We note that $k=0$ iff $c$ is critical, $k=1$ iff $c$ is collapsible, and $k\ge 2$ iff $c$ is redundant. By considering ordered $n$-tuples, we can similarly define $\tilde r$ from the free group on 1-cells of $D_n\Gamma$ to the free group on critical 1-cells of $D_n\Gamma$.

By rewriting the boundary word of a critical 2-cell in terms of critical 1-cells, it is possible to compute a presentation of $B_n\Gamma$ (or $\pi_1(D_n\Gamma/\sim)$, respectively) using a Morse complex of $UD_n\Gamma$ (or $D_n\Gamma$). However, the computation of $\tilde r$ is usually tedious and the following lemma somewhat shortens it.

\begin{lem}{\rm (Kim-Ko-Park \cite{KKP})}\label{lem:rewriting}
Let $c$ be a redundant 1-cell and $v$ be a unblocked vertex. Suppose that for the edge $e$ starting from $v$, there is no vertex $w$ that is either in $c$ or an end vertex of an edge in $c$ and satisfies $\tau(e)< w<\iota(e)$. Then $\tilde r(c)=\tilde r(V_e(c))$ where $V_e(c)$ denotes the $1$-cell obtained from $c$ by replacing $\iota(e)$ by $\tau(e)$.
\end{lem}

\begin{exa}\label{ex:T4n3}
We show that $B_3\Theta_4$ and $P_3\Theta_4$ are surface groups. These will serve counterexamples later.
\end{exa}
\begin{figure}[ht]
\psfrag{0}{\small0}
\psfrag{1}{\small1}
\psfrag{3}{\small3}
\psfrag{4}{\small4}
\psfrag{5}{\small5}
\psfrag{A}{\small$2(=A)$}
\psfrag{d1}{\small$d_1$}
\psfrag{d2}{\small$d_2$}
\psfrag{d3}{\small$d_3$}
\centering
\includegraphics[height=2.5cm]{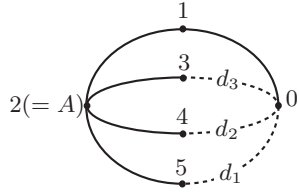}
\caption{$\Theta_4$ with the maximal tree and the order}
\label{fig:T4n3}
\end{figure}

Choose a maximal tree and an order on vertices as Figure~\ref{fig:T4n3}. First we compute $B_3\Theta_4$. There is eight critical 1-cells $A_2(1,0,0)$, $A_2(1,0,1)$, $A_3(1,0,0)$, $A_3(0,1,0)$, $A_3(1,1,0)$, $d_1$, $d_2$, $d_3$ and three critical 2-cells $A_2(1,0,0)\cup d_1$, $A_3(1,0,0)\cup d_2$, $A_3(0,1,0)\cup d_3$.
Using Lemma~\ref{lem:rewriting}, relators are given as follows:
\begin{align*}
\tilde r\circ\partial_w(A_2(1,0,0)\cup d_1)=&\tilde r(\{2\mbox{-}4,3,5\}\{d_1,2,3\}\{2\mbox{-}4,0,3\}^{-1}\{d_1,3,4\}^{-1})\\
=&\{2\mbox{-}4,3,5\}\{d_1,1,2\}\{2\mbox{-}4,0,3\}^{-1}\{d_1,1,2\}^{-1}\\
=&A_2(1,0,1)\cdot d_1\cdot A_2(1,0,0)^{-1}\cdot d_1^{-1}.
\end{align*}
Similarly,
$$\tilde r\circ\partial_w(A_3(1,0,0)\cup d_2)=A_3(1,1,0)\cdot d_2\cdot A_3(1,0,0)^{-1}\cdot d_2^{-1}\cdot A_3(0,1,0)^{-1}$$
and
\begin{align*}
\tilde r\circ\partial_w(A_3(0,1,0)\cup d_3)&\\
=A_3(1,1,0)\cdot A_2(1,0,&0)\cdot d_3\cdot A_3(0,1,0)^{-1}\cdot d_3^{-1}\cdot A_3(1,0,0)^{-1}\cdot A_2(1,0,1)^{-1}.
\end{align*}
We perform Tietze transformations that add three generators and relations as follows:
\begin{align*}
D_2=&A_3(0,1,0)\cdot d_2\cdot A_2(1,0,0)^{-1}\cdot A_3(0,1,0)^{-1},\ D_3=A_2(1,0,0)\cdot A_3(1,0,0)\cdot d_3\\
B\ =&A_3(0,1,0)\cdot A_2(1,0,0)\cdot A_3(1,0,0)\cdot A_2(1,0,0)^{-1}\cdot A_3(0,1,0)^{-1}
\end{align*}
Then we eliminate $A_2(1,0,1)$, $A_3(1,1,0)$, $A_3(1,0,0)$, $d_2$ and $d_3$. Thus $B_3\Theta_4$ has a presentation with six generators and one relator as follows:
$$\langle A_2(1,0,0),A_3(0,1,0),B, d_1,D_2,D_3\ |\ [d_1,A_2(1,0,0)][D_3,A_3(0,1,0)][B,D_2]\rangle.$$
This is a fundamental group of an orientable closed surface of genus 3.

In fact, $UD_3\Theta_4$ is an orientable closed surface of genus $3$. So we can see that its sixfold cover $D_3\Theta_4$ is an orientable surface of genus $13$ by considering Euler characteristics.
\qed

The rewriting algorithm seems exponential in the size of graphs. Fortunately, we need not precisely compute the boundary word of a critical 2-cell since we are only interested in the number of generators and how to eliminate generators via Tietze transformations. We use the technique developed in \S\ref{ss31:PMatrix} for $UD_n\Gamma$ and the parallel technique developed in \S\ref{ss33:H1PB2} for $D_2\Gamma$. Recall that the orders on critical 1-cells an on critical 2-cells was important ingredients for the techniques. Using the presentation matrices for $H_1(B_n\Gamma)$ or $H_1(D_2\Gamma,(D_2\Gamma)^0)$ over bases of 2-cells and 1-cells ordered reversely, critical 1-cells were classified into pivotal, free, and separating 1-cells.
\begin{lem}\label{lem:Ng-nonpivotal}{\rm [Elimination of Pivotal 1-Cells]}
Assume that $\Gamma$ has a maximal tree and an order according to Lemma~\ref{lem:treeandorder}. Then $B_n\Gamma$ and $\pi_1(D_2\Gamma/\sim)$ are generated by free and separating 1-cells.
\end{lem}
\begin{proof}
There is no difference between $B_n\Gamma$ and $\pi_1(D_2\Gamma/\sim)$ in our argument. We discuss only $B_n\Gamma$. The proof for $\pi_1(D_2\Gamma/\sim)$ is exactly the same except the fact that permutations are used as subscripts to express critical cells.

Consider pairs $(c_2,c_1)$ of a pivotal 2-cell $c_2$ that produces a pivotal 1-cell $c_1$. Then either $s(c_1)\ge2$ or $c_1$ is of the form $d(\vec\delta_\ell)$. In \S\ref{ss31:PMatrix}, a pivotal 1-cell $c_1$ is the largest summand of $\widetilde\partial(c_2)$ and so is not a summand of $\widetilde\partial(c'_2)$ for a pivotal 2-cell $c'_2<c_2$. We want to obtain the corresponding noncommutative version. 

We need to slightly modify the order on critical 1-cells when only pivotal 1-cells are compared. For an edge $e$ in $\Gamma$, set $t(e)=0$ if $e$ is in the maximal tree $T$ and $t(e)=1$ otherwise. Declare $e>e'$ if $(\tau(e),t(e),\iota(e))>(\tau(e'),t(e'),\iota(e'))$. The set of pivotal 1-cells are linearly ordered by the triple $(s(c),e,\vec a)$ under the modified order on edges. We modified the order on the set of pivotal 2-cells accordingly, that is, $c_2>c'_2$ for pivotal 2-cells $c_2$ and $c'_2$ if $c_1>c'_1$ when $(c_2,c_1)$ and $(c'_2,c'_1)$ are pairs of a pivotal 2-cell and the corresponding pivotal 1-cell.

Let $\partial_w(c)$ denote the boundary word of a given pivotal 2-cell $c$. We claim the following:
\begin{itemize}
\item[(a)] $c_1$ appears in the word $\tilde r\circ\partial_w(c_2)$ exactly once (as a letter or the inverse of a letter).
\item[(b)] Under the order defined above, $c_1$ is the largest pivotal 1-cells appeared in $\tilde r\circ\partial_w(c_2)$
\end{itemize}

Note that (b) implies that $c_1$ does not appear in $\tilde r\circ\partial_w(c'_2)$ for any pivotal 2-cell $c'_2<c_2$.
Then via Tietze transformations, we can inductively eliminate pivotal 1-cells from the set of generators given by critical 1-cells in $UD_n\Gamma$. Thus $B_n\Gamma$ is generated by free and separating 1-cells. Note that it is easy to perform inductive eliminations of pivotal 1-cells in decreasing order because no substitution is required.

To show our claim, we have to analyze each term in $\partial_w(c_2)$ due to the lack of luxury such as Lemmas~\ref{lem:boundary-1} and \ref{lem:boundary-2}. First consider the image of a redundant 1-cell under $\tilde r$. Let $c=\{e,v_1,\cdots,v_{n-1}\}$ be an 1-cell. Repeated applications of Lemma~\ref{lem:rewriting} imply that for any critical 1-cell $c'=\{e',v'_1,\cdots,v'_{n-1}\}$ appearing in $\tilde r(c)$, the vertex $\tau(e')$ is of valency $\ge 3$ in $\Gamma$ and of the form $v_i\wedge v_j$ or $v_i\wedge\tau(e)$ or $v_i\wedge \iota(e)$ and moreover $s(c')$ is less than or equal to the number of vertices that do not lie on the 0-th branch of $\tau(e')$ among $v_1,\ldots,v_{n-1}$, $\iota(e)$, and $\tau(e)$. By Lemma~\ref{lem:but}, if $c'$ contains a deleted edge, then $c$ also contains a deleted edge and $c'$ appears only once in $\tilde r(c)$. And the terminal vertex of the edge in $c_1$ is the larger one between terminal vertices of two edges in $c_2$.

Since $c_2$ is pivotal, the edge in $c_2$ with the smaller terminal vertex blocks no vertices (see the proof of Lemma~\ref{lem:rows}), there are two possibilities for $\partial_w(c_2)$ as follows:
\begin{align*}
\partial_w(A_k(\vec a)\cup d')&=A_k(\vec a)\cup\iota(d')\cdot d'\cup\dot A(\vec a)\cdot\{A_k(\vec a)\cup\tau(d')\}^{-1}\cdot\{d'\cup A(\vec a+\vec\delta_k)\}^{-1}\\
\partial_w(d(\vec a)\cup d')&=d(\vec a)\cup\iota(d')\cdot d'\cup\dot A(\vec a)\cdot\{d(\vec a)\cup\tau(d')\}^{-1}\cdot\{d'\cup A(\vec a)\cup\iota(d)\}^{-1}
\end{align*}
where $A=\tau(d)$ and $k=g(A,\iota(d))$.
Let $c$ be a redundant 1-cell in the above boundary words, $e$ be the edge in $c$, and $v$ be a vertex of valency$\ge 3$ in $\Gamma$ other than the base vertex 0. Then the number of vertices that do not lie on the 0-th branch of $v$ among vertices in $c$ and end vertices of $e$ is less than or equal to $s(c_2)+1$.

In the case of $c_2=A_k(\vec a)\cup d'$, the corresponding pivotal 1-cell $c_1$ is $A_k(\vec a+\vec\delta_{g(A,\iota(d'))})$ by Lemma~\ref{lem:but}. Repeated applications of  Lemma~\ref{lem:rewriting} implies $\tilde r(A_k(\vec a)\cup\iota(d'))=c_1$. Consider other three redundant 1-cells. Since $s(c_1)=s(c_2)+1$ and $\tau(d')\wedge A=\tau(d')<A$, the words $\tilde r(d'\cup\dot A(\vec a))$ and $\tilde r(A_k(\vec a)\cup\tau(d'))$ contain no pivotal 1-cells $>c_1$. Since $d'\cup A(\vec a+\vec\delta_k)$ contains a vertex $<\iota(A_k)$, $a_i\ge1$ for some $i<k$, that is, $p(\vec a)<k$. If $c=A_{p(\vec a)}(\vec a+\vec\delta_k-\vec\delta_{p(\vec a)}+\vec\delta_{g(A,\iota(d')})$ is pivotal then $c<c_1$. This imply that $\tilde r(d'\cup A(\vec a))$ contains no pivotal 1-cells $>c_1$. We are done.

In the case of $c_2=d(\vec a)\cup d'$, $c_1=d(\vec a+\vec\delta_{g(\tau(d),\iota(d'))})$ by Lemma~\ref{lem:but} and so (a) is true since $c_1$ contains a deleted edge. Both words $\tilde r(d'\cup\dot A(\vec a))$ and $\tilde r(d(\vec a)\cup\tau(d'))$ contain no pivotal 1-cell $>c_1$ by the same argument as for $A_k(\vec a)\cup d'$.
For any vertex $v>\tau(d)$ of valency $\ge 3$ in $\Gamma$, there is only one vertex in $d(\vec a)\cup\iota(d')$ that do not lie on the 0-th branch of $v$. If $c'$ is a critical 1-cell in $\tilde r(d(\vec a)\cup\iota(d'))$ such that the terminal vertex of the edge in $c$ is larger than $\tau(d)$, then $s(c')=1$ and so it is a critical 1-cell of the form $B_{k'}(\vec\delta_{\ell'})$ and so $c'$ is not pivotal. Similarly $\tilde r(d'\cup A(\vec a)\cup\iota(d))$ contains no pivotal 1-cell $>c_1$. This completes the proof.
\end{proof}

\begin{lem}\label{lem:Ngenerators}{\rm [Fewest Generators]}
$B_n\Gamma$ ($\pi_1(D_2\Gamma/\sim)$, respectively) has a presentation over $m$ generators
for the rank $m$ of $H_1(UD_n\Gamma;\mathbb Z_2)$ ($H_1(D_2\Gamma, (D_2\Gamma)^0)$, respectively).
\end{lem}

\begin{proof}
We discuss only $UD_n\Gamma$. The proof for $D_2\Gamma/\sim$ is essentially the same.
Lemma~\ref{lem:Ng-nonpivotal} gives a presentation for $B_n\Gamma$ over free and separating 1-cells. Since each free 1-cell contributes to the rank of the first homology, we leave them. To consider separating 1-cell, let $d\cup d'$ be a pivotal 2-cell whose boundary word contains a pivotal 1-cell of the form $d(\vec\delta_m)$ with $g(\tau(d),\iota(d))=m$  and $d\cup d''$ be another critical 2-cell whose boundary word also contains $d(\vec\delta_m)$ as the largest critical 1-cell. Recall that the row corresponds to the difference of the two critical 2-cells $d\cup d'$ and $d\cup d''$ consists of two nonnegative entries $\pm 1$ that correspond to separating 1-cells $\mbox{$\wedge$}(d,d')$ and $\mbox{$\wedge$}(d,d'')$ in a presentation matrix for $H_1(B_n\Gamma)$. For the group presentation, this can be done by the Tietze transformation that eliminates the pivotal 1-cell $d(\vec\delta_m)$. After the elimination, a new relator $w(d,d',d'')$ is obtained by a substitution from the two boundary words. And $w(d,d',d'')$ contains separating 1-cells $\mbox{$\wedge$}(d,d')$ and $\mbox{$\wedge$}(d,d'')$. Furthermore, since
$$\tilde r\partial_w(d\cup d')=\tilde r(d'\cup\iota(d)\cdot d\cup\tau(d')\cdot\{d'\cup\tau(d)\}^{-1}\cdot\{d\cup\iota(d')\}^{-1}),$$
the terminal vertex of the edge in a critical 1-cell other than $\mbox{$\wedge$}(d,d')$ in $\tilde r\partial_w(d\cup d')$ is $\le\tau(d)$ by a similar argument as in the proof of Lemma~\ref{lem:Ng-nonpivotal}. Thus $\mbox{$\wedge$}(d,d')$ and $\mbox{$\wedge$}(d,d'')$ are greater than any other critical 1-cell of the form $A_k(\vec\delta_\ell)$ and so they are greater than other separating 1-cells in $w(d,d',d')$. Moreover, the exponent sum of each generator other than $\mbox{$\wedge$}(d,d')$ and $\mbox{$\wedge$}(d,d'')$ in $w(d,d',d'')$ is zero.

By our setup, if two separating 1-cells $\mbox{$\wedge$}(d,d')$ and $\mbox{$\wedge$}(d,d'')$ are homologous, then (i) there is a relator $w(d,d',d'')$ that contains each of them once, (ii) any other separating 1-cells in $w(d,d',d'')$ is less than them, and (iii) the exponent sum of each generator other than them is 0. Clearly the converse is also true.
Consider a labeled graph $G$ in which separating 1-cells are vertices and there are edges labeled by $w(d,d',d'')$ between two vertices $\mbox{$\wedge$}(d,d')$ and $\mbox{$\wedge$}(d,d'')$. The number of connected components of $G$ is exactly the number of homology classes of separating 1-cells. We are done if each connected component becomes a graph with one vertex (and loops) via inductive edge contractions.

Starting from the vertex $\mbox{$\wedge$}(d,d')$ that is the smallest separating 1-cell in $G$, we eliminate $\mbox{$\wedge$}(d,d')$ via a Tietze transformation as follows: Choose the smallest vertex $\mbox{$\wedge$}(d,d'')$ among all vertices adjacent to $\mbox{$\wedge$}(d,d')$, contract an edge $w(d,d',d'')$ (choose any edge if it is a multiedge) by throwing away the vertex $\mbox{$\wedge$}(d,d')$, solve $w(d,d',d'')$ for $\mbox{$\wedge$}(d,d')$, and assign new labels obtained by substitutions to all other edges that used to be incident to $\mbox{$\wedge$}(d,d')$. Then all edge labels except for loops again have the properties (i), (ii), and (iii) above. In particular, (i) follows from (ii) since $\mbox{$\wedge$}(d,d')$ was the smallest vertex in $G$. To iterate the process, let $G$ be the modified graph. Go to the smallest vertex in $G$ and start again. Since separating 1-cells are linearly ordered, this iteration clearly turn each connected component of $G$ into a graph with only the largest vertex together with loops. Note that the exponent sum of the vertex in the label of each loop is either 0 or $\pm 2$ due to the property of original labels.
\end{proof}

\begin{thm}\label{thm:crg}
If $\Gamma$ is a finite connected planar graph (a finite connected graph, respectively), then $B_n\Gamma$ ($P_2\Gamma$, respectively) is a commutator-related group.
\end{thm}
\begin{proof}
Note that if $\Gamma$ is planar, $H_1(B_n\Gamma)\cong \mathbb Z^m$ for the rank $m$ of $H_1(B_n\Gamma;\mathbb Z_2)$. Now the theorem is immediate from Proposition~\ref{pro:CRG} and Lemma~\ref{lem:Ngenerators}. In fact, a careful analysis of the proof of Lemma~\ref{lem:Ngenerators} can also prove the theorem without Proposition~\ref{pro:CRG}.
\end{proof}

\subsection{Presentations of $B_2\Gamma$ and $P_2\Gamma$}\label{ss42:PGB2G}

A group $G$ is {\em simple-commutator-related} if $G$ has a presentation whose
relators are commutators. Clearly a right-angled Artin group is simple-commutator-related and a simple-commutator-related group is commutator-related.

In~\cite{FS2}, Farley and Sabalka conjectured that $B_2\Gamma$ is simple-commutator-related for a planar graph $\Gamma$ and relators are commutators of two words that represent disjoint circuits on the planar graph.  In a private correspondence, Abrams conjectured that $P_2\Gamma$ is simple-commutator-related for a planar graph $\Gamma$. There has been some doubt on these conjectures (for example, see \cite{Kur}). By combining our result with the result by Barnett and Farber in~\cite{BF}, we will prove that for a planar graph $\Gamma$, both $B_2\Gamma$ and $P_2\Gamma$ is simple-commutator-related and relators are commutators of disjoint circuits on $\Gamma$. So these conjectures are true.

First we need the following lemma proved by Barnett and Farber in \cite{BF}.

\begin{lem}{\rm (Barnett and Farber \cite{BF})}\label{lem:BF}
Let $\Gamma\subset\mathbb R^2$ be a planar graph and $U_0$, $U_1,\cdots,U_r$ be the connected components of $\mathbb R^2-\Gamma$ with $U_0$ denoting the unbounded component and $S(i)$ denote $\partial\overline U_i\subset\Gamma$. Then homology classes $[S(i)\times S(j)]$ with $S(i)\cap S(j)=\emptyset$ freely generate $H_2(D_2\Gamma)$.
\end{lem}

To show the lemma, Barnett and Farber noticed the exact sequence
$$0\to H_2(D_2\Gamma)\to H_2(\Gamma\times\Gamma)\to H_2(\Gamma\times\Gamma,D_2\Gamma).$$
There is a corresponding exact sequence
$$0\to H_2(UD_2\Gamma)\to H_2(\Gamma\times\Gamma/\sim)\to H_2(\Gamma\times\Gamma/\sim,UD_2\Gamma)$$
where $\Gamma\times\Gamma/\sim$ is the quotient space obtained by $(x,y)\sim(y,x)$.
Under this equivalence relation, the homology classes $[S(i)\times S(j)]$ and $[S(j)\times S(i)]$ are identified and give a homology class $\{S(i)\times S(j)\}$ in $H_2(UD_2\Gamma)$.
Lemma~\ref{lem:BF} implies that homology classes $\{S(i)\times S(j)\}$ with $S(i)\cap S(j)=\emptyset$ freely generate $H_2(UD_2\Gamma)$. In fact, $\beta_2(D_2\Gamma)=2\beta_2(UD_2\Gamma)$ by the formulae for second Betti numbers in Sec~\ref{ss32:H1} and Sec~\ref{ss33:H1PB2}. We are now ready for the last theorem.

\begin{thm}\label{thm:presentationforn2}
For a planar graph $\Gamma$, both $B_2\Gamma$ and $P_2\Gamma$ are simple-commutator-related and relators are commutators of two disjoint circuits on $\Gamma$. In fact, there is a presentation of $B_2\Gamma$ ($P_2\Gamma$, respectively) over $\beta_1$ generators such that it has $\beta_2$ relators that are all commutators where $\beta_1$ and $\beta_2$ are the first and second Betti numbers of
$B_2\Gamma$ ($P_2\Gamma$, respectively).
\end{thm}
\begin{proof}
There is no difference between $B_2\Gamma$ and $P_2\Gamma$ in our argument. We discuss only $P_2\Gamma$. Each torus $S(i)\times S(j)$ in Lemma~\ref{lem:BF} is embedded in the discrete configuration space $D_2\Gamma$. Since each  circuit in $\Gamma$ contains at least a deleted edge, so does each $S(i)$. So the embedded torus $S(i)\times S(j)$ remains as an immersed torus $T_{ij}$ in a Morse complex $M_2\Gamma$ since deleted edges gives critical 1-cells.
The immersed tori may intersect each other but are never identified since they generate $H_2(M_2\Gamma)$.

Each Tietze transformation performed in the proofs of Lemma~\ref{lem:Ng-nonpivotal} and Lemma~\ref{lem:Ngenerators} is an elimination of a pair of a generator and a relation. In the cell complex $M_2\Gamma$, this corresponds to collapsing of a canceling pair of a 1-cell and a 2-cell. Let $M'_2\Gamma$ denote the cell complex obtained from $M_2(\Gamma)$ by collapsing all canceling pairs corresponding to Tietze transformations performed in the proofs of the two lemmas. Each immersed torus $T_{ij}$ in $M_2\Gamma$ remains as an immersed  torus $T'_{ij}$ in $M'_2\Gamma$ after collapsing even though it may become complicated.

By Lemma~\ref{lem:Ngenerators}, the cell complex $M'_2\Gamma$ has two 0-cells and $(m-1)$ 1-cells for the rank $m$ of $H_1(D_2\Gamma/\sim)$ since the identification space $D_2\Gamma/~$ can also be obtained by adding a 1-cell between two base vertices which remains in $M'_2\Gamma$. By consideration of Euler characteristics, the number of 2-cells in $M'_2\Gamma$ is equal to the rank of $H_2(D_2\Gamma)$ and so equal to the number of (ordered) tori $S(i)\times S(j)$. Therefore each 2-cell must form an immersed torus $T'_{ij}$ and produces a relator that must be a commutator.
\end{proof}

Note that Theorem~\ref{thm:presentationforn2} is false for braid index $n\ge3$. For example, $B_3\Theta_4$ and $P_3\Theta_4$ are surface groups (see Example~\ref{ex:T4n3}). One can show that $B_3\Theta_4/(B_3\Theta_4)_{3}$ is isomorphic to $B_n\Theta_4/(B_n\Theta_4)_{3}$ for $n\ge4$ where $G_3$ denotes the third lower central subgroup of a group $G$. Thus $B_n\Theta_4$ is not simple-commutator-related for $n\ge3$. If $\Gamma$ contains a subgraph $\Theta_4$, $B_n\Gamma$ ($P_3\Gamma$, respectively) has a subgroup that is not simple-commutator-related since there is a local isometric embedding from $UD_n\Theta_4$ ($P_3\Theta_4$, respectively) to $UD_n\Gamma$ ($P_3\Gamma$, respectively). Thus it seems reasonable to propose the following conjecture:

\begin{con}
For a planar graph $\Gamma$, $B_n\Gamma$ for $n\ge3$ and $P_3\Gamma$ are simple-commutator-related if and only if $\Gamma$ does not contain a subgraph $\Theta_4$.
\end{con}

For instance, Farley and Sabalka showed in \cite{FS} that every tree braid group is simple-commutator-related.

\bibliographystyle{amsplain}

\end{document}